\documentclass[10pt,a4paper]{amsart}
\pdfoutput=1
\usepackage[subsec,skipmbb]{rune}
\usepackage{runediag}

\title[The Universal Property of Bispans]{On Distributivity in Higher Algebra I: \\ The Universal Property of Bispans}

\author{Elden Elmanto}
\address{Department of Mathematics, Harvard University, Cambridge,
  USA}
\urladdr{http://eldenelmanto.com}
  
\author{Rune Haugseng}
\address{Department of Mathematical Sciences, NTNU, Trondheim, Norway}
\urladdr{http://folk.ntnu.no/runegha}
\date{\today}

\newcommand{\igpds}{$\infty$-groupoids}
\theoremstyle{definition}
\newtheorem{variant}[thm]{Variant}

\newcommand{\Bispan}{\txt{Bispan}}
\newcommand{\BISPAN}{\txt{BISPAN}}
\newcommand{\CAT}{\txt{CAT}}
\newcommand{\FUN}{\txt{FUN}}
\newcommand{\MAP}{\txt{MAP}}
\newcommand{\ADJ}{\txt{ADJ}}
\newcommand{\CATI}{\CAT_{\infty}}
\newcommand{\CatIT}{\Cat_{(\infty,2)}}

\newcommand{\Sch}{\mathrm{Sch}}
\newcommand{\AlgSpc}{\mathrm{AlgSpc}}
\newcommand{\Perf}{\mathrm{Perf}}
\newcommand{\QCoh}{\mathrm{QCoh}}
\newcommand{\SpDM}{\mathrm{SpDM}}
\newcommand{\FP}{\mathcal{FP}}
\newcommand{\all}{\txt{all}}
\newcommand{\fet}{\txt{f\'et}}
\newcommand{\fold}{\txt{fold}}

\newcommand{\sm}{\mathrm{sm}}
\newcommand{\smqp}{\mathrm{smqp}}
\newcommand{\proj}{\mathrm{proj}}
\newcommand{\prop}{\mathrm{prop}}

\newcommand{\ostar}{\circledast}
\newcommand{\fin}{\mathrm{fin}}

\newcommand{\poly}{\mathrm{poly}}
\newcommand{\stab}{\mathrm{stab}}

\newcommand{\igpd}{$\infty$-groupoid}

\usepackage{extarrows}
\newcommand{\longequal}{\xlongequal{\,\,}}

\newcommand{\twop}{2\txt{-op}}
\newcommand{\onop}{1\txt{-op}}

\newcommand{\Flpadj}{F\txt{-lpreadj}}
\newcommand{\Fradj}{F\dradj}
\newcommand{\Fladj}{F\dladj}

\newcommand{\Ldist}{L\txt{-dist}}

\newcommand{\ddist}{\txt{-dist}}

\newcommand{\Fcart}{F\txt{-cart}}

\newcommand{\fmnd}{\mathfrak{mnd}}

\newcommand{\ladj}{\txt{ladj}}
\newcommand{\radj}{\txt{radj}}
\newcommand{\dradj}{\txt{-radj}}
\newcommand{\dladj}{\txt{-ladj}}
\newcommand{\drcoadj}{\txt{-rcoadj}}
\newcommand{\Frcoadj}{F\drcoadj}
\newcommand{\Pair}{\txt{Pair}}
\newcommand{\Trip}{\txt{Trip}}

\newcommand{\SH}{\mathrm{SH}}
\newcommand{\HH}{\mathrm{H}}

\newcommand{\flf}{\txt{flf}}
\newcommand{\qp}{\txt{qp}}
\newcommand{\xF}{\mathbb{F}}
\newcommand{\Yo}{\mathsf{y}}
\newcommand{\Vect}{\txt{Vect}}
\newcommand{\Rep}{\txt{Rep}}
\newcommand{\dRep}{\Rep^{\infty}}
\newcommand{\angled}[1]{\langle #1 \rangle}
\DeclareMathOperator{\Spec}{Spec}
\DeclareMathOperator{\CMon}{CMon}
\DeclareMathOperator{\PolyFun}{PolyFun}
\DeclareMathOperator{\AnFun}{AnFun}
\DeclareMathOperator{\Sq}{Sq}
\crefformat{equation}{(#2#1#3)}

\newcommand{\drpullback}{\arrow[phantom]{dr}[very near
  start,description]{\lrcorner}}
\newcommand{\ddrpullback}{\arrow[phantom]{ddr}[very near start,description]{\lrcorner}}
\newcommand{\dlpullback}{\arrow[phantom]{dl}[very near start,description]{\llcorner}}

\newcommand{\urpullback}{\arrow[phantom]{ur}[very near start,description]{\urcorner}}

\begin{document}

\begin{abstract}
  Structures where we have both a contravariant (pullback) and a
  covariant (pushforward) functoriality that satisfy base change can
  be encoded by functors out of ($\infty$-)categories of spans (or
  correspondences). In this paper we study the more complicated setup
  where we have two pushforwards (an ``additive'' and a
  ``multiplicative'' one), satisfying a distributivity relation. Such
  structures can be described in terms of bispans (or polynomial
  diagrams). We show that there exist $(\infty,2)$-categories of
  bispans, characterized by a universal property: they corepresent
  functors out of $\infty$-categories of spans where the pullbacks
  have left adjoints and certain canonical 2-morphisms (encoding
  base change and distributivity) are invertible. This gives a
  universal way to obtain functors from bispans, which amounts to
  upgrading ``monoid-like'' structures to ``ring-like'' ones. For
  example, symmetric
  monoidal $\infty$-categories can be described as product-preserving
  functors from spans of finite sets, and if the tensor product is
  compatible with finite coproducts our universal property gives
  the canonical semiring structure using the coproduct and tensor
  product. More interestingly, we encode the additive and
  multiplicative transfers on equivariant spectra as a functor from
  bispans in finite $G$-sets, extend the norms for finite \'etale
  maps in motivic spectra to a functor from certain bispans in
  schemes, and make $\Perf(X)$ for $X$ a spectral Deligne--Mumford stack
  a functor of bispans using a multiplicative pushforward for finite
  \'etale maps in addition to the usual pullback and pushforward
  maps. Combining this with the polynomial functoriality of $K$-theory
  constructed by Barwick, Glasman, Mathew, and Nikolaus, we obtain
  norms on algebraic $K$-theory spectra.
\end{abstract}

\maketitle
\tableofcontents

\section{Introduction}
This paper is the first part of a project aimed at better understanding
certain sophisticated ring-like structures that occur in ``homotopical
mathematics''.
By this we mean not just the theory of $E_{\infty}$-rings, where
additions and multiplications are indexed over finite sets, but also
more exotic structures occuring in equivariant and motivic homotopy
theory where operations can be indexed over finite $G$-sets and finite
\'etale morphisms, respectively. Such structures are also relevant to
derived algebraic geometry and algebraic $K$-theory.

In the present paper we construct equivariant and motivic versions of
the canonical semiring structure on a symmetric monoidal \icat{} whose
tensor product commutes with finite coproducts.

In the $G$-equivariant case this structure encodes the compatibility of
additive and multiplicative transfers (or norms) along maps of finite
$G$-sets. In the case of genuine $G$-spectra such multiplicative
transfers were defined by Hill, Hopkins, and
Ravenel~\cite{HHRKervaire} (extending a construction on the level of
cohomology groups due to Greenlees--May
\cite{GreenleesMayMU,BohmannNorm}) and played a key role in their
solution of the Kervaire invariant one problem; more recently, they have
been considered as the defining structure of an equivariant symmetric
monoidal \icat{} in ongoing work of Barwick, Dotto, Glasman, Nardin,
and Shah~\cite{BDGNS1}.

In the motivic version, we have multiplicative transfers along finite
\'etale morphisms and additive transfers along smooth morphisms of
schemes. Such multiplicative transfers were constructed for motivic
spectra (and in a number of related examples) by Bachmann and
Hoyois~\cite{norms}. These generalize, among other constructions,
Fulton--Macpherson's norms on Chow groups \cite{fulton-mac} and
Joukhovitski's norms on $K_0$ \cite{kzero}.

We will also show that the \icats{} $\Perf(X)$ of perfect
quasicoherent sheaves on a spectral Deligne--Mumford stack $X$ have a
similar structure given by a multiplicative pushforward for finite
\'etale maps in addition to the usual pushforward and pullback
functors. In all these cases we will obtain the canonical ``semiring''
structures using a universal property of $(\infty,2)$-categories of
\emph{bispans}, which is the main result of this paper.

\subsection{Spans and commutative monoids}
Before we explain what we mean by bispans, it is helpful to first
recall the relation between commutative monoids and spans: if $\xF$
denotes the category of finite sets, then we can define a
(2,1)-category $\Span(\xF)$ where
\begin{itemize}
\item objects are finite sets,
\item morphisms from $I$ to $J$ are \emph{spans} (or correspondences)
  \[
    \begin{tikzcd}[row sep=small]
      {} & S \arrow{dl} \arrow{dr} \\
      I & & J,
    \end{tikzcd}
  \]
\item composition is given by pullback: the composite
  \[ \left(\begin{tikzcd}[row sep=small]
      {} & T \arrow{dl} \arrow{dr} \\
      J & & K,
    \end{tikzcd} \right) \circ    \left(\begin{tikzcd}[row sep=small]
      {} & S \arrow{dl} \arrow{dr} \\
      I & & J,
    \end{tikzcd} \right)\]
is the outer span in the diagram
\[
  \begin{tikzcd}[row sep=small]
    {} & {} & S \times_{K} T \arrow{dl} \arrow{dr} \\
    & S \arrow{dl} \arrow{dr} & & T \arrow{dl} \arrow{dr} \\
    I & & J & & K,
  \end{tikzcd}
  \]
\item 2-morphisms are isomorphisms of spans.
\end{itemize}
If $M$ is a commutative monoid in $\Set$, we can use the monoid structure to define a functor
\[ \Span(\xF) \to \Set \]
which takes $I \in \xF$ to $M^{I} := \prod_{i \in I} M$ and a span $I
\xfrom{f} S \xto{g} J$ to the composite $g_{\otimes}f^{*}$ where
$f^{*}\colon M^{I} \to M^{S}$ is given by composition with $f$ (so
$f^{*}\phi(s) = \phi(fs)$) and $g_{\otimes}$ is defined using the
product on $M$ by
\[ g_{\otimes}(\phi)(j) = \prod_{s \in g^{-1}(j)} \phi(s).\]
This is compatible with composition of spans, since a pullback square
gives a canonical isomorphism of fibres and we have $(gg')_{\otimes}
= g_{\otimes}g'_{\otimes}$ as the multiplication is associative.

It can be shown that if $\mathbf{C}$ is any category with finite
products, every functor $\Phi \colon \Span(\xF) \to \mathbf{C}$ such
that $\Phi(I) \cong \Phi(*)^{\times |I|}$ via the canonical maps
arises in this way from a commutative monoid in $\mathbf{C}$. More
precisely, we can identify commutative monoids in $\mathbf{C}$ with
product-preserving functors $\Span(\xF) \to \mathbf{C}$. (In other
words, the homotopy category of $\Span(\xF)$ is precisely the
\emph{Lawvere theory} for commutative monoids.) This is also true
homotopically:
\begin{thm}
  Let $\mathcal{C}$ be an \icat{} with finite products. There is a
  natural equivalence of \icats{} between commutative monoids in
  $\mathcal{C}$ and product-preserving functors $\Span(\xF) \to \mathcal{C}$.
\end{thm}
The earliest proof of this seems to be the in thesis of Cranch
\cite{CranchThesis,CranchSpan}; other proofs (as special cases of
different generalizations) are due to Bachmann--Hoyois~\cite[Appendix
C]{norms} and Glasman~\cite[Appendix A]{GlasmanStrat}. In addition, it
appears in
Harpaz~\cite[Section 5.2]{HarpazAmbi}  as the bottom case of
his theory of $m$-commutative monoids.

\subsection{Bispans and commutative semirings}
We can ask for a similar description for commutative
semirings. In this case, we have two operations --- addition and multiplication
--- so we want a (2,1)-category $\Bispan(\xF)$ whose objects
are again finite sets, with a morphism from $I$ to $J$ given by a
\emph{bispan} (or polynomial diagram)
\begin{equation}\label{eq:bispan1}
  \begin{tikzcd}[row sep=small]
    {} & X \arrow{r}{f} \arrow{dl}[swap]{p} & Y \arrow{dr}{q} \\
    I & & & J.
  \end{tikzcd}
\end{equation}
If $R$ is a commutative semiring in $\Set$, we want a functor
\[
\Bispan(\xF) \to \Set
\] that takes a set $I$ to $R^{I}$ and the bispan
\cref{eq:bispan1} to $q_{\oplus}f_{\otimes}p^{*}$ where
\begin{itemize}
\item $p^{*} \colon R^{I}\to R^{X}$ is defined by composing with $p$,
  \[ p^{*}(\phi)(x) = \phi(px),\]
\item $f_{\otimes} \colon R^{X} \to R^{Y}$ is defined by multiplying
  in $R$ fibrewise,
  \[ f_{\otimes}(\phi)(y) = \prod_{x \in f^{-1}(y)} \phi(x),\]
\item $q_{\oplus}$ is defined by adding in $R$ fibrewise,
  \[ q_{\oplus}(\phi)(j) = \sum_{y \in q^{-1}(j)} \phi(y).\]
\end{itemize}
The question is then whether there is a way to define composition of
bispans so that this gives a functor. Given a pullback square
\begin{equation}
  \label{eq:setpb}
  \begin{tikzcd}
  I' \arrow{r}{g} \arrow{d}[swap]{i} & J' \arrow{d}{j} \\
  I \arrow{r}{f} & J
\end{tikzcd}
  \end{equation}
in $\xF$, we have identities $g_{\otimes}i^{*} = j^{*}f_{\otimes}$ and
$g_{\oplus}i^{*} = j^{*}f_{\oplus}$ as before, but now we also need to
deal with compositions of the form $v_{\otimes}u_{\oplus}$ for $u
\colon I \to J$ and $v \colon J \to K$. Using the distributivity of
addition over multiplication, for $\phi \colon I \to R$ and $k \in K$ we can write
\begin{equation}
  \label{eq:disteqn}
 v_{\otimes}u_{\oplus}(\phi)(k) = \prod_{j \in J_{k}} \sum_{i \in
    I_{j}} \phi(i) = \sum_{(i_{j}) \in \prod_{j \in J_{k}} I_{j}}
  \prod_{t \in J_{k}} \phi(i_{t}). 
\end{equation}
We can interpet this in terms of a \emph{distributivity diagram} in
$\xF$: if we let $h \colon X \to K$ be the family of sets $X_{k} = \prod_{j \in
J_{k}} I_{j}$ (so that $h = v_{*}u$ where $v_{*}$ is the right adjoint
to pullback along $v$), then the pullback $v^{*}X$ has a canonical map
to $I$ over $J$: on the fibre $(v^{*}X)_{j}$, which is the product $\prod_{j'
  \in J_{v(j)}} I_{j'}$, we take the projection to the factor
$I_{j}$. This gives a commutative diagram
\begin{equation}
  \label{eq:distdiagF}
  \begin{tikzcd}
    {} & v^{*}X \arrow{r}{\tilde{v}} \arrow{dd} \arrow{dl}[swap]{\epsilon}\arrow[phantom]{ddr}[very near
    start,description]{\lrcorner} &
    X \arrow{dd}{h} \\
    I \arrow{dr}{u} \\
     & J \arrow{r}{v} & K
  \end{tikzcd}
\end{equation}
where the square is cartesian, and we can rewrite the distributivity
relation \cref{eq:disteqn} as
\[ v_{\otimes}u_{\oplus} =
  h_{\oplus}\tilde{v}_{\otimes}\epsilon^{*}.\]
This means we will get a functor $\Bispan(\xF) \to \Set$ from the
commutative semiring $R$ if we define the composition of two bispans
\[ I \xfrom{s} E \xto{p} B \xto{t} J,\]
\[ J \xfrom{u} F \xto{q} C \xto{v} K,\]
as the outer bispan in the following diagram:
 \begin{equation}
   \label{eq:bispancomp}
 \begin{tikzcd}
   {} &   &   & G \arrow[bend left=32]{rrr}{w} \arrow[bend
   right=20]{dddlll}[above]{r} \arrow{rr}{p''} \arrow{dl}{\epsilon'} 
	\drpullback &                
	& X \arrow[phantom]{ddr}[very near start,description]{\lrcorner}\arrow{r}{\tilde{q}} \arrow{dl}{\epsilon}
   \arrow{dd}{q^{*}q_{*}\pi}& D \arrow{dd}{q_{*}\pi} \arrow[bend left=15]{dddr}{x} \\
   {} &   & Y \drpullback \arrow{rr}{p'} \arrow{dl}{u''} &   &  B \times_{J} F \arrow[phantom]{dd}[very near start]{\rotatebox{-45}{$\lrcorner$}} \arrow{dl}{u'} \arrow{dr}{\pi} \\
   {} & E\arrow[swap]{rr}{p} \arrow{dl}{s} &   & B \arrow[swap]{dr}{t} &     { }           &
   F \arrow{dl}{u}  \arrow[swap]{r}{q} & C \arrow[swap]{dr}{v} \\
   I  &   &   &   & J              &   &   & K.
 \end{tikzcd}
\end{equation}
Here we have used a distributivity diagram for $q$ and the pullback
$\pi$.

An explicit construction of a (2,1)-category $\Bispan(\xF)$ with this
composition is given in the
thesis of Cranch~\cite{CranchThesis}, where it is also proved that
this has products (given by the disjoint union of sets), so that we
can define commutative semirings in $\mathcal{S}$ as functors
$\Bispan(\mathbb{F}) \to \mathcal{S}$ that preserve finite
products. Alternatively, one can relate bispans of finite sets to
\emph{polynomial functors}, which gives an easier definition of
$\Bispan(\xF)$ (as the complicated composition law
\cref{eq:bispancomp} corresponds to the ordinary composition of such
functors); this approach was carried out by
Gambino--Kock~\cite{GambinoKock}, who also show that the homotopy
category of $\Bispan(\xF)$ is the Lawvere theory for commutative
semirings, so that commutative semirings in an ordinary category
$\mathbf{C}$ with finite products are equivalent to product-preserving
functors
\[ \Bispan(\xF) \to \mathbf{C}.\]
We expect that the homotopical analogue of this statement\footnote{Specifically, the definition
of commutative semirings in an \icat{} with finite products in terms
of $\Bispan(\xF)$ should be equivalent to that of
Gepner--Groth--Nikolaus~\cite{GGN}.} is also
true, but this has not yet been proved.

\subsection{The universal property of spans}
If $\mathcal{C}$ is a symmetric monoidal \icat{} such that
$\mathcal{C}$ has finite coproducts and the tensor product preserves
these in each variable, then we expect that $\mathcal{C}$ has a
canonical semiring structure in $\CatI$ with multiplication and
addition given by the tensor product and coproduct, respectively. This
follows\footnote{This semiring structure is also constructed in
  \cite{GGN} by a different method.}  from the universal property of
an $(\infty,2)$-category of bispans in $\xF$, which is a special case
of our main result. Before we state this, it is convenient to first
recall the simpler universal property of the \itcat{} $\SPAN(\xF)$ of
spans in $\xF$, which can be used to prove that an \icat{} with finite
coproducts has a canonical symmetric monoidal structure.

Here $\SPAN(\xF)$ has finite sets as objects, spans as morphisms, and
morphisms of spans as 2-morphisms, \ie{} 2-morphisms are commutative
diagrams
\[
  \begin{tikzcd}
    {} & X \arrow{dd} \arrow{dr} \arrow{dl} \\
    I & & J \\
    & Y. \arrow{ul} \arrow{ur}
  \end{tikzcd}
\]
Suppose $\mathcal{X}$ is an \itcat{}. A functor $\Phi \colon \xF^{\op} \to
\mathcal{X}$ is called \emph{left adjointable} if for every morphism
$f \colon I \to J$ in $\xF$, the morphism $f^{\ostar} := \Phi(f) \colon \Phi(J) \to
\Phi(I)$ in $\mathcal{X}$ has a left adjoint $f_{\oplus}$, and for every
pullback square \cref{eq:setpb} in $\xF$, the canonical
(Beck--Chevalley or mate) transformation
\[ g_{\oplus}i^{\ostar} \to j^{\ostar}f_{\oplus}\]
is an equivalence. 
\begin{thm}
  Restricting along the inclusion $\xF^{\op} \to \SPAN(\xF)$ (of the
  subcategory containing only the backwards maps and no non-trivial
  2-morphisms) gives a natural equivalence between functors
  $\SPAN(\xF) \to \mathcal{X}$ and left adjointable functors
  $\xF^{\op} \to \mathcal{X}$.
\end{thm}
This is a special case of a recent result of
Macpherson~\cite{MacphersonCorr}, which we review in more generality
below in \S\ref{sec:spans}. Another proof is sketched in the book of
Gaitsgory and Rozenblyum~\cite{GaitsgoryRozenblyum1} where this
universal property is used to encode a ``six-functor formalism'' for
various categories of quasicoherent sheaves on derived
schemes. Lastly, we note that the analogous result for ordinary
2-categories seems to have been first proved by
Hermida~\cite{Hermida}*{Theorem A.2}.

\subsection{The universal property of bispans}
We now want to consider a 2-category $\BISPAN(\xF)$ whose objects are finite
sets, with morphisms given by bispans and 2-morphisms by commutative
diagrams of the form
\begin{equation}
  \label{eq:bispanmor}
  \begin{tikzcd}
    {} & E \arrow{dd} \arrow{r}  \arrow{dl}
    \arrow[phantom]{ddr}[very near
    start,description]{\lrcorner} & B \arrow{dd} \arrow{dr} \\
    I & & & J \\
     & E' \arrow{r} \arrow{ul} & B' \arrow{ur}
  \end{tikzcd}  
\end{equation}
where the middle square is cartesian. If we look at the subcategory
where the morphisms are bispans whose rightmost leg is invertible and
with no non-trivial 2-morphisms, we get an inclusion
$\Span(\xF) \to \BISPAN(\xF)$. A special case of our main result gives
a universal property of $\BISPAN(\xF)$ in terms of this subcategory:
\begin{thm}\label{thm:BispanSet}
  Let $\mathcal{X}$ be an \itcat{}. Restricting along the inclusion
  $\Span(\xF) \to \BISPAN(\xF)$ gives an equivalence between functors
  $\BISPAN(\xF) \to \mathcal{X}$ and distributive functors $\Span(\xF)
  \to \mathcal{X}$.
\end{thm}
Here a functor $\Phi \colon \Span(\xF) \to \mathcal{X}$ is
\emph{distributive} if
\begin{itemize}
\item for every morphism $f \colon I \to J$ in $\xF$, the morphism
  $f^{\ostar} := \Phi(J \xfrom{f} I \longequal I)$ in $\mathcal{X}$
  has a left adjoint $f_{\oplus}$,
\item for every pullback square \cref{eq:setpb} in $\xF$, the
  Beck--Chevalley transformation $g_{\oplus}i^{\ostar} \to
  j^{\ostar}f_{\oplus}$ is an equivalence,
\item for every distributivity diagram \cref{eq:distdiagF}, the
  \emph{distributivity transformation}
  \[  h_{\oplus}\tilde{v}_{\otimes}\epsilon^{\ostar} \to
    v_{\otimes}u_{\oplus}, \]
  which is defined as a certain composite of units and counits, is an
  equivalence in $\mathcal{X}$.
\end{itemize}
Note that the only property of $\xF$ we have used in the definition of
distributive functors is the existence of distributivity
diagrams. These exist in any locally cartesian closed \icat{}, and
more generally we can consider triples $(\mathcal{C}, \mathcal{C}_{F},
\mathcal{C}_{L})$ consisting of an \icat{} $\mathcal{C}$ with a pair
of subcategories $\mathcal{C}_{F}$ and $\mathcal{C}_{L}$ such that
\begin{itemize}
\item pullbacks along morphisms in $\mathcal{C}_{F}$ and
  $\mathcal{C}_{L}$ exist in $\mathcal{C}$, and both subcategories are
  preserved under base change,
\item there exist suitable distributivity diagrams in $\mathcal{C}$ for any
  composable pair of morphisms $l \colon x \to y$ in
  $\mathcal{C}_{L}$, $f \colon y \to z$ in $\mathcal{C}_{F}$.
\end{itemize}
We can then generalize the notion of distributive functors above to
that of \emph{$L$-distributive functors} $\Span_{F}(\mathcal{C}) \to \mathcal{X}$, where
$\Span_{F}(\mathcal{C})$ is the \icat{} of spans in $\mathcal{C}$
whose forward legs are required to lie in $\mathcal{C}_{F}$. Our main
result in this paper is then the following generalization of
\cref{thm:BispanSet}:
\begin{thm}\label{thm:main}
  For $(\mathcal{C}, \mathcal{C}_{F}, \mathcal{C}_{L})$ as above,
  there exists an \itcat{} \[\BISPAN_{F,L}(\mathcal{C})\] such that
  \begin{itemize}
  \item objects are objects of $\mathcal{C}$,
  \item morphisms are bispans
    \[ x \xfrom{p} e \xto{f} b \xto{l} y \]
    where $f$ is in $\mathcal{C}_{F}$ and $l$ is in $\mathcal{C}_{L}$,
  \item 2-morphisms are diagrams of the form  \cref{eq:bispanmor},
  \item morphisms compose as in \cref{eq:bispancomp}.
  \end{itemize}
  The \itcat{} $\BISPAN_{F,L}(\mathcal{C})$ has the universal property
  that restricting to the subcategory $\Span_{F}(\mathcal{C})$ gives
  for any \itcat{} $\mathcal{X}$ an equivalence between functors
  $\BISPAN_{F,L}(\mathcal{C}) \to \mathcal{X}$ and $L$-distributive
  functors $\Span_{F}(\mathcal{C}) \to \mathcal{X}$.
\end{thm}
The analogue of this result for ordinary 2-categories (at least in the
case where $\mathcal{C} = \mathcal{C}_{F} = \mathcal{C}_{L}$) is due
to Walker~\cite{WalkerBispan}.

\subsection{Equivariant and algebro-geometric bispans}
We will now look briefly at some examples of \cref{thm:main} beyond
the case of finite sets, coming from equivariant and motivic homotopy
theory and derived algebraic geometry. These examples are
discussed in more detail in \S\ref{sec:ex}.

Let us first consider the equivariant setting, over a finite group
$G$. In all of our discussion above it is straightforward to replace
the category $\xF$ of finite sets with the category $\xF_{G}$ of
\emph{finite $G$-sets}. The analogue of a commutative monoid is then a
functor
\[ M \colon \Span(\xF_{G}) \to \Set \]
that preserves products. This is (essentially\footnote{Mackey functors
  are usually viewed as taking values in $\Ab$ rather than $\Set$;
  since the functor induces commutative monoid structures on its
  values, this amounts to asking for these monoid structures to be
  grouplike. The relation between Mackey functors in $\Ab$ and $\Set$
  is thus analogous to that between abelian groups and commutative monoids.})
the same thing as a
\emph{Mackey functor} \cite{DressMackey}, an algebraic structure where for a subgroup $H
\subseteq G$ we have restrictions $M^{G} \to M^{H}$ and transfers
$M^{H} \to M^{G}$ satisfying a base change property that can be
interpreted in terms of double cosets. Mackey functors play an
important role in group theory, and every genuine $G$-spectrum $E$ has
an underlying Mackey functor $\pi_{0}E$.

Similarly, the $G$-analogue of a commutative semiring is a
product-preserving functor $\Bispan(\xF_{G}) \to \Set$, which is
essentially\footnote{Again, the usual notion of a Tambara functor
  takes values in $\Ab$, which gives the equivariant version of a
  commutative ring rather than a semiring.} a \emph{Tambara functor} \cite{Tambara,StricklandTambara,blumberg-hill}.
This is a structure that has both an ``additive'' and a
``multiplicative'' transfer, satisfying a distributivity relation. If
$E$ is a genuine $G$-$E_{\infty}$-ring spectrum, then $\pi_{0}E$ has
the structure of a Tambara functor \cite{BrunTambara}.

If we replace the category of sets with the \icat{} of spaces, a
theorem of Nardin\footnote{Building on the description of $G$-spectra
  as ``spectral Mackey functors'', originally due to Guillou and May
  \cite{GuillouMaySpMack}.} \cite{NardinThesis}*{Corollary A.4.1}
shows that connective $G$-spectra are equivalently product-preserving
functors $\Span(\xF_{G}) \to \mathcal{S}$ that are grouplike,
generalizing the classical description of connective spectra as
grouplike commutative monoids in $\mathcal{S}$.\footnote{This can be
  seen as an \icatl{} version of more classical descriptions of
  equivariant infinite loop spaces,
cf.~\cite{Shimakawa,Ostermayr,MayMerlingOsorno,GMMOinfloops}.} 

The analogue for ring spectra is also expected to hold: connective
genuine $G$-$E_{\infty}$-ring spectra should be equivalent to
product-preserving functors $\Bispan(\xF_{G}) \to \mathcal{S}$.

Now we turn to the ``categorified'' versions of these structures: For
$H$ a subgroup of $G$, the (additive) transfer from $H$-spectra to
$G$-spectra is classical\footnote{See \eg{} \cite{LMMS}*{\S II.4}},
but the multiplicative transfer or \emph{norm} was only introduced
fairly recently by Hill--Hopkins--Ravenel as part of the foundational
setup for \cite{HHRKervaire}. This inspired a plethora of work on
equivariant symmetric monoidal structures
\cite{HopkinsHill,GMMOsymmonGcat,RubinNorm} and its relation to
equivariant homotopy-coherent commutativity (in particular
\cite{BlumbergHillNinfty} and subsequent work on
$N_{\infty}$-operads), culminating from our point of view in the
approach of
Barwick, Dotto, Glasman, Nardin, and Shah~\cite{BDGNS1}, where a
$G$-symmetric monoidal \icat{} can be viewed as a product-preserving
functor
\[ \Span(\xF_{G}) \to \CatI,\]
\ie{} a ``categorified Mackey functor''.

If the contravariant (restriction) functors have left adjoints that
satisfy base change and distributivity, \cref{thm:main} allows us to
upgrade such $G$-symmetric monoidal structures to functors from
$\BISPAN(\xF_{G})$, which encodes the distributive compatibility of
multiplicative and additive transfers. We will see that this applies
in particular to genuine $G$-spectra, giving a ``categorified Tambara
functor'' structure on $G$-spectra.

Next, we look at the motivic setting, where it is more instructive to
first work in the categorified context. By this we mean Ayoub's
construction of a functor from schemes to categories
$X \mapsto \SH(X)$ which satisfies a full six functors formalism
\cite{ayoub-thesis1}, vastly expanding Voevodsky's notes in
\cite{voevodsky-four}; we also refer the reader to the book of
Cisinski and D\'eglise \cite{cisinski-deglise} for another exposition,
\cite{hoyois-sixops} for an $\infty$-categorical enhancement of this
construction in the more general motivic-equivariant setting, as well
as the more recent \cite{drew-gallauer} for a universal property of
this construction. Here $\SH(X)$ denotes the \icat{} of motivic spectra
over a scheme $X$.

In this context, given a smooth morphism of schemes
$f \colon X \rightarrow Y$ over a base $S$, the pullback functor
$f^* \colon \SH(Y) \rightarrow \SH(X)$ admits a left adjoint,
$f_{\sharp} \colon \SH(X) \rightarrow \SH(Y)$. This is a categorified version
of the \emph{additive pushforward}: if $f$ is the fold map
$\nabla\colon Y^{\amalg I} \rightarrow Y$, then $\nabla_{\sharp}$ computes
the $I$-indexed direct sum. The compatibility of $f_{\sharp}$ with
pullbacks yields a functor
\begin{equation} \label{eq:bh1}
\SH \colon \Span_{\sm}(\Sch_S) \to \CatI.
\end{equation}

An important additional functoriality of $\SH$ was recently discovered
by Bachmann and Hoyois in \cite{norms}: given a finite \'etale
morphism $f \colon X \rightarrow Y$ we have the \emph{multiplicative
  pushforward} or \emph{norm}
$f_{\otimes} \colon \SH(X) \rightarrow \SH(Y)$, which in the case when
$f$ is the fold map computes the $I$-indexed tensor product. This
also satisfies base change, and so can
be encoded by a functor
\begin{equation} \label{eq:bh2}
  \SH \colon \Span_{\fet}(\Sch_{S}) \to \CatI,
\end{equation}
which leads to the correct notion of a coherent multiplicative structure in motivic
homotopy theory --- a \emph{normed motivic spectrum} --- as a
section of the unstraightening of~\eqref{eq:bh2} that is cocartesian
over the backwards maps in $\Span_{\fet}(\Sch_{S})$.

The motivic bispan category should combine these two structures, giving an additive pushforward
along smooth morphisms and a multiplicative pushforward along finite
\'etale morphisms. For technical reasons (due to the non-existence of
Weil restriction of schemes in general), for our motivic bispan
categories we either have to restrict to morphisms between schemes
that are smooth and quasiprojective or work with algebraic
spaces. Thus we consider 2-categories of the form
$\BISPAN_{\fet,\sm}(\AlgSpc_{S})$ where $\AlgSpc_{S}$ means the
category of algebraic spaces over $S$, and we promote $\SH$ to a functor
\begin{equation} 
\SH \colon\BISPAN_{\fet,\sm}(\AlgSpc_{S}) \to \CatI;
\end{equation}
see Theorem~\ref{thm:mot-bispans}.

The decategorification of the above structure has been studied by
Bachmann in \cite{mot-tambara}. Working over a field, and restricting
to a category of bispans between smooth schemes,
$\Bispan_{\fet,\sm}(\mathrm{Sm}_{k})$, Bachmann proved that the structure
of a normed algebra in the abelian category of homotopy modules (the
heart of the so-called homotopy $t$-structure on motivic spectra) is
encoded by certain functors out of this bispan category to abelian
groups (appropriately christened \emph{motivic Tambara functors}), at
least after inverting the exponential characteristic of $k$.

We also note that there is a discrepancy with the classical and
finite-equivariant story: finite \'etale transfers are \emph{a priori}
not sufficient to encode the structure of a motivic spectrum. Instead,
the correct kind of transfers are framed transfers in the sense of
\cite{EHKSY1}. In particular, the category of framed correspondences,
where the backward maps encode framed transfers, is manifestly an
$\infty$-category. In other words, the additive and multiplicative
transfers are rather different in the motivic story; for example we do
not know if the space of units of a normed motivic spectrum has framed
transfers (see \cite[Section 1.5]{norms} for a discussion). For us,
this means that a more robust theory of bispans in the motivic setting
which encodes framed transfers is open for future investigations.

Finally, we consider an example in the context of derived algebraic
geometry: If $\Perf(X)$ denotes the \icat{} of perfect quasicoherent
sheaves on a spectral Deligne--Mumford stack,
Barwick~\cite{BarwickMackey} has shown that the pullback and
pushforward functors extend to a functor
\[ \Perf \colon \Span_{\FP}(\SpDM) \to \CatI\]
where $\SpDM$ is the \icat{} of spectral Deligne--Mumford stacks and
$\FP$ is a certain class of morphisms (for which pushforwards preserve
perfect objects and base change is satisfied). We promote this to a
functor of \itcats{}
\[ \BISPAN_{\fet,\FP'}(\SpDM)^{\twop} \to \CATI,\]
using a multiplicative pushforward functor for finite \'etale maps
(which exists by results of Bachmann--Hoyois \cite{norms}), where $\FP'$ is a certain
subclass of $\FP$ for which Weil restrictions exist.

\subsection{Norms in algebraic $K$-theory}
A combination of the present work and \cite{polynomials} produces
concrete examples of Tambara functors valued in $\mathcal{S}$, the
$\infty$-category of spaces, via (connective) algebraic
$K$-theory, which we discuss in \S\ref{sec:tambara}. The motivation
for our results traces back to classical
representation theory: given a finite group $G$, the classical
representation ring of $G$ is a certain Grothendieck ring:
\begin{equation}\label{eq:rep}
\mathrm{Rep}(G,\mathbb{C}) \cong K_0(\Fun(BG, \Vect^{\mathrm{fd}}_{\mathbb{C}})).
\end{equation}
In particular, the tensor product of representations induces the
multiplicative structure on $\mathrm{Rep}(G,\mathbb{C})$. It is
natural to consider the formation of representation rings as a functor
in $G$, where the functoriality encodes various operations in
representation theory such as induction and restrictions. From this
viewpoint, one can enhance the multiplicative structure on
$\mathrm{Rep}(G,\mathbb{C})$ to one parametrized by cosets: if
$K \subset G$ is a subgroup, then we have a map given by the operation
of \emph{tensor induction}:
\[
\mathrm{Rep}(K,\mathbb{C}) \rightarrow \mathrm{Rep}(G,\mathbb{C}) \qquad V \mapsto \otimes_{G/K} V
\]
As reviewed in Example~\ref{ex:g-rep}, all of this functoriality can
be encoded as a functor out of a bispan category formed out of finite
$G$-sets.

Representation theory with ``fancy coefficients'' entails replacing the
category $\mathrm{Vect}^{\mathrm{fd}}_{\mathbb{C}}$ with a more
sophisticated symmetric monoidal ($\infty$-)category
$\mathcal{C}$. This line of investigation arguably began with the subject
of modular representation theory, which takes
$\mathcal{C}$ to be $\mathrm{Vect}^{\mathrm{fd}}_{\mathbb{F}_p}$. This is
especially subtle when $G$ is a $p$-group because of the failure of
the category $\Fun(BG, \Vect^{\mathrm{fd}}_{\mathbb{F}_p})$ to be
semisimple. More recently, Treumann has also considered replacing
vector spaces with $KU$-module spectra, thus taking $\mathcal{C}$ to be
$\Perf_{\mathrm{KU}}$ where $KU$ is the
$\mathbb{E}_{\infty}$-ring spectrum representing complex topological
$K$-theory \cite{treumann} and suggests that, up to $p$-completion, representation theory
over $KU$ is, in a precise way, a smooth deformation of representation
theory over the $p$-adic integers \cite[1.7]{treumann}.

In light of the last example, which is homotopical in nature, it is
natural to consider the $K$-theory \emph{space} $\Omega^{\infty}K(\Fun(BG, \mathcal{C}))$
for $\mathcal{C}$ a small stable $\infty$-category; its homotopy groups are the higher $K$-groups. Using the
universal property of bispans, coupled with the main result of
\cite{polynomials}, we offer the following result concerning its
functoriality in the variable $G$:
\begin{thm}\label{thm:rep-thy-main}
  Let $G$ be a finite group and
  $\mathcal{C}$ a symmetric monoidal $\infty$-category, then the
  presheaf on $G$-orbits:
\[
\Omega^{\infty}K_{G}(\mathcal{C})\colon\mathcal{O}_{G}^{\op} \rightarrow \mathcal{S}  \qquad G/H \mapsto \Omega^{\infty}K(\Fun(BH, \mathcal{C})),
\]
extends canonically as
\[
\begin{tikzcd}
\mathcal{O}_{G}^{\op}  \ar{rr}{\Omega^{\infty}K_{G}(\mathcal{C})} \ar{d}& &  \mathcal{S}\\
\Bispan(\xF_G) \ar[dashed,swap]{urr}{\Omega^{\infty}\widetilde{K}_{G}(\mathcal{C})}  & &.
\end{tikzcd}
\]
\end{thm}

The proof of Theorem~\ref{thm:rep-thy-main} is quite simple given our
main theorem: we use this to deduce that $\Fun(BG,\mathcal{C})$ upgrades
to a functor out of bispans into the \icat{} of
$\infty$-categories. Using a general criterion which we detail in
\S\ref{sec:poly-p}, we prove that the multiplicative pushforward
enjoys the property of being a \emph{polynomial functor} in the sense
of Goodwillie calculus. A recent breakthrough of Barwick, Glasman,
Mathew and Nikolaus \cite{polynomials} proves that the formation of
algebraic $K$-theory spaces is functorial in polynomial functors,
which then gives Theorem~\ref{thm:rep-thy-main}.

We believe that Theorem~\ref{thm:rep-thy-main} could have
computational significance. Using an equivariant analog of
\cite[Example 7.25]{norms}, Theorem~\ref{thm:rep-thy-main} produces
genuine-equivariant \emph{power operations} on these
representation-theoretic gadgets. In particular, the coherence
afforded by the language of bispans results in relations between
these operations.

By exactly the same methods we also show that the algebraic
$K$-theory of a genuine $G$-$E_{\infty}$-ring spectrum is also a
Tambara functor valued in $\mathcal{S}$. The latter are expected to be
precisely the \emph{connective} $G$-$E_{\infty}$-ring spectra, and
given this we prove:
\begin{thm}[see Theorem~\ref{genuinering}] The formation of algebraic $K$-theory preserves $G$-$E_{\infty}$-rings.
\end{thm}
We believe this is a completely new structure on
algebraic $K$-theory in equivariant homotopy theory that significantly
extends several recent results in the literature; see
Remark~\ref{rem:g-stuff}. 

In the algebro-geometric context the same method also shows that, for
instance the algebraic $K$-theory of schemes (or more generally
spectral Deligne--Mumford stacks) has multiplicative transfers along
finite \'etale morphisms.

\subsection{Notation}
This paper is written in the language of \icats{}. We use the following reasonably standard notation:
\begin{itemize}
\item $\mathcal{S}$ is the \icat{} of spaces, \ie{} $\infty$-groupoids.
\item $\CatI$ is the \icat{} of \icats{}.
\item $\CATI$ is the $(\infty,2)$-category of \icats{}.
\item $\CatIT$ is the \icat{} of \itcats{}.
\item If $\mathcal{C}$ and $\mathcal{D}$ are \icats{} or
  \itcats{}, we write $\Fun(\mathcal{C},\mathcal{D})$
  for the \icat{} of functors from $\mathcal{C}$ to $\mathcal{D}$.
\item If $\mathcal{C}$ and $\mathcal{D}$ are $(\infty,2)$-categories,
  we write $\FUN(\mathcal{C},\mathcal{D})$ for the
  $(\infty,2)$-category of functors from $\mathcal{C}$ to
  $\mathcal{D}$.
\item We write $(\blank)^{(1)} \colon \Cat_{(\infty,2)} \to \CatI$ for the
  functor taking an \itcat{} to its underlying \icat{}. (Thus
  $(\blank)^{(1)}$ is right adjoint to the inclusion of $\CatI$ into
  $\CatIT$.)
\item We write $(\blank)^{\simeq}$
  for the functors $\CatI \to \mathcal{S}$
  and $\CatIT \to \mathcal{S}$ taking an \icat{} or \itcat{} to its
  underlying \igpd{}.
\item If $\mathcal{C}$ is an \icat{} and $x$ and $y$ are objects of
  $\mathcal{C}$, we write $\Map_{\mathcal{C}}(x,y)$ for the space of
  maps from $x$ to $y$ in $\mathcal{C}$.
\item If $\mathcal{C}$ is an \itcat{} and $x$ and $y$ are objects of
  $\mathcal{C}$, we write $\MAP_{\mathcal{C}}(x,y)$ for the \icat{} of
  maps from $x$ to $y$ in $\mathcal{C}$.
\item If $\mathcal{X}$ is an \itcat{} we write $\mathcal{X}^{\onop}$ for
  the \itcat{} obtained by reversing the morphisms in $\mathcal{X}$
  and $\mathcal{X}^{\twop}$ for that obtained by reversing the
  2-morphisms.
\end{itemize}
We also adopt the following standard notation for functors between slices of an \icat{} $\mathcal{C}$:
\begin{itemize}
\item If $f \colon x \to y$ is a morphism in $\mathcal{C}$, we have a
  functor \[f_!\colon \mathcal{C}_{/x} \rightarrow \mathcal{C}_{/y}\]
  such that $f_!(t \rightarrow x) = t \rightarrow x \rightarrow y$,
  \ie{} given by composition with $f$.
\item If pullbacks
along $f$ exist in $\mathcal{C}$ then
$f_{!}$ has a right adjoint
\[f^*\colon \mathcal{C}_{/y} \rightarrow \mathcal{C}_{/x}.\]
\item If $f^{*}$ has a further right adjoint, this will be denoted by
  \[ f_* \colon \mathcal{C}_{/x} \to \mathcal{C}_{/y}.\]
  (This right adjoint exists for all $f$ precisely when $\mathcal{C}$ is locally cartesian closed, for example if $\mathcal{C}$ is an $\infty$-topos.)
\end{itemize}

\subsection{Acknowledgments}
Much of this paper was written while the second author was employed by
the IBS Center for Geometry and Physics in a position funded by the
grant IBS-R003-D1 of the Institute for Basic Science, Republic of
Korea.  The first author would like to thank Clark Barwick for his
influence in thinking about things in terms of spans. We thank Andrew
Macpherson and Irakli Patchkoria for helpful discussions, Marc Hoyois for some corrections on an earlier draft, and the
Gorilla Brewery of Busan for providing the setting for the initial
inspiration for this paper. We also thank the anonymous referee for
suggesting a much simpler approach to our main result, by reducing it
to the universal property for spans.

\section{Bispans and distributive functors}

\subsection{$(\infty,2)$-categories and adjunctions}
Throughout this paper we work with \itcats{}, and in this section we
review some basic results we use from the theory of \itcats{},
particularly regarding adjunctions. There are several
equivalent ways to define these objects and their homotopy theory,
including Rezk's $\Theta_{2}$-spaces \cite{RezkThetaN} and Barwick's
2-fold Segal spaces \cite{BarwickThesis}. We can also view \itcats{}
as \icats{} enriched in \icats{}, which can be rigidified to
categories strictly enriched in quasicategories (see \cite{enrcomp}); the latter is the
model used in the papers of Riehl and Verity. We will not review the
details of any of these constructions here, as we do not need to refer
to any particular model of \itcats{} in this paper.

We will, however, use the Yoneda lemma for \itcats{}, which is a
special case of Hinich's Yoneda lemma for enriched \icats{}
\cite{HinichYoneda}:
\begin{thm}[Hinich] \label{yoneda}
  For any \itcat{} $\mathcal{X}$ there exists a fully faithful functor
  of \itcats{}
  \[ \Yo_{\mathcal{X}} \colon \mathcal{X} \to \FUN(\mathcal{X}^{\op},
    \CATI)\]
  such that for any functor $\Phi \colon \mathcal{X}^{\op} \to \CATI$
  there is a natural equivalence of $\infty$-categories \[\Phi(d) \simeq
  \MAP_{\FUN(\mathcal{X}^{\op},\CATI)}(\Yo_{\mathcal{X}}(d), \Phi).\]
\end{thm}
\begin{remark}
  Hinich's work does use a specific model for \itcats{}, namely a
  certain definition of enriched \icats{} specialized to enrichment in
  $\CatI$. Hinich's definition has been compared to the
  original one of Gepner--Haugseng~\cite{enr} by
  Macpherson~\cite{MacphersonEnr}, and for enrichment in $\CatI$ the
  latter is equivalent to complete 2-fold Segal spaces \cite{enrcomp},
  which in turn is known by work of Barwick and
  Schommer-Pries~\cite{BarwickSchommerPriesUnicity} to be equivalent to
  most other approaches to $(\infty,2)$-categories (including the
  complicial sets of Verity by recent work of Gagna--Harpaz--Lanari~\cite{GagnaHarpazLanaryScale}).
\end{remark}

Recall that there exists a \emph{universal adjunction}. This is a
$2$-category $\ADJ$ with two objects $-$ and $+$ and generated by
$1$-morphisms
$L \colon \Delta^1 = \{ - \rightarrow + \} \rightarrow \ADJ$ and
$R \colon \Delta^1 = \{ + \rightarrow - \} \rightarrow \ADJ$ such that
$L$ is left adjoint to $R$; see \cite{RiehlVerityAdj} for an explicit
combinatorial definition of this 2-category. An adjunction in a
2-category can then
equivalently be described as a functor from $\ADJ$. This universal
property also holds in \itcats{}, where we can formulate it more
precisely as follows:
\begin{thm}[Riehl--Verity]\label{adjthm}
  Given an $(\infty,2)$-category $\mathcal{X}$, the induced maps of spaces
\[
L^*,\, R^*:\Map_{\CatIT}(\ADJ, \mathcal{X}) \rightarrow \Map_{\CatIT}(\Delta^1, \mathcal{X}),
\]
are both inclusions of components whose images are the subspaces
\[ \Map^L_{\CatIT}(\Delta^1, \mathcal{X}),\, \Map^R_{\CatIT}(\Delta^1, \mathcal{X}) \,\,\subset\,\, \Map_{\CatIT}(\Delta^1,\mathcal{X})\] spanned by those functors that are left and right
adjoints, respectively.  
\end{thm}
For details, see \cite{RiehlVerityAdj}*{Theorem 4.3.11 and
  4.4.18}. See also \cite{HarpazNuitenPrasmaInfty2} for an alternative
proof, using the cotangent complex of $(\infty,2)$-categories.

We need to upgrade this to a statement about \itcats{} rather than just
\igpds{}. This follows from the next observation, which
identifies the morphisms and 2-morphisms in the \itcat{} of
adjunctions using some results from \cite{adjmnd}; to state this we need
some terminology that will be important throughout the paper:
\begin{defn}
  Let $\mathcal{X}$ be an \itcat{} and consider a commutative square
  \[
    \begin{tikzcd}
      x' \arrow{r}{g'} \arrow{d}{\xi} & y'  \arrow{d}{\eta} \\
      x \arrow{r}{g} & y
    \end{tikzcd}
  \]
  in $\mathcal{X}$. If $g$ and $g'$ are left adjoints, with
  corresponding right adjoints $h$ and $h'$, then we can use the units
  and counits of the adjunctions to define a
  \emph{mate} (or \emph{Beck--Chevalley}) \emph{transformation} $\xi h' \to h \eta$ as the composite
  \[ \xi h' \to hg \xi h' \simeq h \eta g' h' \to h \eta.\]
  We say the square is \emph{right adjointable} if this mate
  transformation is an equivalence. Dually, if $g$ and $g'$ are right adjoints, with left adjoints $f$
  and $f'$, we say the square is \emph{left adjointable} if the mate
  transformation
  \[ f \eta \to f \eta g' f' \simeq f g \xi f' \to \xi f' \]
  is an equivalence.
\end{defn}

\begin{propn}\label{adjinarrows}
  Let $\mathcal{X}$ and $\mathcal{Y}$ be \itcats{}. A 1-morphism in
  the \itcat{} $\FUN(\mathcal{X}, \mathcal{Y})$, \ie{} a natural
  transformation $\eta \colon F \to G$ of functors $F,G \colon
  \mathcal{X} \to \mathcal{Y}$, is a right (left) adjoint \IFF{}
  \begin{enumerate}[(1)]
  \item for every object $x \in \mathcal{X}$, the morphism $\eta_{x}
    \colon F(x) \to G(x)$ is a right (left) adjoint in $\mathcal{Y}$,
  \item for every morphism $f \colon x \to x'$ in $\mathcal{X}$, the
    commutative square
    \[
      \begin{tikzcd}
        F(x) \arrow{r}{\eta_{x}} \arrow{d}{F(f)} & G(x)
        \arrow{d}{G(f)} \\
        F(x') \arrow{r}{\eta_{x'}} & G(x')
      \end{tikzcd}
    \]
    is left (right) adjointable.
  \end{enumerate}
\end{propn}
\begin{proof}
  We consider the case of right adjoints; the left adjoint case can be
  proved similarly, and also follows by duality. Let 
  $\FUN(\mathcal{X},\mathcal{Y})_{\txt{lax}}$ denote the \itcat{} of
  functors from $\mathcal{X}$ to $\mathcal{Y}$ with lax natural
  transformations as morphisms (see \cite{adjmnd}*{\S 3} for a precise
  definition). We can view the natural transformation $\eta$ as a morphism in
  $\FUN(\mathcal{X},\mathcal{Y})_{\txt{lax}}$. By
  \cite{adjmnd}*{Theorem 4.6} it has a right adjoint here \IFF{}
  $\eta_{x}$ has a right adjoint in $\mathcal{Y}$ for
  every $x \in \mathcal{X}$; this right adjoint is given on
  morphisms by taking mates. If $\eta$ has a right adjoint $\rho$ in
  $\FUN(\mathcal{X}, \mathcal{Y})$ then by uniqueness this is also a
  right adjoint in $\FUN(\mathcal{X},\mathcal{Y})_{\txt{lax}}$; hence
  $\eta$ must be given objectwise by left adjoints and the naturality
  squares of $\rho$ are the corresponding mate squares --- in
  particular, these mate squares must commute, so conditions (1) and
  (2) hold. Conversely, if these conditions hold for $\eta$ then
  $\eta$ has a right adjoint $\rho$ in
  $\FUN(\mathcal{X},\mathcal{Y})_{\txt{lax}}$ and the lax naturality
  squares of $\rho$ actually commute. By \cite{adjmnd}*{Corollary
    3.17} this means that $\rho$ is in the image of the canonical functor $\FUN(\mathcal{X},
  \mathcal{Y}) \to \FUN(\mathcal{X},\mathcal{Y})_{\txt{lax}}$;
  moreover, this functor is
  locally fully faithful so the unit 
  and counit of the adjunction also lie in $\FUN(\mathcal{X},
  \mathcal{Y})$, and so $\eta$ has a right adjoint in $\FUN(\mathcal{X},
  \mathcal{Y})$, as required.
\end{proof}

\begin{notation}
  Let $\mathcal{X}$ be an \itcat{}. We write $\FUN(\Delta^{1},
  \mathcal{X})^{\ladj}$ for the locally full sub-\itcat{} of $\FUN(\Delta^{1},
  \mathcal{X})$ whose objects are the morphisms that are left adjoints
  and whose morphisms are the right adjointable squares. Similarly, we
  write $\FUN(\Delta^{1}, \mathcal{X})^{\radj}$ for the sub-\itcat{} of
  right adjoints and left adjointable squares.
\end{notation}

\begin{cor}
  The functors $L^{*}, R^{*} \colon \FUN(\ADJ, \mathcal{X}) \to
  \FUN(\Delta^{1}, \mathcal{X})$ identify the \itcat{} $\FUN(\ADJ, \mathcal{X})$
  with the sub-\itcats{} $\FUN(\Delta^{1},
  \mathcal{X})^{\ladj}$ and $\FUN(\Delta^{1},
  \mathcal{X})^{\radj}$, respectively.
\end{cor}
\begin{proof}
  We consider the case of left adjoints; the proof for right adjoints
  is the same. For any \itcat{} $\mathcal{Y}$ we have a natural commutative square
  \[
    \begin{tikzcd}
    \Map(\mathcal{Y}, \FUN(\ADJ, \mathcal{X})) \arrow{r}{\sim}
    \arrow{d}{L^{*}} & \Map(\ADJ, \FUN(\mathcal{Y}, \mathcal{X})) \arrow{d}{L^{*}}
    \\
    \Map(\mathcal{Y}, \FUN(\Delta^{1}, \mathcal{X})) \arrow{r}{\sim} &
    \Map(\Delta^{1}, \FUN(\mathcal{Y}, \mathcal{X})).      
    \end{tikzcd}
  \]
  Here \cref{adjthm} implies that the right vertical map is a
  monomorphism of \igpds{} with image the components of
  $\Map(\Delta^{1}, \FUN(\mathcal{Y}, \mathcal{X}))$ that correspond
  to left adjoints in $\FUN(\mathcal{Y}, \mathcal{X})$. Using 
  \cref{adjinarrows} we can identify these as precisely those in the
  image of the subspace $\Map(\mathcal{Y}, \FUN(\Delta^{1},
  \mathcal{X})^{\ladj})$ under the bottom horizontal equivalence. By
  the Yoneda Lemma it follows that $L^{*} \colon \FUN(\ADJ,
  \mathcal{X}) \to \FUN(\Delta^{1}, \mathcal{X})^{\ladj}$ is an equivalence.
\end{proof}

\subsection{Adjointable functors and spans}\label{sec:spans}
In this section we introduce (left and right) adjointable functors and
review the universal property of \itcats{} of spans in terms of these.

\begin{defn}
  A \emph{span pair} $(\mathcal{C},\mathcal{C}_{F})$ consists of an
  \icat{} $\mathcal{C}$ together with a wide subcategory $\mathcal{C}_{F}$
  (\ie{} one containing all objects and equivalences) such that given morphisms $x \xto{f} y$ in
  $\mathcal{C}_{F}$ and $z \xto{g} y$ in
  $\mathcal{C}$, the pullback
\[
  \begin{tikzcd}
    x \times_{y} z \arrow{r}{f'} \arrow{d}[swap]{g'} \drpullback & z \arrow{d}{g} \\
    x \arrow{r}{f} & y
  \end{tikzcd}
\]
exists in $\mathcal{C}$, and moreover $f'$ is also in
$\mathcal{C}_{F}$. If $(\mathcal{C}, \mathcal{C}_{F})$ and
$(\mathcal{C}',\mathcal{C}'_{F})$ are span pairs, then a morphism of
span pairs
$(\mathcal{C}, \mathcal{C}_{F}) \to (\mathcal{C}',\mathcal{C}'_{F})$
is a functor $\phi \colon \mathcal{C} \to \mathcal{C}'$ such that
$\phi(\mathcal{C}_{F}) \subseteq \mathcal{C}'_{F}$ and $\phi$
preserves pullbacks along morphisms in $\mathcal{C}_{F}$. We write
$\Pair$ for the \icat{} of span pairs, which can be defined as a
subcategory of $\Fun(\Delta^{1}, \CatI)$.
\end{defn}

Given a span pair $(\mathcal{C}, \mathcal{C}_{F})$ we can, as in
\cite[Section 3]{BarwickMackey}, define an \icat{}
$\Span_{F}(\mathcal{C})$ whose objects are the objects of
$\mathcal{C}$, with morphisms from $x$ to $y$ given by spans
\[
  \begin{tikzcd}
    {} & z \arrow{dl}[above left]{g} \arrow{dr}{f} \\
    x & & y,
  \end{tikzcd}
\]
where $f$ is in $\mathcal{C}_{F}$; we
compose spans by taking pullbacks. Following \cite[Section 5]{spans} we can
upgrade this to an \itcat{} $\SPAN_{F}(\mathcal{C})$ whose
2-morphisms are morphisms of spans, \ie{} diagrams
\[
  \begin{tikzcd}
    {} & z \arrow{dd} \arrow{dl} \arrow{dr} \\
    x & & y \\
    &  z' \arrow{ul} \arrow{ur}
  \end{tikzcd}
\]
in $\mathcal{C}$, where $z \to z'$ can be any morphism in
$\mathcal{C}$.

\begin{warning}
  In \cite{spans}, the \itcat{} of spans in $\mathcal{C}$ was denoted
  $\Span_{1}^{+}(\mathcal{C})$, while $\SPAN_{n}(\mathcal{C})$ was
  used for an $n$-uple \icat{} of spans.
\end{warning}

\begin{remark}
  In the \itcat{} $\SPAN_{F}(\mathcal{C})$, every morphism of the form
  \[ [f]_{B} :=  x \xfrom{f} z \xto{\id} z \] with $f$ in $\mathcal{C}_{F}$ has a left
  adjoint, namely the reversed span
  \[ [f]_{F} := z \xfrom{\id} z \xto{f} x.\]
  The counit is the 2-morphism
  \[
    \begin{tikzcd}
      {} & x \arrow{dl}[swap]{f} \arrow{dd}{f} \arrow{dr}{f} & & & {[f]_{F}\circ [f]_{B}} \arrow{dd} \\
      y & & y \\
       & y \arrow[equals]{ul}\arrow[equals]{ur} & & & \id_{y},
     \end{tikzcd}
   \]
   and the unit is
  \[
    \begin{tikzcd}
      {} & x \arrow[equals]{dl} \arrow{dd}{\Delta} \arrow[equals]{dr} & & & \id_{x} \arrow{dd} \\
      x & & x \\
       & x \times_{y}x \arrow{ul} \arrow{ur} & & & {[f]_{B} \circ [f]_{F},}
     \end{tikzcd}
   \]
   where the fibre product $x \times_{y}x$ is over two copies of $f$
   and $\Delta$ is the corresponding diagonal.
\end{remark}

The \itcat{} $\SPAN_{F}(\mathcal{C})$ enjoys a universal property:
roughly speaking, it is obtained from $\mathcal{C}^{\op}$ by freely
adding left adjoints for morphisms in $(\mathcal{C}_{F})^{\op}$. To
state this more precisely, we need some definitions:
\begin{defn}
  Let $(\mathcal{C}, \mathcal{C}_{F})$ be a span
  pair and $\mathcal{X}$ an \itcat{}. A functor $\Phi \colon \mathcal{C}^{\op}
  \to \mathcal{X}$ is \emph{left $F$-preadjointable}
  if for every morphism $f \colon x \to y$ in $\mathcal{C}_{F}$ the
  1-morphism $f^{\ostar} := \Phi(f) \colon \Phi(y) \to \Phi(x)$ in $\mathcal{X}$ has a left
  adjoint $f_{\oplus}$ in $\mathcal{X}$. 
\end{defn}

\begin{defn}
  We say that $\Phi$ is \emph{left
    $F$-adjointable} if it is left $F$-preadjointable and
  for every cartesian square
    \begin{equation} \label{eq:pull}
   \begin{tikzcd}
      x \times_{y} z \arrow{r}{f'} \arrow{d}{g'} & z \arrow{d}{g} \\
      x \arrow{r}{f} & y,
    \end{tikzcd}
  \end{equation}
  in $\mathcal{C}$ with $f$ in
  $\mathcal{C}_{F}$, the commutative square
  \[
    \begin{tikzcd}
      \Phi(y) \arrow{r}{f^{\ostar}} \arrow{d}{g^{\ostar}} & \Phi(x)
      \arrow{d}{g'^{\ostar}} \\
      \Phi(z) \arrow{r}{f'^{\ostar}} & \Phi(x \times_{y} z)
    \end{tikzcd}
  \]
  in $\mathcal{X}$ is left adjointable. We write
  $\Map_{\Fladj}(\mathcal{C}^{\op}, \mathcal{X})$ for the subspace of
  $\Map(\mathcal{C}^{\op},\mathcal{X})$ whose components are the left
  $F$-adjointable functors.
\end{defn}

\begin{remark}
  In other words, $\Phi$ is left $F$-adjointable if for every
  cartesian square \cref{eq:pull} with $f$ in $\mathcal{C}_{F}$, the
  Beck--Chevalley transformation
  \begin{equation}\label{eq:bcr}
f'_{\oplus}g'^{\ostar} \to   g^{\ostar} f_{\oplus}
  \end{equation}
  is an equivalence.
\end{remark}

\begin{thm}[Gaitsgory-Rozenblyum \cite{GaitsgoryRozenblyum1}, Macpherson \cite{MacphersonCorr}]\label{thm:spanuniv}
  Let $(\mathcal{C}, \mathcal{C}_{F})$ be a span
  pair, and let
  $\mathcal{X}$ be an \itcat{}. The inclusion of the backwards maps
  $\mathcal{C}^{\op} \to \SPAN_{F}(\mathcal{C})$ gives a
  monomorphism of $\infty$-groupoids
  \[ \Map(\SPAN_{F}(\mathcal{C}), \mathcal{X}) \to
    \Map(\mathcal{C}^{\op}, \mathcal{X}) \] with image $\Map_{\Fladj}(\mathcal{C}^{\op},\mathcal{X})$.
\end{thm}

\begin{remark} 
  In~\cite{GaitsgoryRozenblyum1}, Gaitsgory-Rozenblyum make use of this
  universal property of spans in order to encode the functoriality of 
  various \icats{} of coherent sheaves on derived schemes. They also
  sketch a proof of Theorem~\ref{thm:spanuniv} using a particular
  construction of $\SPAN_{F}(\mathcal{C})$.  Macpherson~\cite{MacphersonCorr} has recently
  given an alternative, model-independent (and complete) proof.   Roughly
  speaking, Macpherson's approach is to first show there exists an
  \itcat{} that represents left adjointable functors and then use the
  universal property to prove that this representing object has the
  expected description in terms of spans. The universal property has
  also been extended to higher categories of iterated spans by
  Stefanich~\cite{StefanichCorr}.
\end{remark}

\begin{variant}\label{var:radjble}
  If $\mathcal{X}$ is an \itcat{}, then we have an equivalence of
  underlying \icats{}
  \[ \mathcal{X}^{(1)} \simeq
  (\mathcal{X}^{\twop})^{(1)},\] while a 1-morphism in $\mathcal{X}$
  is a right adjoint \IFF{} it is a left adjoint in
  $\mathcal{X}^{\twop}$. We therefore say that a functor $\Phi \colon
  \mathcal{C}^{\op} \to \mathcal{X}$ is \emph{right
    $F$-adjointable} if the 2-opposite
  functor
  \[\mathcal{C}^{\op} \simeq (\mathcal{C}^{\op})^{\twop} \xto{\Phi^{\twop}} 
  \mathcal{X}^{\twop}\] is left $F$-adjointable.
  Theorem~\ref{thm:spanuniv} then tells us that
  the right $F$-adjointable functors correspond
  to functors $\SPAN_{F}(\mathcal{C})^{\twop} \to \mathcal{X}$.
\end{variant}

\begin{variant}\label{var:rcoadj}
  We say a functor $\Phi \colon \mathcal{C} \to \mathcal{X}$ is
  \emph{right $F$-coadjointable} if the 1-opposite functor
  \[ \Phi^{\op} \colon \mathcal{C}^{\op} \to \mathcal{X}^{\op} \]
  is left $F$-adjointable. Theorem~\ref{thm:spanuniv} then tells us that
  the right $F$-coadjointable functors correspond
  to functors $\SPAN_{F}(\mathcal{C})^{\op} \to \mathcal{X}$. We can
  also combine both variants, and say that $\Phi \colon \mathcal{C}
  \to \mathcal{X}$ is \emph{left $F$-coadjointable} if
  \[ \mathcal{C}^{\op} \simeq \mathcal{C}^{\op,\twop}
    \xto{\Phi^{\op,\twop}} \mathcal{X}^{\op,\twop} \] is left
  $F$-adjointable. The left $F$-coadjointable functors then correspond
  to functors $\SPAN_{F}(\mathcal{C})^{\op,\twop} \to \mathcal{X}$.
\end{variant}

\begin{remark}
  Unpacking the definition, and recalling that reversing the 1-morphisms
  in an \itcat{} swaps left and right adjoints, we see that a functor
  $\Phi \colon \mathcal{C} \to \mathcal{X}$ is right $F$-coadjointable
  if for every morphism $f \colon x \to y$ in $\mathcal{C}_{F}$ the
  1-morphism $f_{\ostar} := \Phi(f) \colon \Phi(x) \to \Phi(y)$ in
  $\mathcal{X}$ has a right
  adjoint $f^{\oplus}$, and for every pullback square \cref{eq:pull} 
  the square
  \[
    \begin{tikzcd}
      \Phi(x \times_{y} z) \arrow{r}{f'_{\ostar}}
      \arrow{d}{g'_{\ostar}} & \Phi(z) \arrow{d}{g_{\ostar}} \\
      \Phi(x) \arrow{r}{f_{\ostar}} & \Phi(y)
    \end{tikzcd}
  \]
  is right adjointable. Note that this is \emph{not} the same as
  $\Phi$ being right $F$-adjointable in the previous sense: this
  condition would involve adjointability for \emph{pushout} squares in
  $\mathcal{C}$ --- indeed, to even be defined this condition would
  require $(\mathcal{C}^{\op},\mathcal{C}^{\op}_{F})$ to be a span
  pair, which may well not be the case.
\end{remark}

\subsection{The $(\infty,2)$-category of spans}
\label{sec:infty-2-category}

For the sake of completeness, in this subsection we include a proof of
\cref{thm:spanuniv}. However, our argument is at most a minor
variation of the proof of Macpherson~\cite{MacphersonCorr} and we make
no claims to originality.  We start by observing that the
presentability of $\CatIT$ implies that left $F$-adjointable functors
are corepresented by some \itcat{}:
\begin{propn}\label{adjrepble}
  Let $(\mathcal{C}, \mathcal{C}_{F})$ be a span pair, with
  $\mathcal{C}$ a small \icat{}. Then the functor
  \[ \Map_{\Fladj}(\mathcal{C}, \blank) \colon \CatIT \to
    \mathcal{S}\]
  is corepresentable by a
  small \itcat{} $\SPAN_{F}(\mathcal{C})$, so that there is a natural
  equivalence
  \begin{equation}
    \label{eq:ladjeq}
  \Map_{\Fladj}(\mathcal{C}, \mathcal{X}) \simeq
    \Map_{\CatIT}(\SPAN_{F}(\mathcal{C}), \mathcal{X})    
  \end{equation}
  for any $\mathcal{X} \in \CatIT$.
\end{propn}
\begin{proof}
  The \icat{} $\CatIT$ is presentable, for instance because it can be
  described as presheaves on $\Theta_{2}$ satisfying Segal and
  completeness conditions, which gives an explicit presentation as an
  accessible localization of an \icat{} of presheaves. To prove that a
  copresheaf on $\CatIT$ is corepresentable it therefore suffices by
  \cite[Proposition 5.5.2.7]{HTT} to
  show that it is accessible and preserves limits.

  We first show that this holds for the copresheaf
  $\Map_{\Flpadj}(\mathcal{C}^{\op}, \blank)$ of left
  $F$-preadjointable functors. By definition, a functor
  $\mathcal{C} \to \mathcal{X}$ is left $F$-preadjointable if it takes
  every morphism in $\mathcal{X}$ to a left adjoint in
  $\mathcal{X}$. We can therefore write
  $\Map_{\Flpadj}(\mathcal{C}^{\op}, \mathcal{X})$ as the pullback
  \[
    \begin{tikzcd}
      \Map_{\Flpadj}(\mathcal{C}^{\op}, \mathcal{X}) \arrow{r} \arrow{d} &
      \Map_{\CatIT}(\mathcal{C}^{\op}, \mathcal{X}) \arrow{d}{(f^{*})_{f \in
        S}} \\
      \prod_{f \in S} \Map_{\CatIT}(\ADJ, \mathcal{X}) \arrow{r}{R^{*}} &
      \prod_{f \in S} \Map_{\CatIT}(\Delta^{1}, \mathcal{X})
    \end{tikzcd}
    \]
   where the product is over the set $S$ of equivalence classes of
   morphisms in $\mathcal{C}_{F}$ and $R \colon \Delta^{1} \to \ADJ$
   is the inclusion of the
   right adjoint of the universal adjunction. From this description it is
   immediate that $\Map_{\Flpadj}(\mathcal{C}^{\op}, \mathcal{X})$ preserves
   limits in $\mathcal{X}$, since this is clear for the other three
   corners of the square. Moreover, since $\CatIT$ is presentable we
   can choose a regular cardinal $\kappa$ such that $S$ is
   $\kappa$-small and the objects $\mathcal{C}^{\op}$, $\ADJ$, and
   $\Delta^{1}$ are all $\kappa$-compact in $\CatIT$. Then we see that
   $\Map_{\Flpadj}(\mathcal{C}^{\op},\blank)$ preserves $\kappa$-filtered
   colimits, since the other corners of the pullback square do so (as
   $\kappa$-filtered colimits in $\mathcal{S}$ commute with
   $\kappa$-small limits, such as our product over $S$).

   For $\Map_{\Fladj}(\mathcal{C}^{\op}, \mathcal{X})$ we impose the
   additional requirement that every cartesian square \cref{eq:pull}
   in $\mathcal{C}$
   where the horizontal maps are in $\mathcal{C}_{F}$ is taken
   to a left adjointable square in $\mathcal{X}$. By
   \cref{adjinarrows} the left adjointable squares are the right
   adjoints in $\mathcal{X}^{\Delta^{1}}$, so if $S'$ denotes the set
   of equivalence classes of relevant cartesian squares in
   $\mathcal{C}$, we can write $\Map_{\Fladj}(\mathcal{C}^{\op},
   \mathcal{X})$ as a pullback
   \[
    \begin{tikzcd}
      \Map_{\Fladj}(\mathcal{C}^{\op}, \mathcal{X}) \arrow{r} \arrow{d} &
      \Map_{\Flpadj}(\mathcal{C}^{\op}, \mathcal{X}) \arrow{d} \\
      \prod_{S'} \Map_{\CatIT}(\ADJ \times \Delta^{1}, \mathcal{X})
      \arrow{r}{(R \times \id)^{*}} &
      \prod_{S'} \Map_{\CatIT}(\Delta^{1} \times \Delta^{1}, \mathcal{X}).
    \end{tikzcd}
  \]
  The same argument as for left $F$-preadjointable maps now implies
  that the presheaf $\Map_{\Fradj}(\mathcal{C}, \blank)$ is also
  accessible and preserves limits, and hence is corepresentable.
\end{proof}

\begin{remark}
  The identity of $\SPAN_{F}(\mathcal{C})$ corresponds under
  \cref{eq:ladjeq} to a left $F$-adjointable functor
  \[i \colon \mathcal{C}^{\op} \to \SPAN_{F}(\mathcal{C}),\] such that
  the equivalence \cref{eq:ladjeq} arises by restriction along $i$.
\end{remark}

\begin{remark}\label{rmk:SPANftr}
  From the universal property we immediately obtain a functor from
  span pairs to \itcats{}: For any \itcat{} $\mathcal{X}$, composition
  with a morphism of span pairs
  $\phi \colon (\mathcal{C}, \mathcal{C}_{F}) \to (\mathcal{C}',
  \mathcal{C}'_{F'})$ restricts to a morphism
  \[ \Map_{F'\dladj}(\mathcal{C}', \mathcal{X}) \to
    \Map_{\Fladj}(\mathcal{C}, \mathcal{X}),\]
  natural in $\mathcal{X} \in \CatIT$. 
  We obtain a functor
  \[\Map_{(\blank)\dladj}(\blank,\blank) \colon \Pair^{\op}
    \times \CatIT \to \mathcal{S}.\] \cref{adjrepble} says that the
  corresponding functor $\Pair^{\op} \to \Fun(\CatIT, \mathcal{S})$
  takes values in corepresentable copresheaves, and so by the Yoneda
  lemma factors through a canonical functor
  $ \SPAN \colon \Pair \to \CatIT$.
\end{remark}

We can upgrade the equivalence of \cref{adjrepble} to a statement
at the level of \itcats{}, rather than just \igpds{}. To state this we
first need some notation:
\begin{defn}
  Let $(\mathcal{C}, \mathcal{C}_{F})$ be a span pair and
  $\mathcal{X}$ an \itcat{}. We say that a natural transformation
  $\eta \colon \mathcal{C}^{\op} \times \Delta^{1} \to
  \mathcal{X}$ is \emph{left $F$-adjointable} if it
  corresponds to a left $F$-adjointable functor
  $\mathcal{C}^{\op} \to \FUN(\Delta^{1},
  \mathcal{X})$. From \cref{adjinarrows} it follows that $\eta$ is
  left $F$-adjointable \IFF{} the components $\eta_{0},\eta_{1}$ are
  both left $F$-adjointable, and for every morphism
  $f \colon x \to y$ in $\mathcal{C}_{F}$, the naturality
  square
  \[
    \begin{tikzcd}
          \eta_{0}(y) \arrow{d}{\eta_{y}} \arrow{r}{f^{\ostar}} &
          \eta_{0}(x) \arrow{d}{\eta_{x}} \\
          \eta_{1}(y) \arrow{r}{f^{\ostar}} & \eta_{1}(x)
    \end{tikzcd}
  \]
  is left adjointable.  Let
  $\Fun_{\Fladj}(\mathcal{C}^{\op}, \mathcal{X})$
  denote the subcategory of
  $\Fun(\mathcal{C}^{\op}, \mathcal{X})$ whose objects are
  the left $F$-adjointable functors and whose morphisms are the left
  $F$-adjointable transformations. From \cref{adjinarrows} we also
  know that a morphism
  $\mathcal{C} \times \mathsf{C}_{2} \to \mathcal{X}$ (where
  $\mathsf{C}_{2}$ is the 2-cell) corresponds to a left
  $F$-adjointable morphism
  $\mathcal{C} \to \FUN(\mathsf{C}_{2}, \mathcal{X})$ \IFF{} the
  component functors and natural transformations are left
  $F$-adjointable. We therefore write
  $\FUN_{\Fladj}(\mathcal{C}^{\op}, \mathcal{X})$ for the locally full
  sub-\itcat{} of $\FUN(\mathcal{C}^{\op}, \mathcal{X})$ whose
  underlying \icat{} is
  $\Fun_{\Fladj}(\mathcal{C}^{\op}, \mathcal{X})$.
\end{defn}

\begin{cor} \label{cor:improve}
  Let $(\mathcal{C}, \mathcal{C}_{F})$ be a span pair and
  $\mathcal{X}$ an \itcat{}. Composition with $i \colon
  \mathcal{C}^{\op} \to \SPAN_{F}(\mathcal{C})$ gives an
  equivalence of \itcats{}
  \[ \FUN(\SPAN_{F}(\mathcal{C}), \mathcal{X}) \isoto
    \FUN_{\Fladj}(\mathcal{C}^{\op}, \mathcal{X}).\]
\end{cor}
\begin{proof}
  For any \itcat{} $\mathcal{Y}$ we have a natural equivalence
  \[
    \begin{split}
      \Map(\mathcal{Y}, \FUN(\SPAN_{F}(\mathcal{C}), \mathcal{X}))  &
      \simeq \Map(\SPAN_{F}(\mathcal{C}), \FUN(\mathcal{Y},
      \mathcal{X})) \\
      & \simeq \Map_{\Fladj}(\mathcal{C}^{\op}, \FUN(\mathcal{Y},
      \mathcal{X})) \\
      & \simeq \Map(\mathcal{Y}, \FUN_{\Fladj}(\mathcal{C}^{\op}, \mathcal{X})),
    \end{split}
  \]
  where the last equivalence follows from the description of adjoints in
  functor \itcats{} in \cref{adjinarrows}.
\end{proof}

\begin{variant} \label{variant:rcoadj}
  Using analogous notation for right $F$-coadjointable functors, we have
  a natural equivalence
  \[\FUN_{\Frcoadj}(\mathcal{C},
  \mathcal{X}) \simeq \FUN(\SPAN_{F}(\mathcal{C})^{\op}, \mathcal{X}).\]
\end{variant}

As
a first step toward getting a handle on the \itcat{}
$\SPAN_{F}(\mathcal{C})$ we have the following observation:
\begin{lemma} \label{lem:ess}
  The functor $i \colon \mathcal{C}^{\op} \to
  \SPAN_{F}(\mathcal{C})$ (corresponding to the identity under
  \cref{eq:ladjeq}) is essentially surjective.
\end{lemma}
\begin{proof}
  Let $\mathcal{I}$ denote the full sub-$(\infty,2)$-category of
  $\SPAN_{F}(\mathcal{C})$ spanned by the objects in the image of
  $i$. Then $i$ factors through $i' \colon \mathcal{C}^{\op}
  \to \mathcal{I}$, and $i'$ is again left $F$-adjointable (since the
  relevant adjoints and 2-morphisms all live in $\mathcal{I}$). Hence $i'$
  corresponds to a functor $\SPAN_{F}(\mathcal{C}) \to
  \mathcal{I}$ such that the composite
  \[ \SPAN_{F}(\mathcal{C}) \to \mathcal{I} \to
    \SPAN_{F}(\mathcal{C})\] is the identity. It follows that the
  inclusion of $\mathcal{I}$ must be essentially surjective, which
  means that $i$ is also essentially surjective.
\end{proof}

To get a handle on the \icats{} of morphisms in
$\SPAN_{F}(\mathcal{C})$, we will use two further ingredients: (1) the
Yoneda Lemma for \itcats{} (\cref{yoneda}) due to Hinich, and (2) the
construction of the free (co)cartesian fibrations due to Gepner,
Nikolaus, and the second author \cite{freepres}. Applying the Yoneda
lemma to $\SPAN_{F}(\mathcal{C})$ we get a canonical family of left
$F$-adjointable functors to $\CATI$:
\begin{propn}\label{coadjyoneda}
  There is a left $F$-adjointable functor
  \[ \mathsf{Y} \colon \mathcal{C}^{\op} \to
    \FUN_{\Frcoadj}(\mathcal{C}, \CATI)\] such that for any right
  $F$-coadjointable functor (in the sense of \cref{var:rcoadj})
  $\Phi \colon \mathcal{C} \to \CATI$ there is a natural equivalence
  \[ \Phi(c) \simeq \MAP_{\Frcoadj}(\mathsf{Y}(c), \Phi),\]
  where the latter denotes the \icat{} of right $F$-coadjointable natural
  transformations.
\end{propn}

\begin{proof}
  Applying Hinich's Yoneda embedding (\cref{yoneda}) to
  $\SPAN_{F,L}(\mathcal{C})$
  we get a functor
  \[ \Yo \colon \SPAN_{F}(\mathcal{C}) \to
    \FUN(\SPAN_{F}(\mathcal{C})^{\op}, \CATI).\]
  By \cref{adjrepble} this corresponds to a left
  $F$-adjointable functor
  \[ \mathsf{Y} \colon \mathcal{C}^{\op} \to \FUN(\SPAN_{F}(\mathcal{C})^{\op}, \CATI)
    \simeq \FUN_{\Frcoadj}(\mathcal{C}, \CATI)
  \]
  via the equivalence of \cref{variant:rcoadj}. Translating the
  universal property of representable presheaves through the latter
  equivalence now gives the result.
\end{proof}

To proceed further, we will work unstraightened, \ie{} with the
cocartesian fibrations corresponding to $\mathsf{Y}(c) \colon
\mathcal{C} \to \CatI$. %

\begin{lemma}\label{lem:Fcoadjstr}
  The straightening equivalence
  \[\Fun(\mathcal{C}, \CatI) \simeq
    \Cat_{\infty/\mathcal{C}}^{\txt{cocart}}\] identifies the
  subcategory $\Fun_{\Frcoadj}(\mathcal{C}, \CATI)$ with a full
  subcategory
  \[ \Cat_{\infty/\mathcal{C}}^{\Frcoadj} \subseteq
    \Cat_{\infty/\mathcal{C}}^{\txt{cocart}+F\txt{-cart}} := \Cat_{\infty/\mathcal{C}}^{\txt{cocart}}
    \times_{\Cat_{\infty/\mathcal{C}_{F}}}
    \Cat_{\infty/\mathcal{C}_{F}}^{\txt{cart}},\]
  which is the \icat{} of
  cocartesian fibrations over $\mathcal{C}$ that have cartesian
  morphisms over $\mathcal{C}_{F}$, and whose morphisms are functors
  over $\mathcal{C}$ that preserve cocartesian
  morphisms as well as cartesian morphisms over
  $\mathcal{C}_{F}$.
\end{lemma}
\begin{proof}
  By \cite{HTT}*{Corollary 5.2.2.4} a cocartesian fibration to
  $\mathcal{C}$ has cartesian morphisms over $\mathcal{C}_{F}$ \IFF{}
  it has locally cartesian morphisms over $\mathcal{C}_{F}$, which is
  equivalent to the corresponding morphisms in $\CatI$ having right
  adjoints (\cf{} \cite{HTT}*{Definition 5.2.2.1}). Thus the
  unstraightening of a right $F$-coadjointable functor gives an object
  of $\Cat_{\infty/\mathcal{C}}^{\txt{cocart}+F\txt{-cart}}$.
  Moreover, by \cite{HA}*{Proposition 4.7.4.17} a morphism over
  $\mathcal{C}$ that preserves cocartesian morphisms and cartesian
  morphisms over $\mathcal{C}_{F}$ corresponds to a natural
  transformation whose naturality squares over $\mathcal{C}_{F}$ are
  right adjointable, which is precisely the requirement for
  coadjointable natural transformations.
\end{proof}

\begin{defn}
  For $c \in \mathcal{C}$, let $\mathcal{Y}_{c} \to \mathcal{C}$ denote the
  cocartesian fibration classified by the right $F$-coadjointable functor
  $\mathsf{Y}(c) \colon \mathcal{C} \rightarrow \CatI$ that
  corresponds via \cref{coadjyoneda} to the functor
  $\SPAN_{F}(\mathcal{C})^{\op} \to \CATI$
  represented by $i(c)$.
\end{defn}

The idea is now to construct a ``candidate'' for $\mathcal{Y}_{c}$ 
using a result from \cite{freepres} together with the following observation:
\begin{lemma}\label{freeFcart}
  If $p \colon \mathcal{E} \to \mathcal{C}$ is a functor such that
  $\mathcal{E}$ has $p$-cartesian morphisms over morphisms in
  $\mathcal{C}_{F}$, then the functor
  \[ \Fun_{/\mathcal{C}}^{\Fcart}(\mathcal{C}_{/x}^{F}, \mathcal{E})
    \to \Fun_{/\mathcal{C}}(\{x\}, \mathcal{E}) \simeq \mathcal{E}_{x}\]
  given by restriction along the inclusion
  \[ \{x\} \simeq \{\id_{x}\} \hookrightarrow
  \mathcal{C}_{/x}^{F},\] is an equivalence, where
$\mathcal{C}_{/x}^{F}$ denotes the full subcategory of
$\mathcal{C}_{/x}$ spanned by morphisms to $x$ in $\mathcal{C}_{F}$, and
  $\Fun_{/\mathcal{C}}^{\Fcart}(\mathcal{C}_{/x}^{F}, \mathcal{E})$
  is the full subcategory of
  $\Fun_{/\mathcal{C}}(\mathcal{C}_{/x}^{F}, \mathcal{E})$ spanned by
  functors that preserve cartesian morphisms over $\mathcal{C}_{F}$.
\end{lemma}
\begin{proof}
  We use the results on relative Kan extensions from \cite{HTT}*{\S
    4.3.2}. For every object $f \colon y \to x$ of
  $\mathcal{C}_{/x}^{F}$ the \icat{}
  \[\{x\}_{f/} := \{x\} \times_{\mathcal{C}^{F}_{/x}}
    (\mathcal{C}^{F}_{/x})_{f/} \simeq \Map_{\mathcal{C}^{F}_{/x}}(f,
    \id_{x}) \] is contractible, since $\id_{x}$ is a terminal object
  in $\mathcal{C}^{F}_{/x}$. Hence a morphism
  \[\{x\}_{f/}^{\triangleleft} \simeq \Delta^{1} \to \mathcal{E}\] is a
  $p$-limit \IFF{} it's a cartesian morphism by \cite{HTT}*{Example
    4.3.1.4}. It follows that a functor
  $\Phi \colon \mathcal{C}_{/x}^{F} \to \mathcal{E}$ over
  $\mathcal{C}$ is a $p$-right Kan extension from $\{x\}$ \IFF{} for every
  $f \colon y \to x$ in $\mathcal{C}_{F}$ it takes the unique morphism
  $f \to \id_{x}$ (which is cartesian over $f$) to a cartesian
  morphism in $\mathcal{E}$. By the 3-for-2 property of cartesian
  morphisms this is equivalent to $\Phi$ preserving cartesian
  morphisms over $\mathcal{C}_{F}$. Hence \cite{HTT}*{Proposition
    4.3.2.15} implies that, since $\mathcal{E}$ has $p$-cartesian
  morphisms over $\mathcal{C}_{F}$, the functor
  $\Fun_{/\mathcal{C}}(\mathcal{C}_{/x}^{F}, \mathcal{E}) \to
  \Fun_{/\mathcal{C}}(\{x\}, \mathcal{E})$ restricts to an equivalence
  from the full subcategory
  $\Fun_{/\mathcal{C}}^{\Fcart}(\mathcal{C}_{/x}^{F}, \mathcal{E})$.
\end{proof}

\begin{construction}
  Since $\mathcal{Y}_{c}$ corresponds to a 
  right $F$-coadjointable functor, it has
  cartesian morphisms over $\mathcal{C}_{F}$. Applying
 \cref{freeFcart} to $\mathcal{Y}_{c}$ and $x = c$, 
  we see that there is a unique commutative square
  \[
    \begin{tikzcd}
      \mathcal{C}_{/c}^{F} \arrow{rr}{\alpha^{-}_{c}} \arrow{dr} & &  \mathcal{Y}_{c}
      \arrow{dl} \\
      & \mathcal{C}
    \end{tikzcd}
  \]
  such that the top horizontal functor preserves cartesian morphisms
  over $\mathcal{C}_{F}$ and takes $\id_{c}$ in $\mathcal{C}_{/c}^{F}$
  to the identity morphism $\id_{i(c)}$ in
  $\SPAN_{F}(\mathcal{C})(i(c),i(c)) \simeq \mathcal{Y}_{c,c}$.

  Now since $\mathcal{Y}_{c} \to \mathcal{C}$ is also a cocartesian fibration, we
  can extend this to a unique functor from the free cocartesian
  fibration \cite[Theorem 4.5]{freepres}
  \begin{equation} \label{eq:freecocart}
\mathcal{B}_{c}:=  \mathcal{C}^{F}_{/c} \times_{\mathcal{C}} \mathcal{C}^{[1]}
     \to \mathcal{C},
    \end{equation}
  giving a unique commutative triangle

\[    \begin{tikzcd}
      \mathcal{B}_{c} \arrow{rr}{\alpha_{c}} \arrow{dr} & &
      \mathcal{Y}_{c} \arrow{dl} \\
      {} & \mathcal{C}
    \end{tikzcd}
  \]
  where the horizontal functor preserves cocartesian morphisms and
  restricts to $\alpha_{c}^{-}$ on $\mathcal{C}_{/c}^{F}$.
\end{construction}

\begin{remark}\label{rem:bc}
  To understand the \itcat{} $\SPAN_{F}(\mathcal{C})$, we are
  going to show that the functor $\alpha_{c}$ is an equivalence. The
  explicit construction of the $\infty$-category $\mathcal{B}_{c}$ in \eqref{eq:freecocart}
  allows us to unpack it easily, revealing the expected definition of
  spans. We carry this out:
  \begin{itemize}
  \item   An object of $\mathcal{B}_{c}$ consists of an object $x \xto{f} c$ in
  $\mathcal{C}^{F}_{/c}$, \ie{} a morphism $f$ to $c$ in
  $\mathcal{C}_{F}$, together with a morphism from $x$ in
  $\mathcal{C}$; in other words, it is precisely a \emph{span}
  \[ y \xfrom{g} x \xto{f} c \]
  with $f$ in $\mathcal{C}_{F}$.
  The functor~\eqref{eq:freecocart} to $\mathcal{C}$ takes this
  to the object $y$. 
\item  A morphism from this object to another object
  \[ y \xfrom{g'} x' \xto{f'} c\]
  in the fibre $\mathcal{B}_{y,c}$ consists of a morphism
  \[
    \begin{tikzcd}
      x \arrow{dr}{f} \arrow{dd}{\xi} \\
      & c \\
      x' \arrow{ur}[swap]{f'}
    \end{tikzcd}
  \]
  in $\mathcal{C}_{/c}$ and a commutative triangle
  \[
    \begin{tikzcd}
      x \arrow{dr}{g} \arrow{dd}{\xi} \\
      & y \\
      x' \arrow{ur}[swap]{g'}
    \end{tikzcd}
    \]
    in $\mathcal{C}$, \ie{} precisely a morphism of spans
    \[
      \begin{tikzcd}
        {} & x \arrow{dl}[swap]{g} \arrow{dr}{f} \arrow{dd}{\xi} \\
        y & & c \\
        & x'. \arrow{ul}{g'} \arrow{ur}{f'}
      \end{tikzcd}
      \]
\item  Moreover, the cocartesian morphism
  over $y \xto{\eta} y'$ in $\mathcal{C}$ is given by composition in
  $\mathcal{C}^{[1]}$ and so takes the span
  \[ y \xfrom{g} x \xto{f} c \] to
  \[ y' \xfrom{\eta g} x \xto{f} c.\]
  \end{itemize}
\end{remark}

To prove that $\alpha_{c}$ is an equivalence we want to use the
universal property of $\mathcal{Y}_{c}$ (\ie{} the Yoneda lemma) to produce a functor $\beta_{c}
\colon \mathcal{Y}_{c} \to \mathcal{B}_{c}$, which will be the inverse
of $\alpha_{c}$. This requires knowing the
following:
\begin{propn}\label{Bccoadj}
  The functor $B_{c} \colon \mathcal{C} \to \CatI$
  classifying the cocartesian fibration $\mathcal{B}_{c} \to \mathcal{C}$
  is right $F$-coadjointable.
\end{propn}

\begin{proof}%
  For $f \colon x \to y$ in $\mathcal{C}$, let us denote the value of
  $B_{c}$ at $f$ by
  \[
  f_{\ostar} \colon \mathcal{B}_{x,c} \to \mathcal{B}_{y,c};
  \]
  here $\mathcal{B}_{x,c}$ can be identified with the fibre product
  $\mathcal{C}_{/x} \times_{\mathcal{C}} \mathcal{C}^{F}_{/c}$, and
  the functor $f_{\ostar}$ is given by composing with $f$ in the first
  factor, \ie{}
  \[
    x \xfrom{g} z \to c
    \quad\mapsto\quad y \xfrom{f \circ g} z \to c.
  \]
  We first prove that the functor $\mathcal{B}_{c} \to \mathcal{C}$ is
  cartesian over $\mathcal{C}_{F}$. In other words, we must show that
  for every morphism $\phi \colon y' \to y$ in $\mathcal{C}_{F}$, the
  functor
  \[ \phi_{\ostar} \colon \mathcal{B}_{y',c} \to \mathcal{B}_{y,c}, \]
  has a right adjoint, which we will denote by $\phi^{\oplus}$.  To
  see that $\phi_{\ostar}$ has a right adjoint it suffices (by a
  reformulation of \cite[Lemma 5.2.4.1]{HTT}) to show that for any
  span $\sigma = (y \xfrom{g} x \xto{f} c)$ ($f$ in
  $\mathcal{C}_{F}$), the \icat{}
  $\mathcal{B}_{y',c/\sigma} := \mathcal{B}_{y',c}
  \times_{\mathcal{B}_{y,c}} \mathcal{B}_{y,c/\sigma}$ has a terminal
  object.

  An object of $\mathcal{B}_{y',c/\sigma}$ is
a commutative diagram 
\begin{equation} \label{eq:ob-y'c}
  \begin{tikzcd}
    y' \arrow{dd}{\phi} & x' \arrow{l}{g'} \arrow{dd} \arrow{dr}{f'} \\
    & & c \\
    y & x. \arrow{l}{g} \arrow{ur}[swap]{f}
  \end{tikzcd}
\end{equation}
More formally, we can identify $\mathcal{B}_{y,c/\sigma}$ with the
  full subcategory $\mathcal{C}_{/x} \times_{\mathcal{C}_{/c}}
  \mathcal{C}^{F}_{/c}$ of $\mathcal{C}_{/x}$ spanned by morphisms $x'
  \to x$ such that $x' \to x \to c$ is in $\mathcal{C}_{F}$. The fibre
  product $\mathcal{B}_{y',c/\sigma}$ we can then identify with the
  full subcategory of $\mathcal{C}_{/p}$, where $p$ is the diagram $y'
  \to y \from x$, spanned by commutative squares
  \[
    \begin{tikzcd}
      x' \arrow{r} \arrow{d} & x \arrow{d} \\
      y' \arrow{r} & y
    \end{tikzcd}
    \]
    such that the composite $x' \to x \to c$ lies in
    $\mathcal{C}_{F}$. A terminal object in $\mathcal{C}_{/p}$ is
    precisely a fibre product $x \times_{y} y'$, which exists since 
    by assumption $\mathcal{C}$ admits all pullback along
    $\phi$. Moreover, this terminal object lies in the full
    subcategory $\mathcal{B}_{y',c/\sigma}$ since the projection $x
    \times_{y} y' \to x$ is a base change of $\phi$ and so lies in
    $\mathcal{C}_{F}$.

    To complete the proof we must show that given a pullback square
    \begin{equation}
      \label{eq:Bcpullback}
    \csquare{\tilde{y}'}{\tilde{y}}{y'}{y}{\tilde{\phi}}{\tilde{\gamma}}{\gamma}{\phi}      
    \end{equation}
  in $\mathcal{C}$ with $\phi$ in $\mathcal{C}_{F}$, the Beck-Chevalley transformation
  \[  \tilde{\gamma}_{\ostar}\tilde{\phi}^{\oplus}  \to \phi^{\oplus}\gamma_{\ostar} \]
  is an equivalence. Evaluating at a span $\tilde{y} \from x \to c$
  this transformation is given by the canonical dashed map in the diagram
  \[
    \begin{tikzcd}
      y' \arrow{dd}{\phi} & \tilde{y}' \arrow{dd}{\tilde{\phi}} \arrow{l}{\tilde{\gamma}} & \tilde{x}
      \arrow{dd} \arrow{l} \arrow{dr} \\
      & & & c \\
      y & \tilde{y} \arrow{l}{\gamma} &  x \arrow{l} \arrow{ur}
    \end{tikzcd}
  \]
    \[
    \begin{tikzcd}
      y' \arrow{dd}{\phi} & \tilde{y}' \arrow{dd}{\tilde{\phi}}
      \arrow{l}[swap]{\tilde{\gamma}} & & \bullet \arrow[dashed]{dl}
      \arrow[color=blue]{dd} \arrow[color=blue]{ll} \arrow{dr} \\
      & & \bullet \arrow[color=red]{dr} \arrow[color=red,crossing
      over]{ull} \arrow[crossing over]{rr} & & c \\
      y & \tilde{y} \arrow{l}{\gamma} & & x \arrow{ll} \arrow{ur}
    \end{tikzcd}
  \]
  where the two coloured squares are cartesian. Since the square
  \cref{eq:Bcpullback} is by assumption cartesian, it follows from the
  pasting lemma for pullback squares that this is indeed an equivalence.
\end{proof}

Translating \cref{coadjyoneda} through the equivalence of
\cref{lem:Fcoadjstr}, we see that \cref{Bccoadj} implies that any object
$X \in \mathcal{B}_{c}$ over $c' \in \mathcal{C}$
corresponds to a morphism $\mathcal{Y}_{c'} \to \mathcal{B}_{c}$. In
particular, we have:
\begin{cor}
  There is a canonical functor $\beta_{c} \colon \mathcal{Y}_{c} \to
  \mathcal{B}_{c}$ over $\mathcal{C}$ corresponding
  to the identity span of $c$; this preserves cocartesian morphisms
  and cartesian morphisms over $\mathcal{C}_{F}$. \qed
\end{cor}

We now need to prove that the functor $\alpha_{c}$ has the same
property:
\begin{propn}\label{alphaccart}
  The functor $\alpha_c \colon  \mathcal{B}_{c} \rightarrow
  \mathcal{Y}_{c}$ over $\mathcal{C}$ preserves cocartesian morphisms
  and cartesian
  morphisms over $\mathcal{C}_{F}$.
\end{propn}

\begin{remark}
  For the proof we first need to discuss the naturality of
  Beck--Chevalley transformations in the following situation:
  Suppose we have a commutative triangle of \icats{}
  \[
    \begin{tikzcd}
      \mathcal{E} \arrow{rr}{F} \arrow{dr}[swap]{p} & & \mathcal{F}
      \arrow{dl}{q} \\
      & \mathcal{B},
    \end{tikzcd}
  \]
  where $p$ and $q$ have both cartesian and cocartesian morphisms
  over $f \colon a \to b$ in $\mathcal{B}$, but $F$ does not
  necessarily preserve these. We write $f_{!}$ for the cocartesian
  pushforward and $f^{*}$ for the cartesian pullback along $f$ for
  both $p$ and $q$ (so the functor $f_{!}$ is left adjoint to
  $f^{*}$). Then we can make the following commutative diagrams for $x
  \in \mathcal{E}_{a}$, $y \in \mathcal{E}_{b}$:
\begin{equation}\label{laxunitcounitdiag}
    \begin{tikzcd}
      Fx \arrow{r} \arrow{d} & Ff^{*}f_{!}x \arrow{d} \\
      f^{*}f_{!}Fx \arrow{r} \arrow{d} & f^{*}Ff_{!}x \arrow{d} \\
      f_{!}Fx \arrow{r} & Ff_{!}x,
    \end{tikzcd}
    \qquad
    \begin{tikzcd}
      Ff^{*}y \arrow{r} \arrow{d} & f^{*}Fy \arrow{d} \\
      f_{!}Ff^{*}y \arrow{r} \arrow{d} & f_{!}f^{*}Fy \arrow{d} \\
      Ff_{!}f^{*}y \arrow{r} & Fy.
    \end{tikzcd}    
  \end{equation}
  In particular, the top left and bottom right squares here encode the
  compatibility of $F$ with the units and counits of the two
  adjunctions $f_{!} \dashv f^{*}$. Now suppose we have a commutative
  square
  \[
    \begin{tikzcd}
      a' \arrow{r}{f'} \arrow{d}{g'} & b' \arrow{d}{g} \\
      a \arrow{r}{f} & b,
    \end{tikzcd}
  \]
  where $p$ and $q$ have cocartesian morphisms over $f,f'$ and both
  cartesian and cocartesian morphisms over $g,g'$. Then we claim that
  the two Beck--Chevalley transformations $f'_{!}g'^{*} \to
  g^{*}f_{!}$, intertwined by $F$, are related by a commutative diagram
  \begin{equation}
    \label{eq:BClaxdiag}
    \begin{tikzcd}
      {} & f'_{!}F(g'^{*}X) \arrow{dl} \arrow{dr} \\
      f'_{!}g'^{*}F(X) \arrow{d} & & F(f'_{!}g'^{*}X) \arrow{d}\\
      g^{*}f_{!}FX \arrow{dr} & & F(g^{*}f_{!}X) \arrow{dl} \\
      & g^{*}F(f_{!}X).      
    \end{tikzcd}
  \end{equation}
  This can be extracted from the following diagram, where we have used
  \cref{laxunitcounitdiag} together with naturality:
  \[
    \begin{tikzcd}[column sep=small]
      {} & {} &  f'_{!}Fg'^{*}X \arrow{dl} \arrow{dd} \arrow{dr}
      \\
      {} &  f'_{!}g'^{*}FX \arrow{d}  & & Ff'_{!}g'^{*}X \arrow{dd} \arrow{ddr} 
       \\
      {} & g^{*}g_{!}f'_{!}g'^{*}FX \arrow{dl}{\sim} &
      g^{*}g_{!}f'_{!}Fg'^{*}X \arrow{l} \arrow{dl}{\sim} \arrow{dr} \\
      g^{*}f_{!}g'_{!}g'^{*}FX \arrow{d} &  g^{*}f_{!}g'_{!}Fg'^{*}X
      \arrow{l} \arrow{d} & & 
      g^{*}g_{!}Ff'_{!}g'^{*}X \arrow{d} & F g^{*}g_{!}f'_{!}g'^{*}X
      \arrow{dl} \arrow{d}{\sim} \\
      g^{*}f_{!}FX \arrow[bend right]{ddrr} & 
      g^{*}f_{!}Fg'_{!}g'^{*}X \arrow{dr} \arrow{l} & &
      g^{*}Fg_{!}f'_{!}g'^{*}X \arrow{dl}[swap]{\sim} &
      Fg^{*}f_{!}g'_{!}g'^{*}X \arrow{dll} \arrow{d} \\
      & &  g^{*}Ff_{!}g'_{!}g'^{*}X \arrow{d} & & Fg^{*}f_{!}X \arrow{dll} \\
      & &  g^{*}Ff_{!}X.      
    \end{tikzcd}
    \]
\end{remark}

\begin{proof}[Proof of \cref{alphaccart}]
  The universal property we used to define $\alpha_{c}$ implies that
  it preserves cocartesian morphisms. Moreover, since $\alpha_{c}$ was
  extended from a
  functor $\mathcal{C}^{F}_{/c} \to \mathcal{Y}_{c}$ that preserved
  cartesian morphisms over $\mathcal{C}_{F}$, we know $\alpha_{c}$
  preserves cartesian morphisms in the image of
  $\mathcal{C}_{/c}^{F}$. In other words, for $f \colon x \to c$ in
  $\mathcal{C}_{F}$, the map $\alpha_{c}(\phi^{\oplus}[h]_{F}) \to
  \phi^{\oplus}\alpha_{c}([h]_{F})$ is an equivalence for all $\phi \colon
  x' \to x$ in $\mathcal{C}_{F}$.

  More generally, for a span $\sigma = (y \xfrom{g} x \xto{f} c)$, we
  need to show that $\alpha_c(\phi^{\oplus}\sigma) \simeq
  \phi^{\oplus}\alpha_c(\sigma)$ for any morphism $\phi \colon y' \to
  y$ in $\mathcal{C}_{F}$. To proceed, let us view
  $\sigma$ as $g_{\ostar}[f]_{F}$.

  Forming the pullback square
  \[
    \begin{tikzcd}
     x' \arrow{d}{g'} \arrow{r}{\xi} & x \arrow{d}{g}\\
      y' \arrow{r}{\phi} & y,
    \end{tikzcd}
  \]
  the Beck-Chevalley transformation yields an equivalence:
  \[
  g'_{\ostar}\xi^{\oplus}[f]_{F} \isoto
  \phi^{\oplus}g_{\ostar}[f]_{F}.
  \] Moreover, from
  \cref{eq:BClaxdiag} we get a natural commutative diagram
  \[
    \begin{tikzcd}
      {} & g'_{\ostar}\alpha_{c}(\xi^{\oplus}[f]_{F}) \arrow{dl}[swap]{(1)} \arrow{dr}{(2)} \\
      g'_{\ostar}\xi^{\oplus}\alpha_{c}([f]_{F}) \arrow{d}[swap]{(3)} & & \alpha_{c}(g'_{\ostar}\xi^{\oplus}[f]_{F}) \arrow{d}{(3)}\\
      \phi^{\oplus}g_{\ostar}\alpha_{c}([f]_{F}) \arrow{dr}[swap]{(2)} & & \alpha_{c}(\phi^{\oplus}g_{\ostar}[f]_{F}) \arrow{dl} \\
      & \phi^{\oplus}\alpha_{c}(g_{\ostar}[f]_{F}),
    \end{tikzcd}
    \]
    where the map labelled (1) is an equivalence since $\alpha_{c}$
    preserves cartesian morphisms from $\mathcal{C}^{F}_{/c}$, those
    labelled (2) are equivalences since $\alpha_{c}$ preserves
    cocartesian morphisms, and those labelled (3) are equivalences
    because the Beck--Chevalley transformations are invertible. 
    Hence the last morphism in the diagram is also an equivalence,
    which shows that $\alpha_{c}$ preserves the cartesian morphism
    $\phi^{\oplus}g_{\ostar}[f]_{F} \to g_{\ostar}[f]_{F}$.
\end{proof}

\begin{cor} \label{cor:bc}
 The functors $\beta_{c} \colon \mathcal{Y}_{c} \to \mathcal{B}_{c}$
 and $\alpha_{c} \colon \mathcal{B}_{c} \to \mathcal{Y}_{c}$ satisfy
 \[ \beta_{c}\alpha_{c} \simeq \id_{\mathcal{B}_{c}}, \qquad
  \alpha_{c}\beta_{c} \simeq \id_{\mathcal{Y}_{c}}.\] Thus
  $\alpha_{c}$ is an equivalence with inverse $\beta_{c}$.
\end{cor}
\begin{proof}
  By construction $\alpha_{c}$ takes the identity span of $c$ to
  \[\id_{c} \, \in \,
  \mathcal{Y}_{c,c} \simeq
  \MAP_{\SPAN_{F,L}(\mathcal{C})}(i(c),i(c)).\] The composite
  $\beta_{c}\alpha_{c}$ is a functor
  $\mathcal{B}_{c} \to \mathcal{B}_{c}$ that preserves cocartesian
  morphisms, hence it is determined by its restriction to
  $\mathcal{C}^{F}_{/c} \to \mathcal{B}_{c}$. This restriction preserves cartesian
  morphisms over $\mathcal{C}_{F}$ and so by \cref{freeFcart} it is
  determined by its value
  at $\id_{c}$, which is the identity span in $\mathcal{B}_{c}$. The
  same holds for the identity of $\mathcal{B}_{c}$ and so
  $\id_{\mathcal{B}_{c}} \simeq \beta_{c}\alpha_{c}$. Conversely,
  $\alpha_{c}\beta_{c}$ is a functor
  $\mathcal{Y}_{c} \to \mathcal{Y}_{c}$ that preserves cocartesian
  morphisms and cartesian morphisms over $\mathcal{C}_{F}$ by
  \cref{alphaccart}. By \cref{coadjyoneda}, interpreted in terms of
  fibrations, this functor is determined by where it sends the
  identity of $c$; since we know this is taken to itself, this functor
  must be the identity $\id_{\mathcal{Y}_{c}}$.
\end{proof}

The equivalence $\mathcal{Y}_{c}\simeq \mathcal{B}_{c}$ allows us to
identify morphisms in $\SPAN_{F}(\mathcal{C})$ with
spans, and 2-morphisms with morphisms of spans. We now check that
composition of spans works as expected:
\begin{propn}
  Composition of spans in $\SPAN_{F}(\mathcal{C})$ is given by taking
  pullbacks.
\end{propn}

\begin{proof}
  By construction, the cocartesian morphisms in $\mathcal{Y}_{c}$
  encode precomposition with the images of morphisms in
  $\mathcal{C}^{\op}$ under the functor $i$: Given a morphism $g
  \colon x \to y$ in $\mathcal{C}$ we have
  \[  \sigma \circ i(g) \simeq g_{\ostar}\sigma\]
  for any span $\sigma$. In particular, from our description of the
  right-hand side we have
  \[ i(g) \simeq  \id_{x} \circ i(g) \simeq g_{\ostar}(\id_{x})
    \simeq (y \xfrom{g} x \longequal x) \simeq [g]_{B},
  \]
  and more generally
  \[ x \xfrom{g} y \xto{f} z \simeq
    g_{\ostar}[f]_{F} \simeq [f]_{F} \circ [g]_{B}.\]
  This means that to describe an arbitrary composition in
  $\SPAN_{F}(\mathcal{C})$ it suffices (by associativity of
  composition) to understand compositions of the form $[g]_{B} \circ
  [f]_{F}$. Note that $[f]_{B}$ is in the image of
  $\mathcal{C}_{F}^{\op}$ under $i$, and therefore admits a left
  adjoint in $\SPAN_{F}(\mathcal{C})$ since $i$ is left $F$-adjointable; let us denote this by
  $[f]_{B}^{\ell}$. We claim that
  \[ [f]_{B}^{\ell} \simeq [f]_{F}.\]
  To see this, we note that precomposition with $[f]_{B}$ is the
  functor $f_{\ostar}$, which admits a right adjoint $f^{\oplus}$. Therefore, by uniqueness of adjoints we
  conclude that precomposition with $[f]^{\ell}_{B}$ must coincide
  with $f^{\oplus}$. Therefore the span $[f]^{\ell}_{B}$ is computed as
  \[
 [f]^{\ell}_{B} \simeq f^{\oplus}(\id) \simeq [f]_{F}.
  \]
  Furthermore,
\[ [g]_{B} \circ
  [f]_{F} \simeq [g]_{B} \circ [f]_{B}^{\ell} \simeq
  f^{\oplus}[g]_{B}, \]\
which we saw above is computed by taking the pullback of $g$ along
$f$, as required.
\end{proof}

\begin{propn}\label{cor:invblebispan}
  A span $\sigma = (y \xfrom{g} x \xto{f} z)$ is invertible as
  a morphism in $\SPAN_{F}(\mathcal{C})$ \IFF{} the components
  $f$ and $g$ are both invertible in $\mathcal{C}$.
\end{propn}
\begin{proof}
  Since we now know composition is given by taking pullbacks, this
  follows as in the proof of \cite{spans}*{Lemma 8.2}.
\end{proof}

\begin{cor}\label{cor:ieqgpd}
  The functor $i \colon \mathcal{C}^{\op} \to
  \SPAN_{F}(\mathcal{C})$ gives an equivalence on underlying
  \igpds{}
  \[ \mathcal{C}^{\op,\simeq} 
    \isoto \SPAN_{F}(\mathcal{C})^{\simeq}.\]
\end{cor}
\begin{proof}
  We know the functor $i$ is essentially surjective on objects, so it
  is enough to show that for any objects $x,y$ the map
  \[ \Map_{\mathcal{C}}(x,y)^{\txt{eq}} \to
    \Map_{\Span_{F}(\mathcal{C})}(ix,iy)^{\txt{eq}} \]
  is an equivalence,
  where we are taking the components of the mapping spaces that
  correspond to equivalences. This is immediate from
  \cref{cor:invblebispan} and our description of the mapping spaces in
  $\Span_{F}(\mathcal{C})$.  
\end{proof}

Combining the results of this section, we have shown:
\begin{thm}\label{thm:SPANdesc}
  Let $(\mathcal{C}, \mathcal{C}_{F})$ be a span
  triple. Then left $F$-adjointable functors out of
  $\mathcal{C}^{\op}$ are corepresented by an \itcat{}
  $\SPAN_{F}(\mathcal{C})$ via a left $F$-adjointable functor $i
  \colon \mathcal{C}^{\op} \to
  \SPAN_{F}(\mathcal{C})$ with the following properties:
  \begin{enumerate}[(i)]
  \item $i$ gives an equivalence
    \[ \SPAN_{F}(\mathcal{C})^{\simeq} \simeq
      \mathcal{C}^{\simeq}\]
     on underlying $\infty$-groupoids,
   \item morphisms from $i(x)$ to $i(y)$ can be identified with  spans
    \[ y \xfrom{g} x \xto{f} z \]
    where $f$ is in $\mathcal{C}_{F}$,
  \item 2-morphisms correspond to morphisms of spans diagrams,
  \item and composition of morphisms is given by taking pullbacks. \qed
  \end{enumerate}
\end{thm}

We end this section by deducing a description of the functor of
\itcats{} corresponding to a left $F$-adjointable functor:
\begin{propn}\label{propn:adjftrspandesc}
  For a left $F$-adjointable functor $\phi \colon \mathcal{C}^{\op}
  \to \mathcal{X}$, the corresponding functor $\Phi \colon
  \SPAN_{F}(\mathcal{C}) \to \mathcal{X}$ can be described as
  follows:
  \begin{enumerate}[(1)]
  \item On objects, $\Phi(c) \simeq \phi(c)$ for $c \in \mathcal{C}$.
  \item On morphisms, $\Phi$ takes a span
    \[ \sigma = (y \xfrom{g} x \xto{f} z) \]
    to the composite $f_{\oplus}g^{\ostar} \colon \phi(x)
    \to \phi(y)$, where $f_{\oplus}$ is the left adjoint to
    $f^{\ostar} := \phi(f)$.
  \item On 2-morphisms, $\Phi$ takes the 2-morphism $\beta \colon \sigma
    \to \sigma'$ given by the commutative diagram
    \[
      \begin{tikzcd}
        {} & x \arrow{dl}[swap]{g} \arrow{dr}{f} \arrow{dd}{h} \\
        y & & z \\
         & x' \arrow{ul}{g'} \arrow{ur}[swap]{f'}
      \end{tikzcd}
    \]
    to the composite
    \begin{equation}
      \label{eq:span2morimage}
     f_{\oplus}g^{\ostar} \simeq f_{\oplus}h^{\ostar}g'^{\ostar}
      \to f_{\oplus}h^{\ostar}f'^{\ostar}f'_{\oplus}g'^{\ostar}
      \simeq f_{\oplus}f^{\ostar}f'_{\oplus}g'^{\ostar} \to
      f'_{\oplus}g'^{\ostar},      
    \end{equation}
 where the first noninvertible arrow is
    an adjunction unit and the second noninvertible arrow is a counit.
  \end{enumerate}
\end{propn}
\begin{proof}
  We know that $\Phi \circ i \simeq \phi$ and that $i$ is an
  equivalence on underlying \igpds{} by \cref{cor:ieqgpd}, which gives
  (1). To prove (2), observe that the bispan $\sigma$ is the composite
  $[f]_{F} \circ [g]_{B}$ in
  $\SPAN_{F}(\mathcal{C})$. Here $[g]_{B}$ is $i(g)$, and so
  \[ g^{\ostar} := \Phi([g]_{B}) \simeq \Phi(i(g)) \simeq \phi(g).\]
  Moreover, the span $[f]_{F}$ is left adjoint to $[f]_{B}$, hence its
  image $\Phi([f]_{F})$ is the left adjoint $f_{\oplus}$ to
  $f^{\ostar}$. In other words, we have
  $\Phi(\sigma) \simeq \Phi([f]_{F}) \circ \Phi([g]_{B}) \simeq
  f_{\oplus}g^{\ostar}$.  To prove (3), first observe that the
  2-morphism $\beta$ is the composite (``whiskering'') of the morphism
  $[g']_{B}$ with the 2-morphism
  $\lambda$ given by
  \[
    \begin{tikzcd}
      {} & x \arrow{dl}[swap]{h} \arrow{dd}{h} \arrow{dr}{f} \\
      x' & & z \\
      & x' \arrow[equals]{ul} \arrow{ur}[swap]{f'}
    \end{tikzcd}
  \]
  and this whiskering corresponds to the first equivalence in \cref{eq:span2morimage}.
  
  It thus suffices to show that $\Phi$ takes $\lambda$ to the composite
    \[ f_{\oplus}h^{\ostar} \to
      f_{\oplus}h^{\ostar}f'^{\ostar}f'_{\oplus} \simeq
      f_{\oplus}f^{\ostar}f'_{\oplus} \to f'_{\oplus} \]
    using the unit for $f'_{\oplus} \dashv f'^{\ostar}$ and the counit
    for $f_{\oplus} \dashv f^{\ostar}$. To show this we will check
    that the morphism $\lambda$ has the corresponding decomposition in
    $\SPAN_{F}(\mathcal{C})$. Indeed, we can decompose $\lambda$
    as the composite
  \[
    \begin{tikzcd}
      {} & x \arrow{dl}[swap]{h} \arrow{d}{\phi} \arrow{dr}{f} \\
      x' & x' \times_{z}x \arrow{l}[swap]{\pi'} \arrow{d}{\pi'}
      \arrow{r}{f\pi} & z \\
      & x', \arrow[equals]{ul} \arrow{ur}[swap]{f'}
    \end{tikzcd}
  \]
  where $\pi,\pi'$ are the projections from $x' \times_{z} x$ to $x$
  and $x'$, respectively, and $\phi$ is the unique morphism such that
  $\pi \phi \simeq \id_{x}$, $\pi'\phi \simeq h$. Now unpacking the
  description of
  units and counits in $\SPAN_{F}(\mathcal{C})$ implies that the top
  morphism in this decomposition is the composite of the unit for
  $[f']_{F} \dashv [f']_{B}$ with the morphism $x' \xfrom{h} x \xto{f} z$
  and the bottom is the composite of $[f']_{F}$ with the counit
  for $[f]_{F} \dashv [f]_{B}$.
\end{proof}

\subsection{Distributive functors}
We now start our discussion of distributivity.  In this section we
introduce the notion of a distributive functor, which we will prove in
the next subsection is corepresented by \itcats{} of bispans.  For the
definition we first need to introduce the notion of a distributivity
diagram, which dictates how the multiplicative and additive
pushforwards for a distributive functor should interact:
\begin{defn} \label{def:dis-square}
  Let $x \xto{l} y \xto{f} z$ be morphisms in an \icat{}
  $\mathcal{C}$. A \emph{distributivity diagram} for $l$ and $f$ is a
  commutative diagram
  \begin{equation} \label{eq:dis-sq}
    \begin{tikzcd}
      {} & w \times_{z} y \ddrpullback \arrow{dd}{\tilde{g}} \arrow{dl}[swap]{\epsilon} \arrow{r}{\tilde{f}} & w
      \arrow{dd}{g} \\
      x \arrow{dr}{l} \\
      {} & y \arrow{r}{f} & z,
    \end{tikzcd}
  \end{equation}
  where the square is cartesian, with the property that for any morphism
  $\phi \colon u
  \to z$, the composite map
  \begin{equation} \label{eq:uni-prop-dis} \Map_{/z}(\phi,g) \to
    \Map_{/y}(f^{*}\phi, \tilde{g}) \xto{\epsilon_*} \Map_{/y}(f^{*}\phi, l)
    \end{equation}
    is an equivalence. The distributivity diagram for $l$ and $f$ is
    necessarily unique if it exists.
\end{defn}

\begin{remark}\label{rmk:distdiagterminal}
  Consider the \icat{} of diagrams of shape \cref{eq:dis-sq} (with the
  square cartesian). If all pullbacks along $f$ exist in
  $\mathcal{C}$, then this is equivalent an object of the fibre product of \icats{} $$\mathcal{C}_{/z}
  \times_{\mathcal{C}_{/y}} \mathcal{C}_{/x},$$ with the functors in
  the pullback being $f^{*} \colon \mathcal{C}_{/z} \to
  \mathcal{C}_{/y}$ and $l_{!} \colon \mathcal{C}_{/x} \to \mathcal{C}_{/y}$.
  The universal property of the distributivity diagram can then be
  reformulated as that of being a terminal object in this \icat{}.
\end{remark}

\begin{defn} \label{def:bi-quad}
  A \emph{bispan triple} $(\mathcal{C}, \mathcal{C}_{F},
  \mathcal{C}_{L})$ consists of an \icat{}
  $\mathcal{C}$ together with two subcategories $\mathcal{C}_{F},
  \mathcal{C}_{L}$ such that the following
  assumptions hold:
  \begin{enumerate}[(a)]
  \item\label{item:Fspanpair} $(\mathcal{C}, \mathcal{C}_{F})$ is a span
    pair.
  \item\label{item:Lspanpair} $(\mathcal{C}, \mathcal{C}_{L})$ is a span
    pair.
  \item For $l \colon x \to y$ in $\mathcal{C}_{L}$ and $f \colon y
    \to z$ in $\mathcal{C}_{F}$ there exists a distributivity diagram
    \cref{eq:dis-sq} where $g$
    is in $\mathcal{C}_{L}$ (and hence $\tilde{f}$ is in
    $\mathcal{C}_{F}$ by \ref{item:Fspanpair} and $\tilde{g}$ is in
    $\mathcal{C}_{L}$ by \ref{item:Lspanpair}).
  \end{enumerate}
\end{defn}

\begin{notation}
  For $x \in \mathcal{C}$ we write $\mathcal{C}_{/x}^{L}$ for the full
  subcategory of $\mathcal{C}_{/x}$ spanned by morphisms $y \to x$ in $\mathcal{C}_{L}$.
\end{notation}

\begin{remark}\label{dist-adj}
  For a fixed $f\colon y \to z$ in $\mathcal{C}_{F}$, if the
  distributivity diagram \cref{eq:dis-sq} exists for all $l$ in
  $\mathcal{C}_{L}$ then the functor
  $f^{*} \colon \mathcal{C}^{L}_{/z} \to \mathcal{C}^{L}_{/y}$ given
  by pullback along $f$ has a right adjoint $f_{*}$. Indeed, from
  \cref{eq:uni-prop-dis} we see that we have
  \[(w \xto{g} z) \simeq f_*(x \xto{l} y)\] which determines the rest of
  the diagram. Note, however, that if the \icat{} $\mathcal{C}_{L}$
  is not all of $\mathcal{C}$ then the property of
  \cref{eq:uni-prop-dis} is slightly stronger than the existence of
  the right adjoint: for this to exist it suffices to consider maps
  from $\phi$ in $\mathcal{C}^{L}_{/z}$, while \cref{eq:uni-prop-dis}
  asks for an equivalence on maps from any $\phi$ in
  $\mathcal{C}_{/z}$. We can characterize the additional assumption on
  these right adjoints needed to have a bispan triple in terms of a
  base change property:
\end{remark}

\begin{lemma}\label{lem:bispantripmatecond}
  Suppose we have a triple $(\mathcal{C}, \mathcal{C}_{F},
  \mathcal{C}_{L})$ consisting of an \icat{}
  $\mathcal{C}$ together with two subcategories $\mathcal{C}_{F},
  \mathcal{C}_{L}$ such that
  \begin{enumerate}[(a)]
  \item $(\mathcal{C}, \mathcal{C}_{F})$ is a span
    pair,
  \item $(\mathcal{C}, \mathcal{C}_{L})$ is a span
    pair,
  \item for any $f \colon x \to y$ in $\mathcal{C}_{F}$ the functor
    $f^{*} \colon \mathcal{C}^{L}_{/y} \to \mathcal{C}^{L}_{/x}$ given
    by pullback along $f$ has a right adjoint $f_{*}$.
  \end{enumerate}
  Then $(\mathcal{C}, \mathcal{C}_{F},
  \mathcal{C}_{L})$ is a bispan triple \IFF{} for every cartesian
  square
  \[
    \begin{tikzcd}
      x' \arrow{r}{f'} \arrow{d}[swap]{\xi} & y' \arrow{d}{\eta} \\
      x \arrow{r}{f} & y
    \end{tikzcd}
  \]
  with $f$ in $\mathcal{C}_{F}$ the commutative square
  \[
    \begin{tikzcd}
      \mathcal{C}^{L}_{/y} \arrow{r}{f^{*}} \arrow{d}[swap]{\eta^{*}} &
      \mathcal{C}^{L}_{/x} \arrow{d}{\xi^{*}} \\
      \mathcal{C}^{L}_{/y'} \arrow{r}{f'^{*}} & \mathcal{C}^{L}_{/x'}
    \end{tikzcd}
  \]
  is right adjointable, \ie{} the mate transformation
  \[ \eta^{*}f_{*} \to f'_{*}\xi^{*} \]
  is invertible.
\end{lemma}
\begin{proof}
  First suppose we have a bispan triple. Then for $l \in
  \mathcal{C}^{L}_{/x}$ and $l' \in \mathcal{C}^{L}_{/y'}$ we have
  natural equivalences
  \[
    \begin{split}
      \Map_{/y'}(l', f'_{*}\xi^{*}l) & \simeq \Map_{/x'}(f'^{*}l',
      \xi^{*}l) \simeq \Map_{/x}(\xi_{!}f'^{*}l', l) \\
      & \simeq \Map_{/x}(f^{*}\eta_{!}l', l) \simeq \Map_{/y}(\eta_{!}l', f_{*}l) \\
      & \simeq \Map_{/y}(l', \eta^{*}f_{*}l),
    \end{split}
  \]
  using the functors $\eta_{!}$ and $\xi_{!}$ given by composition
  with $\eta$ and $\xi$, respectively, which act as left adjoints to
  $\eta^{*}$ and $\xi^{*}$ when pullbacks along $\eta$ and $\xi$ exist,
  and the full strength of condition \cref{eq:uni-prop-dis} for
  distributivity diagrams, which implies that we have the
  second-to-last equivalence even though $\eta_{!}l'$ is not necessarily in
  $\mathcal{C}_{L}$.

  Now suppose our triple satisfies the assumption on Beck--Chevalley
  transformations. To check that it is a bispan triple we must show
  that for $l \colon c \to x$ in $\mathcal{C}_{L}$ and $f \colon x \to
  y$ in $\mathcal{C}_{F}$ the pushforward $f_{*}l$ has the universal
  property \cref{eq:uni-prop-dis}, \ie{} that for every morphism
  $\eta \colon y' \to y$ we have a natural equivalence
  \[ \Map_{/y}(\eta, f_{*}l) \simeq \Map_{/x}(f^{*}\eta, l).\]
  Denoting the pullback square containing $\eta$ and $f$ as above, we
  have
  \[
    \begin{split}
      \Map_{/y}(\eta, f_{*}l) & \simeq \Map_{/y}(\eta_{!}\id_{y'},
      f_{*}l) 
       \simeq \Map_{/y'}(\id_{y'}, \eta^{*}f_{*}l) \\
      & \simeq \Map_{/y'}(\id_{y'}, f'_{*}\xi^{*}l)
       \simeq \Map_{/x'}(f'^{*}\id_{y'}, \xi^{*}l) \\
      & \simeq \Map_{/x'}(\id_{x'}, \xi^{*}l) 
       \simeq \Map_{/x}(\xi_{!}\id_{x'}, l) \\
      & \simeq \Map_{/x}(f^{*}\eta, l),
    \end{split}
  \]
  where the fourth equivalence holds because $\id_{y'}$ is in
  $\mathcal{C}_{L}$.
\end{proof}

\begin{remark}\label{rem:lcc}
  If $\mathcal{C}$ is locally cartesian closed, then
  all distributivity diagrams exist in $\mathcal{C}$ for any choice of
  $\mathcal{C}_{L}$; this is the case of $\mathcal{C}$ is an
  $\infty$-topos, for example.
\end{remark}

\begin{notation}
  We use the following notation for a functor $\Phi \Span_{F}(\mathcal{C}) \to
  \mathcal{X}$: 
  For any morphism $f \colon x \to y$ in $\mathcal{C}$, we write
    \[ f^{\ostar} := \Phi([f]_{B}) \colon 
    \colon \Phi(y) \to \Phi(x), \]
  and if $f$ lies in $\mathcal{C}_{F}$ we also write
\[  f_{\otimes} := \Phi([f]_{F}) \colon \Phi(x) \to \Phi(y).\]
\end{notation}

For the next definitions we fix a bispan triple $(\mathcal{C}, \mathcal{C}_{F},
\mathcal{C}_{L})$, an \itcat{} $\mathcal{X}$, and a functor
  \[\Phi \colon
    \Span_{F}(\mathcal{C}) \to \mathcal{X}.\]

\begin{defn} \label{def:L-dis}
  Given $l \colon x \to y$ in $\mathcal{C}_{L}$ and $f \colon y \to z$
  in $\mathcal{C}_{F}$, we have by assumption a distributivity diagram as in
  \cref{eq:dis-sq} in $\mathcal{C}$. If $\Phi|_{\mathcal{C}^{\op}}$ is
  left $L$-adjointable we define the
  \emph{distributivity transformation} for $l$ and $f$ as the composite
    \begin{equation} \label{eq:dis-tr}
    g_{\oplus}\tilde{f}_{\otimes}\epsilon^{\ostar} \to
    g_{\oplus}\tilde{f}_{\otimes}\epsilon^{\ostar}l^{\ostar}l_{\oplus} \simeq
    g_{\oplus}\tilde{f}_{\otimes}\tilde{g}^{\ostar}l_{\oplus} \simeq
    g_{\oplus}g^{\ostar}f_{\otimes}l_{\oplus} \to f_{\otimes}l_{\oplus},
    \end{equation}
  where the
  first map uses the unit for the adjunction
  $l_\oplus \dashv l^\ostar$, the second equivalence uses the
  functoriality of $\Phi$ for compositions of spans, and the last map
  uses the counit of the adjunction $g_\oplus \dashv g^\ostar$.
\end{defn}

\begin{defn} \label{def:dis}
  We say the functor $\Phi$ is \emph{$L$-distributive} if
  $\Phi|_{\mathcal{C}^{\op}}$  is left
  $L$-adjointable, and the distributivity transformation
  \cref{eq:dis-tr} is an 
  equivalence for all $l$ in $\mathcal{C}_{L}$ and $f$ in $\mathcal{C}_{F}$.  If the
  context is clear, we simply call $\Phi$ \emph{distributive}. We write 
  \[
  \Map_{\Ldist}(\Span_{F}(\mathcal{C}), \mathcal{X}) \subset
  \Map_{\CatIT}(\Span_{F}(\mathcal{C}), \mathcal{X}),
  \]
  for the subspace spanned by the $L$-distributive functors.
\end{defn}

\begin{remark}\label{rmk:dist=adj}
  The distributivity transformation \cref{eq:dis-tr} is precisely the
  mate transformation for the commutative square
  \begin{equation}\label{eq:dis-mate}    
    \begin{tikzcd}
      \Phi(y) \arrow{r}{l^{\ostar}} \arrow{dd}{f_{\otimes}} & \Phi(x)
      \arrow{d}{\epsilon^{\ostar}} \\
      & \Phi(w \times_{z} y) \arrow{d}[swap]{\tilde{f}_{\otimes}} \\
      \Phi(z) \arrow{r}{g^{\ostar}} & \Phi(w).
    \end{tikzcd}
  \end{equation}
  We can therefore reformulate the condition for the functor $\Phi$ to
  be $L$-distributive as: for every distributivity diagram
  \cref{eq:dis-sq}, the square \cref{eq:dis-mate} is left
  adjointable.
\end{remark}

\begin{variant}\label{var:rightdist}
  Let $(\mathcal{C}, \mathcal{C}_{F}, \mathcal{C}_{L})$ be a bispan
  triple and $\mathcal{X}$ an \itcat{}. We say that a functor
  $\Phi \colon \Span_{F}(\mathcal{C})^{\op} \to \mathcal{X}$ is
  \emph{$L$-codistributive} if the opposite functor
  \[ \Phi^{\op} \colon \Span_{F}(\mathcal{C}) \to \mathcal{X}^{\op} \]
  is $L$-distributive. %
  Similarly, we say that a functor
  $\Phi \colon \Span_{F}(\mathcal{C}) \to \mathcal{X}$ is \emph{right
    $L$-distributive} if the functor
  \[\Span_{F}(\mathcal{C}) \simeq \Span_{F}(\mathcal{C})^{\twop} \xto{\Phi^{\twop}}  \to \mathcal{X}^{\twop} \]
  is $L$-distributive. Since left adjoints in $\mathcal{X}^{\twop}$
  correspond to right adjoints in $\mathcal{X}$, this amounts to: for
  every morphism $l \in \mathcal{C}_{L}$ the morphism $l^{\ostar}$ has
  a \emph{right} adjoint $l_{\ostar}$, and the restriction of $\Phi$
  to $\mathcal{C}^{\op} \to \mathcal{X}$ is right
  $L$-adjointable. Moreover, given a distributivity diagram as
  in~\eqref{eq:dis-sq} with $l$ in $\mathcal{C}_{L}$ and $f$ in
  $\mathcal{C}_{F}$, we have a right distributivity transformation
  \[
    f_{\otimes}l_{\ostar} \to
    g_{\ostar}g^{\ostar}f_{\otimes}l_{\ostar} \simeq
    g_{\ostar}\tilde{f}_{\otimes}\epsilon^{\ostar}l^{\ostar}l_{\ostar}
    \to g_{\ostar}\tilde{f}_{\otimes}\epsilon^{\ostar},\] which is
  required to be an equivalence.
\end{variant}

Just as for adjointability, there is a natural notion of an
$L$-distributive transformation:
\begin{defn}
  By \cref{adjinarrows}, an $L$-distributive functor
  $\Span_{F}(\mathcal{C}) \to \mathcal{X}^{\Delta^{1}}$
  corresponds to a natural transformation
  \[ \eta \colon \Span_{F}(\mathcal{C}) \times \Delta^{1} \to
    \mathcal{X} \] such that both $\eta_{0}$ and $\eta_{1}$ are
  $L$-distributive functors, and the mate square for the required left
  adjoints commutes. We call such a natural transformation an
  \emph{$L$-distributive transformation}. We let
  $\Fun_{\Ldist}(\Span_{F}(\mathcal{C}), \mathcal{X})$ denote the
  subcategory of $\Fun(\Span_{F}(\mathcal{C}),\mathcal{X})$ consisting
  of $L$-distributive functors and $L$-distributive transformations.
  From \cref{adjinarrows} we also know that a functor
  $\Span_{F}(\mathcal{C}) \to \mathcal{X}^{\mathsf{C}_{2}}$ is
  $L$-distributive \IFF{} its underlying functors and natural
  transformations to $\mathcal{X}$ are $L$-distributive, \ie{} \IFF{}
  the adjoint morphism
  \[ \mathsf{C}_{2} \to \FUN(\Span_{F}(\mathcal{C}), \mathcal{X}) \] factors
  through $\Fun_{\Ldist}(\Span_{F}(\mathcal{C}), \mathcal{X})$ on
  underlying \icats{}. We therefore write
  $\FUN_{\Ldist}(\Span_{F}(\mathcal{C}), \mathcal{X})$ for the locally
  full subcategory of $\FUN(\Span_{F}(\mathcal{C}), \mathcal{X})$
  whose underlying \icat{} is
  $\Fun_{\Ldist}(\Span_{F}(\mathcal{C}), \mathcal{X})$.
\end{defn}

\begin{propn}\label{propn:LdistinFUN}
  For any \itcat{} $\mathcal{Y}$ there is a natural equivalence
  \[ \Map_{\CatIT}(\mathcal{Y}, \FUN_{\Ldist}(\Span_{F}(\mathcal{C}),
    \mathcal{X})) \simeq \Map_{\Ldist}(\Span_{F}(\mathcal{C}),
    \FUN(\mathcal{Y}, \mathcal{X})). \]
\end{propn}
\begin{proof}
  We claim that the two sides are identified under the natural
  equivalence
  \[ \Map_{\CatIT}(\mathcal{Y}, \FUN(\Span_{F}(\mathcal{C}),
    \mathcal{X})) \simeq \Map(\Span_{F}(\mathcal{C}),
    \FUN(\mathcal{Y}, \mathcal{X})). \] Indeed, the subspace
  $\Map_{\CatIT}(\mathcal{Y}, \FUN_{\Ldist}(\Span_{F}(\mathcal{C}),
  \mathcal{X}))$ of the left-hand side consists of those functors
  $\Phi \colon \mathcal{Y} \to \FUN(\Span_{F}(\mathcal{C}),
  \mathcal{X})$ such that for every object $y \in \mathcal{Y}$ the
  image $\Phi(y)$ is an $L$-distributive functor and for every
  morphism $f \colon y \to y'$ the image $\Phi(f)$ is an
  $L$-distributive natural transformation. Since the distributivity
  transformation is given pointwise by distributivity transformations
  in $\mathcal{X}$, and equivalences in
  $\FUN(\mathcal{Y}, \mathcal{X})$ are detected by evaluation at all
  objects of $\mathcal{Y}$, these conditions precisely correspond to
  $L$-distributivity for the adjoint functor
  $\Span_{F}(\mathcal{C}) \to \FUN(\mathcal{Y}, \mathcal{X})$ by
  \cref{adjinarrows}.
\end{proof}

\subsection{The $(\infty,2)$-category of
  bispans}\label{subsec:Bispan2cat}
Our goal in this subsection is to prove our main theorem:
\begin{thm}\label{thm:main2}
  Suppose $(\mathcal{C}, \mathcal{C}_{F}, \mathcal{C}_{L})$ is a
  bispan triple. Then there exists an \itcat{}
  $\BISPAN_{F,L}(\mathcal{C})$ equipped with an $L$-distributive functor
  \[j \colon \Span_{F}(\mathcal{C}) \to \BISPAN_{F,L}(\mathcal{C})\]
  such that
  \begin{enumerate}[(i)]
  \item\label{it:bispanprop} Composition with $j$ gives an equivalence
    \[ \FUN(\BISPAN_{F,L}(\mathcal{C}), \mathcal{X}) \isoto
      \FUN_{\Ldist}(\Span_{F}(\mathcal{C}), \mathcal{X})\]
    for any \itcat{} $\mathcal{X}$.
  \item\label{it:bispangpd} On underlying \igpds{} $j$ gives an equivalence
    \[ \mathcal{C}^{\simeq} \simeq \Span_{F}(\mathcal{C})^{\simeq}
      \isoto \BISPAN_{F,L}(\mathcal{C})^{\simeq}.\]    
  \item\label{it:bispanmor} Morphisms in $\BISPAN_{F,L}(\mathcal{C})$ are bispans
    \[   \begin{tikzcd}
        {} & X \arrow{r}{f} \arrow{dl}[swap]{p} & Y \arrow{dr}{l} \\
        I & & & J
      \end{tikzcd}
    \]
    with $f$ in $\mathcal{C}_{F}$ and $l$ in $\mathcal{C}_{L}$, with
    composition given by \cref{eq:bispancomp}.
  \item\label{it:bispan2mor} 2-morphisms in $\BISPAN_{F,L}(\mathcal{C})$ are given by
    diagrams of the form \cref{eq:bispanmor}.
  \end{enumerate}
\end{thm}

In order to prove this we will reinterpret the notion of
distributivity as a special case of left adjointability, as follows:
\begin{thm}\label{thm:distisadj}
  Suppose $(\mathcal{C}, \mathcal{C}_{F})$ is a span pair.
  \begin{enumerate}[(1)]
  \item\label{bispantripleispair} $(\mathcal{C}, \mathcal{C}_{F},\mathcal{C}_{L})$ is a bispan
    triple \IFF{} $(\Span_{F}(\mathcal{C})^{\op},\mathcal{C}_{L})$ is
    a span pair.
  \item\label{distftrisadj} In case \ref{bispantripleispair} holds, a functor
    $\Phi \colon \Span_{F}(\mathcal{C}) \to \mathcal{X}$ is
    $L$-distributive \IFF{} it is left $L$-adjointable.
  \end{enumerate}
\end{thm}
Given this, the universal property we want is precisely that of the
\itcat{} $\SPAN_{L}(\Span_{F}(\mathcal{C})^{\op})$ --- this amounts to
interpreting a bispan as a ``span in spans'':
\[
  \begin{tikzcd}
        {} & X \arrow{r}{f} \arrow{dl}[swap]{p} & Y \arrow{dr}{l} \\
        I & & & J
      \end{tikzcd}
      = 
  \begin{tikzcd}
    & X \arrow{dr}{f} \arrow{dl}[swap]{p} & & Y \arrow{dr}{l}
    \arrow[equals]{dl} \\
    I & & Y \arrow[dashed]{ll} \arrow[dashed]{rr} & & J.
  \end{tikzcd}
\]
The observation that the composition law for bispans can be
interpreted as a pullback in spans is due to
Street~\cite{StreetPoly}. To prove this in our \icatl{} setting we
must first describe certain pullback squares in $\Span_{F}(\mathcal{C})^{\op}$:

\begin{notation}
Let us write $\Sq(\mathcal{C}) := \Fun([1] \times [1], \mathcal{C})$
for the \icat{} of squares in an \icat{} $\mathcal{C}$, and given two
morphisms $f,g$ in $\mathcal{C}$ with common codomain we write
$\Sq_{f,g}(\mathcal{C})$ for the fibre of $\Sq(\mathcal{C})$ at $(f,g)
\in \Fun(\Lambda^{2}_{2}, \mathcal{C})$ (where we view $[1] \times
[1]$ as $(\Lambda^{2}_{2})^{\triangleleft}$). Then a pullback of $f$
and $g$ is precisely a terminal object in
$\Sq_{f,g}(\mathcal{C})$. Note also that evaluation at the initial object
in $[1] \times [1]$ gives a right fibration $\Sq_{f,g}(\mathcal{C})
\to \mathcal{C}$.
\end{notation}

\begin{notation}
  We can identify $\Span_{F}(\mathcal{C})^{\op}$ as the \icat{} of
  spans in $\mathcal{C}$ where the \emph{backwards} map must lie in
  $\mathcal{C}_{F}$, with composition given by pullbacks in
  $\mathcal{C}$ as usual. To emphasize this, we will use the notation
  \[ \Span_{F,\all}(\mathcal{C}) :=
    \Span_{F}(\mathcal{C})^{\op}.\]
  For $f \colon x \to y$ in $\mathcal{C}_{F}$, we then write $[f]_{B}$ for the
  morphism $y \xfrom{f} x \xto{=} x$ in $\Span_{F,\all}(\mathcal{C})$, and
  for any $f$ in $\mathcal{C}$ we write $[f]_{F}$ for the span $x
  \xfrom{=} x \xto{f} y$.
\end{notation}

\begin{remark}\label{rmk:generalsqsofspans}
 A square in $\Span_{F,\all}(\mathcal{C})$
  can then be identified with a diagram of shape
  \[
    \begin{tikzcd}
      \bullet & \bullet \arrow{r} \arrow{l} & \bullet \\
      \bullet \arrow{u} \arrow{d} & \bullet \arrow{r} \arrow{l}
      \arrow{u} \arrow{d} \urpullback \dlpullback & \bullet \arrow{u} \arrow{d} \\
            \bullet & \bullet \arrow{r} \arrow{l} & \bullet,
    \end{tikzcd}
  \]
  where the two indicated squares are cartesian and the backwards maps
  all lie in $\mathcal{C}_{F}$. Composing this with a span (labelled
  in red) we get the
  diagram %
  \begin{equation}
    \label{eq:sqgencart}
    \begin{tikzcd}
      \bullet \\
       & \bullet \arrow[color=red]{ul} \arrow[color=red]{dr} & \bullet
       \arrow[phantom]{d}[very near start,description]{\llcorner}   \arrow{l} \arrow{dr}
       \arrow{ull} \arrow{drr} \\
       & \bullet \arrow{u} \arrow{dr} \arrow{uul} \arrow{ddr}
       \arrow[phantom]{r}[very near start,description]{\urcorner}
       & \bullet & \bullet \arrow{r} \arrow{l} & \bullet \\
      & & \bullet \arrow{u} \arrow{d} & \bullet \arrow{r} \arrow{l}
      \arrow{u} \arrow{d} \urpullback \dlpullback & \bullet \arrow{u} \arrow{d} \\
      & &      \bullet & \bullet \arrow{r} \arrow{l} & \bullet,
    \end{tikzcd}
  \end{equation}
  which exhibits a cartesian morphism in
  $\Sq_{\phi,\psi}(\Span_{F,\all}(\mathcal{C}))$ where $\phi$ and
  $\psi$ are the bottom and right sides of the square.
\end{remark}

\begin{remark}
  In particular, we can identify objects
  of $\Sq_{[f]_{F},[g]_{F}}(\Span_{F,\all}(\mathcal{C}))$ with diagrams of
  the form
    \[
    \begin{tikzcd}
      \bullet & \bullet \arrow{r} \arrow{l} & \bullet \\
      \bullet \arrow{u} \arrow{d} & \bullet \arrow{r} \arrow[equals]{l}
      \arrow[equals]{u} \arrow{d} \urpullback \dlpullback & \bullet \arrow[equals]{u} \arrow{d}{g} \\
            \bullet & \bullet \arrow{r}{f} \arrow[equals]{l} & \bullet,
    \end{tikzcd}
  \]
  which we can simplify to
  \begin{equation}
    \label{eq:twofsq}
    \begin{tikzcd}
      \bullet & \bullet \arrow{l} \arrow{r} \arrow{d} & \bullet \arrow{d}{g} \\
      & \bullet \arrow{r}{f} & \bullet,
    \end{tikzcd}
  \end{equation}
  where the top left arrow is required to lie in $\mathcal{C}_{F}$.
  Simplifying \cref{eq:sqgencart} similarly, we see that the cartesian
  morphism in $\Sq_{[f]_{F},[g]_{F}}(\Span_{F,\all}(\mathcal{C}))$  over a span 
  is given by a diagram of the form
  \begin{equation}
    \label{eq:twofcart}
    \begin{tikzcd}
      \bullet & \bullet \arrow[color=red]{l} \arrow[color=red]{d}  & \bullet \arrow{l}
      \arrow{d} \dlpullback \arrow[color=blue]{dr} \arrow[bend right=30,color=blue]{dd}
      \arrow[bend right=30,color=blue]{ll}\\
      & \bullet & \bullet  \arrow[crossing over,color=orange]{l} \arrow[color=orange]{r} \arrow[color=orange]{d} & \bullet \arrow{d}{g} \\
       & & \bullet \arrow{r}{f}&  \bullet,
    \end{tikzcd}
  \end{equation}
  where the blue arrows indicate the source object, the orange ones
  indicate the target object, and the red ones
  indicate the span we compose with.
\end{remark}

\begin{propn}\label{propn:ffpb}
  Let $(\mathcal{C},\mathcal{C}_{F})$ be a span pair. Given morphisms
  $f \colon a \to c$ and $g \colon b \to c$ in $\mathcal{C}$, the
  fibre product of $[f]_{F}$ and $[g]_{F}$ exists in
  $\Span_{F,\all}(\mathcal{C})$ \IFF{} the fibre product
  $d := a \times_{c}b$ of $f$ and $g$ exists in $\mathcal{C}$, in
  which case it is given by the diagram
      \[
    \begin{tikzcd}
      d & d \arrow{r} \arrow[equals]{l} & b \\
      d \arrow[equals]{u} \arrow{d} & d \arrow{r} \arrow[equals]{l}
      \arrow[equals]{u} \arrow{d} \drpullback \urpullback \dlpullback & b \arrow[equals]{u} \arrow{d}{g} \\
            a  & a \arrow{r}{f} \arrow[equals]{l} & c.
    \end{tikzcd}
  \]
\end{propn}
\begin{proof}
  We first check that if the diagram
  \[
    \begin{tikzcd}
      y & x \arrow{l}[swap]{\phi} \arrow{r}{u} \arrow{d}{v} & b \arrow{d}{g} \\
      & a \arrow{r}{f} & c
    \end{tikzcd}
  \]
  is a terminal object in
  $\Sq_{[f]_{F},[g]_{F}}(\Span_{F,\all}(\mathcal{C}))$, then $\phi$ is
  necessarily an equivalence. To see this, we first observe that since
  the diagram
  \begin{equation}
    \label{eq:twofeqob}
    \begin{tikzcd}
      x & x \arrow[equals]{l} \arrow{r}{u} \arrow{d}{v} & b \arrow{d}{g} \\
      & a \arrow{r}{f} & c
    \end{tikzcd}    
  \end{equation}
  is also an object of
  $\Sq_{[f]_{F},[g]_{F}}(\Span_{F,\all}(\mathcal{C}))$, there exists a
  unique span $x \xfrom{\alpha} y' \xto{\beta} y$ with $\alpha$ in
  $\mathcal{C}_{F}$, such that composing the terminal object with the
  span gives this object, \ie{}
  \[
    \begin{tikzcd}
      x & y' \arrow{l}{\alpha} \arrow{d}{\beta}  & x \arrow{l}{\phi'}
      \arrow{d}{\beta'} \dlpullback \arrow{dr}{u} \arrow[bend
      right=30]{dd}[near end,swap]{v}
      \arrow[bend right=30,equals]{ll}\\
      & y & x  \arrow[crossing over]{l}[swap]{\phi} \arrow{r}{u} \arrow{d}{v} & b \arrow{d}{g} \\
       & & a \arrow{r}{f} &  c.
    \end{tikzcd}
  \]
  But then composing the terminal object with the span $y
  \xfrom{\phi\alpha} y' \xto{\beta} y$ produces the terminal object, so
  by uniqueness this span is equivalent to the identity of $y$. Thus
  we may take $y' = y$ and under this identification $\beta$ and $\phi\alpha$
  are identified with $\id_{y}$. Then the pullbacks $\beta'$ and
  $\phi'$ may also be identified with $\id_{y}$ and $\phi$,
  respectively, which tells us that we have $\alpha \phi\simeq
  \id_{x}$ and $\phi \alpha \simeq \id_{y}$. Thus $\phi$ is an
  equivalence, as we wanted to show.

  Now observe that composing \cref{eq:twofeqob}
  with a span $z \xfrom{\alpha} y \xto{\beta} x$, we obtain the object
    \[
    \begin{tikzcd}
      z & y \arrow{l}[swap]{\alpha} \arrow{r}{u\beta} \arrow{d}{v\beta} & b \arrow{d}{g} \\
      & a \arrow{r}{f} & c.
    \end{tikzcd}
  \]
  Thus an object
    \[
    \begin{tikzcd}
      z & y \arrow{l}[swap]{\alpha} \arrow{r}{\gamma} \arrow{d}{\delta} & b \arrow{d}{g} \\
      & a \arrow{r}{f} & c
    \end{tikzcd}
  \]
  of $\Sq_{[f]_{F},[g]_{F}}(\Span_{F,\all}(\mathcal{C}))$
  has a unique map to \cref{eq:twofeqob} \IFF{} there is a unique
  diagram
  \[
    \begin{tikzcd}
      y \arrow{dr}{\beta}\arrow[bend left=20]{drr}{\gamma} \arrow[bend
      right=20]{ddr}{\delta} \\
      & x \arrow{r}{u} \arrow{d}{v} & b \arrow{d} g \\
      & a \arrow{r}{f} & c
    \end{tikzcd}
    \]
  in $\mathcal{C}$. This is true for all objects of
  $\Sq_{[f]_{F},[g]_{F}}(\Span_{F,\all}(\mathcal{C}))$ \IFF{} the square
  in \cref{eq:twofeqob} is cartesian in $\mathcal{C}$, which is what
  we wanted to prove.
\end{proof}

\begin{remark}
  Returning to \cref{rmk:generalsqsofspans}, an object of
  $\Sq_{[f]_{B},[g]_{F}}(\Span_{F,\all}(\mathcal{C}))$ can be identified
  with a diagram
    \[
    \begin{tikzcd}
      \bullet & \bullet \arrow{r} \arrow{l} & \bullet \\
      \bullet \arrow{u} \arrow{d} & \bullet \arrow{r} \arrow{l}
      \arrow[equals]{u} \arrow{d} \urpullback \dlpullback & \bullet \arrow[equals]{u} \arrow{d}{g} \\
            \bullet & \bullet \arrow[equals]{r}\arrow{l}{f} & \bullet,
    \end{tikzcd}
  \]
  which we can simplify to
  \begin{equation}
    \label{eq:bfsq}
    \begin{tikzcd}
      {} & \bullet \arrow{r} \arrow{dl} \arrow{dd}
      \arrow[phantom]{ddr}[very near start,description]{\lrcorner}& \bullet
      \arrow{dd} \arrow{r} & \bullet \\
      \bullet \arrow{dr}{g} \\
       & \bullet \arrow{r}{f} & \bullet.
    \end{tikzcd}
  \end{equation}
  Composing this with a span gives, by simplifying
  \cref{eq:sqgencart}, a diagram of the form
  \begin{equation}
    \label{eq:bfsqcart}
    \begin{tikzcd}
      {} & \bullet \arrow[color=blue]{r} \arrow[bend
      right=20,color=blue]{ddl} \arrow{d} \drpullback & \bullet
      \arrow{r} \arrow[bend left=20,color=blue]{ddd} \arrow[bend left=25,color=blue]{rr}
      \arrow{d} \drpullback& \bullet \arrow[color=red]{r} \arrow[color=red]{d} & \bullet \\
      {} & \bullet \arrow[color=orange]{r} \arrow[color=orange]{dl} \arrow{dd}
      \arrow[phantom]{ddr}[very near start,description]{\lrcorner}& \bullet
      \arrow[color=orange]{dd} \arrow[crossing over,color=orange]{r} & \bullet \\
      \bullet \arrow{dr}{g} \\
       & \bullet \arrow{r}{f} & \bullet,
    \end{tikzcd}
  \end{equation}
  with the blue arrows indicating the source object, the orange ones
  the target object, and the red ones
  the span we compose with.
\end{remark}

\begin{propn}\label{propn:bfpb}
  Let $(\mathcal{C},\mathcal{C}_{F})$ be a span pair. Given morphisms
  $g \colon a \to b$ in $\mathcal{C}$ and $f \colon b \to c$ in $\mathcal{C}_{F}$, the
  fibre product of $[f]_{B}$ and $[g]_{F}$ exists in
  $\Span_{F,\all}(\mathcal{C})$ \IFF{} there exists a distributivity
  diagram
  \[
    \begin{tikzcd}
      {} & d \arrow{r}{f'} \arrow{dl}{\epsilon} \arrow{dd}{h'}
      \arrow[phantom]{ddr}[very near start,description]{\lrcorner} &
      e \arrow{dd}{h} \\
      a \arrow{dr}{g} \\
      & b \arrow{r}{f} & c
    \end{tikzcd}
  \]
  in $\mathcal{C}$, in which case it is given by the diagram
  \[
    \begin{tikzcd}
      e & d \arrow{r}{\epsilon} \arrow{l}{f'} & a \\
      e \arrow[equals]{u} \arrow{d}{h} & d \arrow{r}{\epsilon} \arrow{l}{f'}
      \arrow[equals]{u} \arrow{d}{h'} \urpullback \dlpullback & a \arrow[equals]{u} \arrow{d}{g} \\
            c & b \arrow[equals]{r}\arrow{l}{f} & b,
    \end{tikzcd}
  \]
\end{propn}
\begin{proof}
  Suppose
  \[
    \begin{tikzcd}
      {} & x \arrow{r} \arrow{dl} \arrow{dd}
      \arrow[phantom]{ddr}[very near start,description]{\lrcorner}& y
      \arrow{dd} \arrow{r}{\phi} & z \\
      a \arrow{dr}{g} \\
       & b \arrow{r}{f} & c
    \end{tikzcd}
  \]
  is a terminal object in
  $\Sq_{[f]_{B},[g]_{F}}(\Span_{F,\all}(\mathcal{C}))$. By the same
  argument as in the proof of \cref{propn:ffpb} we see that $\phi$
  must be an equivalence. Moreover, if we compose an object of the form
  \begin{equation}\label{eq:bfsqeq}
    \begin{tikzcd}
      {} & x \arrow{r} \arrow{dl} \arrow{dd}
      \arrow[phantom]{ddr}[very near start,description]{\lrcorner}& y
      \arrow{dd} \arrow[equals]{r} & y \\
      a \arrow{dr}{g} \\
       & b \arrow{r}{f} & c
    \end{tikzcd}
  \end{equation}
  with a span $y \xfrom{\beta} z \xto{\gamma} w$ we see from
  \cref{eq:twofcart} that we get the the outer part of the diagram
  \[
    \begin{tikzcd}
      {} & z' \arrow{r} \drpullback \arrow{d} \arrow[bend
      right=20]{ddl} & z \arrow{d}{\beta} \arrow{r}{\gamma} &
      w
      \\
      {} & x \arrow{r} \arrow{dl} \arrow{dd}
      \arrow[phantom]{ddr}[very near start,description]{\lrcorner}& y
      \arrow{dd}  \\
      a \arrow{dr}{g} \\
      & b \arrow{r}{f} & c.
    \end{tikzcd}
  \]
  Thus from \cref{rmk:distdiagterminal} we see that every object  of
  $\Sq_{[f]_{B},[g]_{F}}(\Span_{F,\all}(\mathcal{C}))$ admits a unique
  morphism to an object \cref{eq:bfsqeq} \IFF{} this is given by a
  distributivity diagram in $\mathcal{C}$, as required.
\end{proof}

\begin{proof}[Proof of \cref{thm:distisadj}]
  Every morphism in $\Span_{F,\all}(\mathcal{C})$ is a composite of
  morphisms of the form $[f]_{B}$ (with $f$ in $\mathcal{C}_{F}$) and
  $[g]_{F}$. Thus $(\Span_{F,\all}(\mathcal{C}), \mathcal{C}_{L})$ is a
  span pair \IFF{} for all $l$ in $\mathcal{C}_{L}$ the pullbacks of
  $[l]_{F}$ along $[g]_{F}$ and $[f]_{B}$ exist for all $g$ in $\mathcal{C}$
  and $f$ in $\mathcal{C}_{F}$, and these pullbacks lie in
  $\mathcal{C}_{L}$. From \cref{propn:ffpb} and \cref{propn:bfpb} we
  see that these conditions are equivalent to
  $(\mathcal{C},\mathcal{C}_{F},\mathcal{C}_{L})$ being a bispan
  triple, which proves \ref{bispantripleispair}. To prove
  \ref{distftrisadj}, observe that since adjointable squares compose,
  the functor $\Phi$ is left $L$-adjointable \IFF{} we get left
  adjointable squares for both types of pullbacks along morphisms in
  $\mathcal{C}_{L}$ separately. From \cref{propn:ffpb} adjointability
  for the first type corresponds to $\Phi|_{\mathcal{C}^{\op}}$ being
  left $L$-adjointable, and adjointability for the second type
  then corresponds to $L$-distributivity by \cref{rmk:dist=adj}.
\end{proof}

\begin{defn}
  In light of \cref{thm:distisadj}, if
  $(\mathcal{C},\mathcal{C}_{F},\mathcal{C}_{L})$ is a bispan triple
  it makes sense to define
  \[ \BISPAN_{F,L}(\mathcal{C}) :=
    \SPAN_{L}(\Span_{F}(\mathcal{C})^{\op}).\]
\end{defn}

\begin{notation}
In keeping with our
conventions so far, we will denote the underlying \icat{} of $\BISPAN_{F,L}(\mathcal{C})$ by
\[
\Bispan_{F,L}(\mathcal{C}):= \BISPAN_{F,L}(\mathcal{C})^{(1)}.
\]
In examples we will often have $\mathcal{C}_{L} \simeq \mathcal{C}$,
in which case we abbreviate
$\BISPAN_{F}(\mathcal{C}) := \BISPAN_{F,L}(\mathcal{C})$. If we also
have $\mathcal{C}_{F} \simeq \mathcal{C}$, we write
$\BISPAN(\mathcal{C})$ for $\BISPAN_{F}(\mathcal{C})$. We also adopt
the same conventions for the \icat{} of bispans
$\Bispan_{F,L}(\mathcal{C})$.
\end{notation}

Applying \cref{cor:improve}, we get:
\begin{cor}
  Let $(\mathcal{C},\mathcal{C}_{F},\mathcal{C}_{L})$ be a bispan
  triple. The \icat{} $\BISPAN_{F,L}(\mathcal{C})$ satisfies
  \[ \FUN(\BISPAN_{F,L}(\mathcal{C}), \mathcal{X}) \simeq
    \FUN_{\Ldist}(\Span_{F}(\mathcal{C}), \mathcal{X})\]
  for any \itcat{} $\mathcal{X}$. \qed
\end{cor}

This proves part \ref{it:bispanprop} of \cref{thm:main2}. We now prove
the remainder:
\begin{proof}[Proof of \cref{thm:main2}]
  Part \ref{it:bispangpd} is immediate from \cref{cor:ieqgpd}. A
  morphism in
  $\SPAN_{L}(\Span_{F,\all}(\mathcal{C}))$ is a ``span of spans''
  \[
    \begin{tikzcd}
      {} & X \arrow{dl}{p} \arrow{dr}{f} & & Y \arrow[equals]{dl}
      \arrow{dr}{l} \\
      I & & Y \arrow[dashed]{ll} \arrow[dashed]{rr} & & J,
    \end{tikzcd}
  \]
  with $f$ in $\mathcal{C}_{F}$ and $l$ in $\mathcal{C}_{L}$, which we
  can think of as a bispan by contracting the identity. To compose
  these we take pullbacks in $\Span_{F,\all}(\mathcal{C})$, which we
  can unpack to give the expected composition law for bispans using
  \cref{propn:ffpb} and \cref{propn:bfpb}. %
  Unpacking the definition of a 2-morphism in
  $\SPAN_{L}(\Span_{F,\all}(\mathcal{C}))$, we get a diagram
  
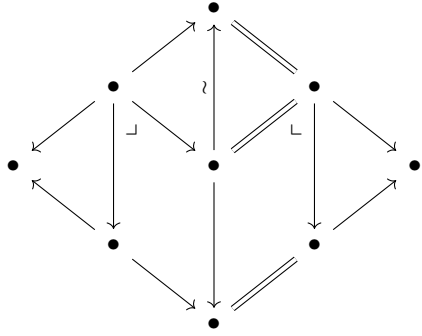
\[
    \begin{tikzcd}
      {} & & \bullet & & \\
       & \bullet \arrow{ur} \arrow{dr} \arrow{dl} \arrow{dd} \arrow[phantom]{dddr}[very near
  start,description]{\lrcorner} &
       & \bullet \arrow[equals]{dl}  \arrow[equals]{ul} \arrow{dr}\arrow{dd} \arrow[phantom]{dddl}[very near
  start,description]{\llcorner} \\
       \bullet & & \bullet \arrow{uu}{\wr} \arrow{dd} & & \bullet \\
       & \bullet \arrow{ul} \arrow{dr} & & \bullet \arrow[equals]{dl} \arrow{ur} \\
        & & \bullet
    \end{tikzcd}
  \]
  where the upward-pointing map is necessarily an equivalence, as
  indicated. Contracting the invertible edges, we get a diagram of
  shape \cref{eq:bispanmor}, as required.
\end{proof}

The universal property also implies that our \itcats{} of bispans are
functorial for morphisms of bispan triples, in the following sense:
\begin{defn}
  A morphism of bispan triples $(\mathcal{C}, \mathcal{C}_{F},
  \mathcal{C}_{L}) \to (\mathcal{C}', \mathcal{C}'_{F'},
  \mathcal{C}'_{L'})$ is a functor $\phi \colon \mathcal{C} \to
  \mathcal{C}'$ such that $\phi(\mathcal{C}_{F}) \subseteq
  \mathcal{C}'_{F'}$, $\phi(\mathcal{C}_{L}) \subseteq
  \mathcal{C}'_{L'}$, and $\phi$ preserves pullbacks along
  $\mathcal{C}_{F}$ and $\mathcal{C}_{L}$ as well as distributivity
  diagrams. We define $\Trip$ to be the subcategory of
  $\Fun(\Lambda^{2}_{2}, \CatI)$ containing the bispan triples and the
  morphisms thereof.
\end{defn}

\begin{propn}
  There is a functor $\BISPAN \colon \Trip \to \CatIT$ that
  takes $(\mathcal{C}, \mathcal{C}_{F}, \mathcal{C}_{L})$ to $\BISPAN_{F,L}(\mathcal{C})$.
\end{propn}
\begin{proof}
  Composition with a morphism of bispan triples $\phi \colon (\mathcal{C}, \mathcal{C}_{F},
  \mathcal{C}_{L}) \to (\mathcal{C}', \mathcal{C}'_{F'},
  \mathcal{C}'_{L'})$ restricts for any \itcat{} $\mathcal{X}$ to a morphism
  \[ \Map_{L'\ddist}(\Span_{F'}(\mathcal{C}'), \mathcal{X}) \to
    \Map_{\Ldist}(\Span_{F}(\mathcal{C}), \mathcal{X}), \]
  so we get a functor
  \[ \Map_{(\blank)\ddist}(\Span_{(\blank)}(\blank), \blank) \colon
    \Trip^{\op} \times \CatIT \to \mathcal{S}.\]
  By \cref{thm:main2} the associated functor
  $\Trip^{\op} \to \Fun(\CatIT, \mathcal{S})$ factors through the full
  subcategory of corepresentable copresheaves, and so by the Yoneda
  lemma arises from a functor $\BISPAN \colon \Trip \to \CatIT$, as
  required.  
\end{proof}

Applying \cref{propn:adjftrspandesc} to
$\SPAN_{L}(\Span_{F,\all}(\mathcal{C})$, we obtain the following description of the functor of
\itcats{} corresponding to a distributive functor:
\begin{cor}\label{propn:distftrbispandesc}
  For an $L$-distributive functor $\phi \colon \Span_{F}(\mathcal{C})
  \to \mathcal{X}$, the corresponding functor $\Phi \colon
  \BISPAN_{F,L}(\mathcal{C}) \to \mathcal{X}$ can be described as
  follows:
  \begin{enumerate}[(1)]
  \item On objects, $\Phi(c) \simeq \phi(c)$ for $c \in \mathcal{C}$.
  \item On morphisms, $\Phi$ takes a bispan
    \[ B = (x \xfrom{s} e \xto{f} b \xto{l} y) \]
    to the composite $l_{\oplus}f_{\otimes}s^{\ostar} \colon \phi(x)
    \to \phi(y)$, where $l_{\oplus}$ is the left adjoint to
    $l^{\ostar}$.
  \item On 2-morphisms, $\Phi$ takes the 2-morphism $\beta \colon B
    \to B'$ given by the commutative diagram
    \[
      \begin{tikzcd}
        {} & e \arrow{dl}[swap]{s} \arrow{r}{f} \arrow{dd}{g}\arrow[phantom]{ddr}[very near
        start,description]{\lrcorner} & b \arrow{dd}{h} \arrow{dr}{l} \\
        x & & & y \\
         & e' \arrow{ul}{s'} \arrow{r}[swap]{f'} & b' \arrow{ur}[swap]{l'}
      \end{tikzcd}
    \]
    to the composite
    \[ l_{\oplus}f_{\otimes}s^{\ostar} \simeq
      l_{\oplus}f_{\otimes}g^{\ostar}s'^{\ostar} \simeq
      l_{\oplus}h^{\ostar}f'_{\otimes}s'^{\ostar} \to
      l_{\oplus}h^{\ostar}l'^{\ostar}l'_{\oplus}f'_{\otimes}s'^{\ostar}
      \simeq  l_{\oplus}l^{\ostar}l'_{\oplus}f'_{\otimes}s'^{\ostar}
      \to l'_{\oplus}f'_{\otimes}s'^{\ostar},\]
      where the first noninvertible arrow is a unit and the second noninvertible arrow is a counit.
  \end{enumerate}
\end{cor}

We end this section with some useful observations about distributivity
diagrams that follow easily from \cref{propn:bfpb}:
\begin{lemma}\label{lem:lcompdist}
  Let $(\mathcal{C}, \mathcal{C}_{F},\mathcal{C}_{L})$ be a bispan
  triple, and suppose given morphisms $l_{1} \colon x\to y$ and $l_{2}
  \colon y \to z$ in $L$ and $f \colon z \to w$ in $F$. Then we can
  make the following diagram:
  \
  \begin{equation}
    \label{eq:Lcompdistdiag}
    \begin{tikzcd}
      {} & {} & \bullet \arrow{dd}{g_{1}'} \arrow{dl}{\epsilon_{1}}
      \arrow{r}{f''} & \bullet \arrow{dd}{g_{1}} \\
      {} & \bullet \arrow{dl}{\epsilon_{2}'} \arrow{dr}{l_{1}'} \\
      x \arrow{dr}{l_{1}} & & \bullet \arrow{r}{f'} \arrow{dl}{\epsilon_{2}} \arrow{dd}{g_{2}'} &
      \bullet  \arrow{dd}{g_{2}} \\
      & y \arrow{dr}{l_{2}} \\
      & & z \arrow{r}{f} & w.
    \end{tikzcd}
  \end{equation}
  Here all three squares are cartesian and the two rightmost give
  distributivity diagrams (so $g_{2} = f_{*}l_{2}$, $g_{1} =
  f'_{*}l_{1}'$). Then the outer diagram is a distributivity diagram
  for $l_{2}l_{1}$ and $f$.
\end{lemma}
\begin{proof}
  This follows immediately from the pasting lemma for pullback squares
  applied to $[l_{2}l_{1}]_{F} \simeq [l_{2}]_{F}\circ [l_{1}]_{F}$ and $[f]_{B}$ in
    $\Span_{F,\all}(\mathcal{C})$.
  \end{proof}

\begin{lemma}
  Let $(\mathcal{C}, \mathcal{C}_{F},\mathcal{C}_{L})$ be a bispan
  triple, and suppose given morphisms $l \colon x\to y$ in $L$ and
  $f_{1} \colon y \to z, f_{2} \colon z \to w$ in $F$. Then we can
  make the following diagram:
  \begin{equation}
    \label{eq:Fcompdistdiag}
    \begin{tikzcd}
      {} & {} & \bullet \arrow{rr}{f_{1}''} \arrow{dl}[swap]{\epsilon_{2}'} \arrow{dddl}{g_{2}''}
      & & \bullet \arrow{r}{f_{2}'} \arrow{dl}[swap]{\epsilon_{2}} \arrow{dddl}{g'_{2}} & \bullet
      \arrow{ddd}{g_{2}} \\
      {} & \bullet \arrow[crossing over]{rr}{f_{1}'}  \arrow{dd}{g_{1}'}
      \arrow{dl}[swap]{\epsilon_{1}} & & \bullet \arrow{dd}{g_{1}} \\
      x \arrow{dr}{l} \\
       & y \arrow{rr}{f_{1}} & & z \arrow{rr}{f_{2}} & & w.
    \end{tikzcd}
  \end{equation}
  Here all three squares are cartesian and two give distributivity
  diagrams (exhibiting $g_{1} = f_{1,*}l$ and $g_{2} = f_{2,*}g_{1}$. 
 Then the outer diagram is a distributivity diagram
  for $l$ and $f_{2}f_{1}$. (Note that the composite square is indeed
  cartesian since the two left squares are cartesian and the middle
  triangle just exhibits $g_{2}'$ as a composite.)
\end{lemma}
\begin{proof}
  This is the pasting lemma for pullback squares applied to
  $[f_{2}f_{1}]_{B} \simeq [f_{1}]_{B}\circ [f_{2}]_{B}$ and $[l]_{F}$.
\end{proof}

\begin{lemma}
  Let $(\mathcal{C}, \mathcal{C}_{F},\mathcal{C}_{L})$ be a bispan
  triple, and suppose we have morphisms $l \colon x \to y$ in
  $\mathcal{C}_{L}$, $f \colon y \to z$ in $\mathcal{C}_{F}$ and an arbitary
  morphism $\zeta \colon z' \to z$. Then we can form the following diagram
  \begin{equation}
    \label{eq:distpb}
    \begin{tikzcd}
      {} & w' \times_{z} y \arrow{ddl}{\epsilon'} \arrow{drr}[swap]{\omega
        \times_{\zeta} \eta} \arrow{rrr}{h'} \arrow{dddd}{f'^{*}g'} & & & w'
      \arrow{drr}{\omega} \arrow{dddd}{g'} \\
      & & & w \times_{z}y \arrow{ddl}{\epsilon} \arrow[crossing over]{rrr}{h} & & & w
      \arrow{dddd}{g} \\
      x' \arrow[crossing over]{drr}{\xi} \arrow{ddr}{l'}  \\
      & & x & & & \phantom{\bullet} \\
      & y' \arrow{rrr}[near start]{f'} \arrow{drr}{\eta} & & & z' \arrow{drr}{\zeta} \\
      & & & y \arrow{rrr}{f}  \arrow[leftarrow,crossing
      over]{uul}[near end]{l} \arrow[leftarrow,crossing over]{uuuu}{f^{*}g} & & & z
    \end{tikzcd}
  \end{equation}
  where the front face is a distributivity diagram for $l$ and $f$,
  and the rest of the diagram is obtained by pulling this back along
  $\zeta$. Then the back face in \cref{eq:distpb} is a distributivity
  diagram for $l'$ and $f'$.  
\end{lemma}
\begin{proof}
  This follows yet again from the pasting lemma for pullback squares,
  now applied to the pullback of $[l]_{F}$ along $[f]_{B}\circ
  [\zeta]_{F} \simeq [\eta]_{F} \circ [f']_{B}$.
\end{proof}

The following proposition can be interpreted as saying that
distributivity transformations are compatible with base change:
\begin{propn}\label{propn:pbdisttr}
  Let $(\mathcal{C}, \mathcal{C}_{F},\mathcal{C}_{L})$ be a bispan
  triple, and suppose we have the diagram \cref{eq:distpb}. Then
  the
  distributivity transformations for $(l,f)$ and $(l',f')$ are related
  by a commutative square
  \begin{equation}
    \label{eq:pbdisttr}
    \begin{tikzcd}
            g'_{\oplus}h'_{\otimes}\epsilon'^{\ostar}\xi^{\ostar}
            \arrow{r} \arrow{d}{\sim}
      & f'_{\otimes}l'_{\oplus}\xi^{\ostar} \arrow{d}{\sim} \\
      g'_{\oplus}\omega^{\ostar}h_{\otimes}\epsilon^{\ostar}
      \arrow{d}{\sim} &
      f'_{\otimes}\eta^{\ostar}l_{\oplus} \arrow{d}{\sim} \\
      \zeta^{\ostar}g_{\oplus}h_{\otimes}\epsilon^{\ostar} \arrow{r}
       & \zeta^{\ostar}f_{\otimes}l_{\oplus}
    \end{tikzcd}
  \end{equation}
\end{propn}
\begin{proof}
  We have the following commutative cube:
  \[
    \begin{tikzcd}
      \Phi(y) \arrow{dddd}{f_{\otimes}} \arrow{rr}{l^{\ostar}} \arrow{dr}{\eta^{\ostar}} & &
      \Phi(x) \arrow{dr}{\xi^{\ostar}} \arrow{dd}[near end]{\epsilon^{\ostar}}
      \\
       & \Phi(y') \arrow[crossing
       over]{rr}[near start]{l'^{\ostar}} & & \Phi(x')
       \arrow{dd}{\epsilon'^{\ostar}} \\
       & & \Phi(w \times_{z} y) \arrow{dr}{(\omega \times_{\zeta}
         \eta)^{\ostar}} \arrow{dd}{h_{\otimes}} \\
       & & & \Phi(w' \times_{z'} y') \arrow{dd}{h'_{\otimes}} \\
       \Phi(z) \arrow{rr}[near end]{g^{\ostar}} \arrow{dr}{\zeta^{\ostar}} & &
       \Phi(w) \arrow{dr}{\omega^{\ostar}} \\
        & \Phi(z') \arrow[leftarrow,crossing over]{uuuu}{f'_{\otimes}}
        \arrow{rr}{g'^{\ostar}} & & \Phi(w').
    \end{tikzcd}
  \]
  This means the following pair of diagrams are identified under the
  equivalences $f'_{\otimes}\eta^{\ostar} \simeq
  \zeta^{\ostar}f_{\otimes}$,
  $h'_{\otimes}\epsilon'^{\ostar}\xi^{\ostar} \simeq
  \omega^{\ostar}h_{\otimes}\epsilon^{\ostar}$ given by the left and
  right faces of the cube:
  \[
    \begin{tikzcd}
      \Phi(y) \arrow{r}{l^{\ostar}}  \arrow{dd}{f_{\otimes}} & \Phi(x)
      \arrow{d}{\epsilon^{\ostar}} \\
      & \Phi(w \times_{z} y) \arrow{d}{h_{\otimes}} \\
      \Phi(z) \arrow{d}{\zeta^{\ostar}} \arrow{r}{g^{\ostar}} &
      \Phi(w) \arrow{d}{\omega^{\ostar}} \\
      \Phi(z') \arrow{r}{g'^{\ostar}} & \Phi(w'),
    \end{tikzcd}
    \qquad\quad
    \begin{tikzcd}
      \Phi(y) \arrow{r}{l^{\ostar}}  \arrow{d}{\eta^{\ostar}} & \Phi(x)
      \arrow{d}{\xi^{\ostar}} \\
      \Phi(y') \arrow{r}{l'^{\ostar}}  \arrow{dd}{f'_{\otimes}} & \Phi(x')
      \arrow{d}{\epsilon'^{\ostar}} \\      
      & \Phi(w' \times_{z'} y') \arrow{d}{h'_{\otimes}} \\
      \Phi(z') \arrow{r}{g'^{\ostar}} & \Phi(w').
    \end{tikzcd}
  \]
  We get a corresponding identification of the mates of the outer
  squares in these two diagrams, which in turn decompose (since mate
  transformations are compatible with vertical pasting of squares)
  into the mates of the two smaller squares; according to
  \cref{rmk:dist=adj} these are the base change transformation
  $g'_{\oplus}\omega^{\ostar} \to \zeta^{\ostar}g_{\oplus}$ and the
  distributivity transformation for $(l,f)$ on the left, and the
  distributivity transformation for $(l',f')$ and the base change
  transformation $l'_{\oplus}\xi^{\ostar} \to \eta^{\ostar}l_{\oplus}$
  on the right. This identification gives precisely a commutative
  diagram
  \[
    \begin{tikzcd}
      g'_{\oplus}h'_{\otimes}\epsilon'^{\ostar}\xi^{\ostar} \arrow{d}{\sim} \arrow{r} &
      f'_{\otimes}l'_{\oplus}\xi^{\ostar} \arrow{r} &
      f'_{\otimes}\eta^{\ostar}l_{\oplus} \arrow{d}{\sim} \\
      g'_{\oplus}\omega^{\ostar}h_{\otimes}\epsilon^{\ostar} \arrow{r} &
      \zeta^{\ostar}g_{\oplus}h_{\otimes}\epsilon^{\ostar} \arrow{r} & \zeta^{\ostar}f_{\otimes}l_{\oplus},
    \end{tikzcd}
  \]
  which we can reorganize into the commutative diagram in the statement.
\end{proof}

\subsection{Symmetric monoidal structures}
In this subsection we will prove that the functor
$\SPAN \colon \Pair \to \CatIT$ preserves products. Applying this in
the special case of bispans, we will see that in some cases a
symmetric monoidal structure on $\mathcal{C}$ induces a symmetric
monoidal structure on $\BISPAN_{F,L}(\mathcal{C})$, and also show that
$\BISPAN(\blank)$ is a functor of \itcats{}. The construction of
symmetric monoidal structures on spans is also discussed in \cite[Part
III, Chapter 9]{GaitsgoryRozenblyum1} (where it is used to encode
Serre duality for Ind-coherent sheaves) and \cite{MacphersonCorr}*{\S
  3.2}.  An explicit construction (not relying on the universal
property) of symmetric monoidal structures on \icats{} of spans is
also given in \cite{BarwickMackey2}.

\begin{propn}
  The \icats{} $\Pair$ and $\Trip$ have finite products, given by
  \[ (\mathcal{C}, \mathcal{C}_{F}) \times (\mathcal{C}',
    \mathcal{C}'_{F'}) \simeq (\mathcal{C} \times \mathcal{C}',
    \mathcal{C}_{F} \times \mathcal{C}'_{F'}),\]
  \[ (\mathcal{C}, \mathcal{C}_{F}, \mathcal{C}_{L}) \times (\mathcal{C}',
    \mathcal{C}'_{F'}, \mathcal{C}'_{L'}) \simeq (\mathcal{C} \times \mathcal{C}',
    \mathcal{C}_{F} \times \mathcal{C}'_{F'}, \mathcal{C}_{L} \times \mathcal{C}'_{L'}).\]
\end{propn}
\begin{proof}
  It follows from the definition of $\Pair$ that the morphism
  \[\Map_{\Pair}((\mathcal{C}, \mathcal{C}_{F}), (\mathcal{C}',
    \mathcal{C}'_{F'})) \to \Map_{\CatI}(\mathcal{C}, \mathcal{C}') \]
  is a monomorphism (and similarly for $\Trip$), so it suffices to
  show that for any span pair $(\mathcal{D}, \mathcal{D}_{F})$ (or
  bispan triple $(\mathcal{D}, \mathcal{D}_{F}, \mathcal{D}_{L})$) a functor $\mathcal{D} \to \mathcal{C} \times
  \mathcal{C}'$ is a morphism of span pairs (or bispan triples) \IFF{}
  the functors $\mathcal{D} \to \mathcal{C}$ and $\mathcal{D} \to
  \mathcal{C}'$ are both morphisms of span pairs (or bispan
  triples). This is clear, since a pair of cartesian squares in
  $\mathcal{C}$ and $\mathcal{C}'$ gives a cartesian square in
  $\mathcal{C} \times \mathcal{C}'$ and vice versa, and similarly for distributivity diagrams.
\end{proof}

\begin{propn}\label{propn:prodladjcond}
  Suppose $(\mathcal{C},\mathcal{C}_{F})$ and $(\mathcal{C}',
  \mathcal{C}'_{F'})$ are span pairs. Then a functor $\Phi \colon
  \mathcal{C}^{\op} \times \mathcal{C}'^{\op} \to \mathcal{X}$ is left
  $(F,F')$-adjointable \IFF{} it is left adjointable in each
  variable, \ie{}
  \begin{itemize}
  \item for every $c \in \mathcal{C}$, the functor $\Phi(c,\blank)$ is
    left $F'$-adjointable,
  \item for every $c' \in \mathcal{C}'$, the functor $\Phi(\blank,c')$ is
    left $F$-adjointable,
  \item for every morphism $c_{1} \to c_{2}$ in $\mathcal{C}$, the
    transformation $\Phi(c_{2},\blank) \to \Phi(c_{1},\blank)$ is
    left $F'$-adjointable,
  \item for every morphism $c'_{1} \to c'_{2}$ in $\mathcal{C}'$, the
    transformation $\Phi(\blank, c'_{2}) \to \Phi(\blank, c'_{1})$ is
    left $F$-adjointable.
  \end{itemize}
\end{propn}
\begin{proof}
  We first observe that $\Phi$ is left $(F,F')$-\emph{preadjointable}
  \IFF{} $(f,\id_{c'})^{\ostar}$ has a left adjoint for all $f$ in $F$
  and $c'$ in $\mathcal{C}'$, and $(\id_{c}, f')^{\ostar}$ has a left adjoint
  for all $c$ in $\mathcal{C}$ and $f'$ in $F'$, since
  $(f,f')^{\ostar} \simeq (f,\id)^{\ostar}(\id,f')^{\ostar}$ and left
  adjoints compose. Thus $\Phi$ is left $(F,F')$-preadjointable
  \IFF{} $\Phi(c,\blank)$ is left $F'$-preadjointable for all $c \in
  \mathcal{C}$ and $\Phi(\blank,c')$ is left $F$-preadjointable for
  all $c'$ in $\mathcal{C}'$.

  A pair of cartesian squares gives a cartesian square in $\mathcal{C}
  \times \mathcal{C}'$, so if $\Phi$ is left $(F,F')$-preadjointable
  then it is left $(F,F')$-adjointable \IFF{} for cartesian squares
  \[
    \begin{tikzcd}
      w \arrow{r}{v}  \arrow{d}{u} & z \arrow{d}{g} \\
      x \arrow{r}{f} & y,
    \end{tikzcd}
    \qquad
    \begin{tikzcd}
      w' \arrow{r}{v'}  \arrow{d}{u'} & z' \arrow{d}{g'} \\
      x' \arrow{r}{f'} & y',
    \end{tikzcd}
  \]
  with $f$ in $\mathcal{C}_{F}$ and $f'$ in $\mathcal{C}'_{F'}$, the square
  \begin{equation}\label{eq:prodladjsq}
    \begin{tikzcd}
      \Phi(y,y') \arrow{r}{(f,f')^{\ostar}} \arrow{d}{(g,g')^{\ostar}} &
    \Phi(x,x') \arrow{d}{(u,u')^{\ostar}} \\
    \Phi(z,z') \arrow{r}{(v,v')^{\ostar}} & \Phi(w,w')
    \end{tikzcd}
  \end{equation}
  is left adjointable. Taking $f = g = \id_{c}$ this implies that
  $\Phi(c,\blank)$ is left $F'$-adjointable, while taking $f =
  \id_{y}$ we see that the transformation $\Phi(g,\blank) \colon
  \Phi(y,\blank) \to \Phi(z,\blank)$ is left $F'$-adjointable. The
  same goes in the other variable, so if $\Phi$ is left
  $(F,F')$-adjointable then the four given conditions hold.

  Conversely, if these conditions hold then we want to show that the
  square \cref{eq:prodladjsq} is left adjointable. We can decompose
  this square into the commutative diagram
  \[
    \begin{tikzcd}
      \Phi(y,y') \arrow{r}{(f,\id)^{\ostar}}
      \arrow{d}{(g,\id)^{\ostar}} & \Phi(x,y')
      \arrow{r}{(\id,f')^{\ostar}} \arrow{d}{(u,\id)^{\ostar}} &
      \Phi(x,x') \arrow{d}{(u,\id)^{\ostar}} \\
      \Phi(z,y') \arrow{r}{(v,\id)^{\ostar}} \arrow{d}{(\id,
        g')^{\ostar}} & \Phi(w,y')
      \arrow{r}{(\id,f')^{\ostar}} \arrow{d}{(\id,g')^{\ostar}} &
      \Phi(w,x') \arrow{d}{(\id, u')^{\ostar}} \\
      \Phi(z,z') \arrow{r}{(v,\id)^{\ostar}} & \Phi(w,z')
      \arrow{r}{(\id,v')^{\ostar}} & \Phi(w,w').
    \end{tikzcd}
  \]
  Here all four squares are left adjointable:
  \begin{itemize}
  \item the top left square since
    $\Phi(\blank,y')$ is left $F$-adjointable,
  \item the top right square since $\Phi(u,\blank)$ is a left
    $F'$-adjointable transformation,
  \item the bottom left square since $\Phi(\blank,g')$ is a left
    $F$-adjointable transformation,
  \item the bottom right square since $\Phi(w,\blank)$ is left $F'$-adjointable.
  \end{itemize}
  Since mate transformations are compatible with horizontal and
  vertical compositions of squares, left adjointable squares are
  closed under both horizontal and vertical compositions. Thus the
  outer square \cref{eq:prodladjsq} is left adjointable, which
  completes the proof.
\end{proof}

\begin{cor}\label{cor:ladjprodptwise}
  Suppose $(\mathcal{C},\mathcal{C}_{F})$ and $(\mathcal{C}',
  \mathcal{C}'_{F'})$ are span pairs. Then there is a natural
  equivalence
  \[ \Map_{(F,F')\dladj}(\mathcal{C}^{\op} \times \mathcal{C}'^{\op}, \mathcal{X})
    \simeq \Map_{\Fladj}(\mathcal{C}^{\op}, \FUN_{F'\dladj}(\mathcal{C}'^{\op},
    \mathcal{X}))\]
  for $\mathcal{X} \in \CatIT$.
\end{cor}

To prove this we also need the following lemma:
\begin{lemma}\label{lem:ladjinFladj}
  Suppose $(\mathcal{C}, \mathcal{C}_{F})$ is a span pair. A morphism
  $\eta \colon \phi \to \psi$ in the \itcat{} $\FUN_{\Fladj}(\mathcal{C}^{\op},
  \mathcal{X})$ has a left adjoint \IFF{}
  \begin{enumerate}[(1)]
  \item the morphism $\eta_{c} \colon \phi(c) \to \psi(c)$ has a left
    adjoint in $\mathcal{X}$ for every $c \in \mathcal{C}$,
  \item the commutative square
    \[
      \begin{tikzcd}
        \phi(c) \arrow{r}{\eta_{c}} \arrow{d} & \psi(c) \arrow{d} \\
        \phi(c') \arrow{r}{\eta_{c'}} & \psi(c')
      \end{tikzcd}
    \]
    is left adjointable for every morphism $c \to c'$ in $\mathcal{C}$.
  \end{enumerate}
  Moreover, a commutative square
  \[
    \begin{tikzcd}
      \phi  \arrow{d}{\lambda} \arrow{r}{\eta} & \psi \arrow{d}{\mu}
      \\
      \phi' \arrow{r}{\eta'} & \psi'
    \end{tikzcd}
  \]
  in $\FUN_{\Fladj}(\mathcal{C}^{\op}, \mathcal{X})$ is left adjointable
  \IFF{} the commutative square in $\mathcal{X}$ obtained by
  evaluation at $c$ is left adjointable in
  $\mathcal{X}$ for every $c \in \mathcal{C}$.
\end{lemma}
\begin{proof}
  By \cref{cor:improve} we have an equivalence
  \[\FUN_{\Fladj}(\mathcal{C}^{\op}, \mathcal{X})
    \simeq \FUN(\SPAN_{F}(\mathcal{C}), \mathcal{X}).\]
  Suppose $H \colon \Phi \to \Psi$ is the morphism corresponding to
  $\eta \colon \phi \to \psi$ under this equivalence. Then we know
  that $H$ has a left adjoint \IFF{}
  \begin{itemize}
  \item $H_{c} \colon \Phi(c) \to \Psi(c)$ has a
    left adjoint for every $c \in \mathcal{C}$,
  \item the square
    \[
      \begin{tikzcd}
        \Phi(c_{1}) \arrow{r}{H_{c_{1}}} \arrow{d} & \Psi(c_{1})
        \arrow{d} \\
        \Phi(c_{2}) \arrow{r}{H_{c_{2}}} & \Psi(c_{2})
      \end{tikzcd}
    \]
    is left adjointable for every morphism $c_{1} \to c_{2}$ in
    $\SPAN_{F}(\mathcal{C})$.
  \end{itemize}
  In terms of $\eta$, these conditions say that $\eta_{c}$ has a left
  adjoint for every $c \in \mathcal{C}$, and 
  for every span $c_{1} \xfrom{g} x \xto{f} c_{2}$ with $f$ in $F$, the outer square in the
  diagram
  \[
    \begin{tikzcd}
      \phi(c_{1}) \arrow{r}{\eta_{c_{1}}} \arrow{d}{g^{\ostar}} &
      \psi(c_{1}) \arrow{d}{g^{\ostar}} \\
      \phi(x)  \arrow{d}{f_{\oplus}} \arrow{r}{\eta_{x}} & \psi(x)
      \arrow{d}{f_{\oplus}} \\
      \phi(c_{2}) \arrow{r}{\eta_{c_{2}}} & \psi(c_{2})
    \end{tikzcd}
  \]
  is left adjointable. Since left adjointable squares compose, and
  the two squares here are those associated to spans where one leg is
  the identity, it is equivalent to require these two squares to be
  left adjointable. For the top square this is the condition we want,
  while the bottom square is automatically left adjointable since its
  mate is obtained by passing to left adjoints everywhere in the
  commutative square
  \[
    \begin{tikzcd}
    \phi(c_{2}) \arrow{d}{f^{\ostar}} \arrow{r}{\eta_{c_{2}}} &
      \psi(c_{2}) \arrow{d}{f^{\ostar}} \\
      \phi(x) \arrow{r}{\eta_{x}} & \psi(x).
    \end{tikzcd}
  \]
  Since the mate of a square of natural transformations is given by
  taking mates objectwise, the characterization of left adjointable
  squares is immediate.
\end{proof}

\begin{proof}[Proof of \cref{cor:ladjprodptwise}]
  Unpacking definitions, a functor $\Phi \colon \mathcal{C}^{\op} \times \mathcal{C}'^{\op}
  \to \mathcal{X}$ corresponds to a functor $\mathcal{C}^{\op} \to
  \Fun_{F'\dladj}(\mathcal{C}'^{\op}, \mathcal{X})$ \IFF{}
  $\Phi(c,\blank)$ is a left $F'$-adjointable functor for every $c
  \in \mathcal{C}$, and $\Phi(c_{2},\blank) \to \Phi(c_{1},\blank)$ is
  a left $F'$-adjointable transformation for every morphism $c_{1}
  \to c_{2}$ in $\mathcal{C}$.

  Moreover, it follows from \cref{lem:ladjinFladj} that such a functor \[\mathcal{C}^{\op} \to
  \Fun_{F'\dladj}(\mathcal{C}'^{\op},\mathcal{X}) \simeq
  \Fun(\SPAN_{F'}(\mathcal{C}'), \mathcal{X})\] is left
  $F$-adjointable precisely when the following conditions hold:
  \begin{itemize}
  \item for every morphism $f\colon c_{1} \to c_{2}$ in $F$ and every
    object $c' \in \mathcal{C}'$, the morphism $(f,\id)^{\ostar} \colon
    \Phi(c_{2},c') \to \Phi(c_{1},c')$ has a left adjoint,
  \item for every morphism $f\colon c_{1} \to c_{2}$ in $F$ and every
    morphism $c'_{1} \to c'_{2}$ in $\mathcal{C}'$, the commutative
    square
    \[
      \begin{tikzcd}
        \Phi(c_{2},c'_{2}) \arrow{r}{(f,\id)^{\ostar}}\arrow{d} &
        \Phi(c_{1},c'_{2}) \arrow{d} \\
        \Phi(c_{2},c'_{1}) \arrow{r}{(f,\id)^{\ostar}} & \Phi(c_{1},c'_{1})
      \end{tikzcd}
    \]
    is left adjointable,
  \item for every pullback square
    \[
      \begin{tikzcd}
        w \arrow{r}{v}  \arrow{d}{u} & z \arrow{d}{g} \\
        x \arrow{r}{f} & y
      \end{tikzcd}
    \]
    in $\mathcal{C}$ with $f$ in $F$, the commutative square
    \[
      \begin{tikzcd}
        \Phi(y,c') \arrow{d}{(g,\id)^{\ostar}}\arrow{r}{(f,\id)^{\ostar}} & \Phi(x,c') \arrow{d}{(u,\id)^{\ostar}} \\
        \Phi(z,c') \arrow{r}{(v,\id)^{\ostar}} & \Phi(w,c')
      \end{tikzcd}
    \]
    is left adjointable.
  \end{itemize}
  These conditions say precisely that $\Phi(\blank,c')$ is a left
  $F$-adjointable functor for every $c' \in \mathcal{C}'$ and
  $\Phi(\blank,c'_{2}) \to \Phi(\blank,c'_{1})$ is a left
  $F$-adjointable transformation for every morphism $c'_{1} \to
  c'_{2}$ in $\mathcal{C}'$. We have thus shown that a functor
  $\Phi \colon \mathcal{C}^{\op} \times \mathcal{C}'^{\op} \to \mathcal{X}$
  corresponds to a left $F$-adjointable functor $\mathcal{C}^{\op} \to
  \Fun_{F'\dladj}(\mathcal{C}'^{\op}, \mathcal{X})$ \IFF{} it satisfies the
  four conditions that we saw characterized left $(F,F')$-adjointable
  functors in \cref{propn:prodladjcond}.
\end{proof}

From Corollary~\ref{cor:ladjprodptwise} we can now deduce that spans preserve products:
\begin{cor}\label{cor:spanpresprod}
  Suppose $(\mathcal{C},\mathcal{C}_{F})$ and $(\mathcal{C}',
  \mathcal{C}'_{F'})$ are span pairs. Then the natural morphism
  \[ \SPAN_{(F,F')}(\mathcal{C} \times \mathcal{C}') \to
    \SPAN_{F}(\mathcal{C}) \times \SPAN_{F'}(\mathcal{C}') \]
  is an equivalence.
\end{cor}
\begin{proof}
  For $\mathcal{X}$ an \itcat{} we have natural equivalences
  \[
    \begin{split}
\Map(\SPAN_{(F,F')}(\mathcal{C} \times \mathcal{C}'),
    \mathcal{X}) & \simeq \Map_{(F,F')\dladj}(\mathcal{C}^{\op} \times
    \mathcal{C}'^{\op}, \mathcal{X})\\
    & \simeq \Map_{\Fladj}(\mathcal{C}^{\op}, \FUN_{F'\dladj}(\mathcal{C}'^{\op},
    \mathcal{X})) \\
    & \simeq \Map_{\Fladj}(\mathcal{C}^{\op}, \FUN(\SPAN_{F'}(\mathcal{C}'),
    \mathcal{X})) \\
    & \simeq \Map(\SPAN_{F}(\mathcal{C}), \FUN(\SPAN_{F'}(\mathcal{C}'),
    \mathcal{X})) \\
    & \simeq \Map(\SPAN_{F}(\mathcal{C}) \times
    \SPAN_{F'}(\mathcal{C}'), \mathcal{X}).\qedhere
  \end{split}
  \]
\end{proof}

\begin{cor}\label{cor:Spansymmmon}
  Suppose $(\mathcal{C}, \mathcal{C}_{F})$ is a span pair and
  $\mathcal{C}$ has a (symmetric) monoidal structure such that the
  tensor product functor is a morphism of span pairs
  \[ \otimes \colon (\mathcal{C} \times \mathcal{C}, \mathcal{C}_{F}
    \times \mathcal{C}_{F}) \to (\mathcal{C}, \mathcal{C}_{F}),\]
  \ie{} given morphisms $f\colon x \to y$ and $f' \colon x' \to y'$ in
  $F$, the morphism $f \otimes f' \colon x \otimes x' \to y \otimes
  y'$ is also in $F$, and given a pair of pullback squares
  \[
    \begin{tikzcd}
      w \arrow{r}{v}  \arrow{d}{u} & z \arrow{d}{g} \\
      x \arrow{r}{f} & y,
    \end{tikzcd}
    \qquad
    \begin{tikzcd}
      w' \arrow{r}{v'}  \arrow{d}{u'} & z' \arrow{d}{g'} \\
      x' \arrow{r}{f'} & y',
    \end{tikzcd}
  \]
  with $f$ and $f'$ in $\mathcal{C}_{F}$, the commutative square
  \[
    \begin{tikzcd}
    w \otimes w' \arrow{r} \arrow{d} & z \otimes z' \arrow{d} \\
    x \otimes x' \arrow{r} & y \otimes y'
  \end{tikzcd}
\]
  is cartesian. Then $\SPAN_{F}(\mathcal{C})$ inherits a (symmetric)
  monoidal structure from that on $\mathcal{C}$.
\end{cor}
\begin{proof}
  Since the functor $\SPAN$ preserves products, it takes
  (commutative) algebras in span pairs to (commutative) algebras in \itcats{}.
\end{proof}

\begin{ex}\label{ex:spancartmon}
  Suppose $(\mathcal{C},\mathcal{C}_{F})$ is a span pair where
  $\mathcal{C}$ has finite products and morphisms in $\mathcal{C}_{F}$
  are closed under products. Products of cartesian squares are always
  cartesian, so in this case \cref{cor:Spansymmmon} implies that the
  cartesian product induces a symmetric monoidal structure on
  $\SPAN_{F}(\mathcal{C})$. This recovers the discussion in
  \cite[Chapter 9, 2.1]{GaitsgoryRozenblyum1} and some cases of
  \cite[Theorem 2.15]{BarwickMackey2}.
\end{ex}

\begin{defn}
  We say an \icat{} $\mathcal{C}$ is \emph{extensive} if $\mathcal{C}$
  has finite coproducts and these satisfy \emph{descent} in the sense that the
  coproduct functor
  \[ \amalg \colon \prod_{i=1}^{n} \mathcal{C}_{/x_{i}} \to
    \mathcal{C}_{/\coprod_{i=1}^{n}x_{i}} \]
  is an equivalence. (Equivalently, pullbacks of the component
  inclusions in finite coproducts always exist, and coproducts are disjoint and
  stable under pullback.)
\end{defn}
  
\begin{ex}\label{ex:spancocart}    
  Suppose $(\mathcal{C}, \mathcal{C}_{F})$ is a span pair where
  $\mathcal{C}$ has finite coproducts and morphisms in
  $\mathcal{C}_{F}$ are closed under coproducts. If $\mathcal{C}$ is
  extensive,
  then coproducts of cartesian squares are again
    cartesian. Hence in this case the coproduct induces a symmetric
    monoidal structure on $\SPAN_{F}(\mathcal{C})$ by
    \cref{cor:Spansymmmon}. The descent condition is satisfied, for
    instance, if $\mathcal{C}$ is an $\infty$-topos, or in the
    category of sets. See also \cite{BarwickMackey}*{\S 4}, where
    $\mathcal{C}$ is called ``disjunctive'' if $\mathcal{C}$ is
    extensive and has
    pullbacks; in this case the coproduct in
    $\mathcal{C}$ gives both the product and coproduct in
    $\Span(\mathcal{C})$ by \cite{BarwickMackey}*{Proposition 4.3}.
\end{ex}

Specializing the preceding discussion to bispans, we get:

\begin{cor}\label{cor:bispan-prod} 
    Suppose $(\mathcal{C}, \mathcal{C}_{F},\mathcal{C}_{L})$ and
  $(\mathcal{C}', \mathcal{C}'_{F'}, \mathcal{C}'_{L'})$ are bispan
  triples. Then the natural morphism
  \[ \BISPAN_{(F,F'),(L,L')}(\mathcal{C} \times \mathcal{C}') \to
    \BISPAN_{F,L}(\mathcal{C}) \times \BISPAN_{F',L'}(\mathcal{C}')\]
  is an equivalence.
\end{cor}
\begin{proof}
  From \cref{cor:spanpresprod} we get a product of span pairs
  \[ (\Span_{F}(\mathcal{C})^{\op}, \mathcal{C}_{L}) \times
    (\Span_{F'}(\mathcal{C}')^{\op}, \mathcal{C}'_{L'}) \simeq
    (\Span_{(F,F')}(\mathcal{C} \times \mathcal{C}'), \mathcal{C}_{L}
    \times \mathcal{C}'_{L'}),\]
  and hence using \cref{cor:spanpresprod} again we have
  \[ \SPAN_{(L,L')}(\Span_{(F,F')}(\mathcal{C} \times
    \mathcal{C}')^{\op}) \simeq
    \SPAN_{L}(\Span_{F}(\mathcal{C})^{\op}) \times
    \SPAN_{L'}(\Span_{F'}(\mathcal{C}')^{\op}),\]
  as required.
\end{proof}

\begin{cor}\label{cor:bispanmon}
  Suppose $(\mathcal{C}, \mathcal{C}_{F}, \mathcal{C}_{L})$ is a bispan
  triple and $\mathcal{C}$ has a (symmetric) monoidal structure such
  that the tensor product is a morphism of bispan triples
  \[ \otimes \colon (\mathcal{C} \times \mathcal{C},
    \mathcal{C}_{F}\times \mathcal{C}_{F}, \mathcal{C}_{L}\times
    \mathcal{C}_{L}) \to (\mathcal{C},
    \mathcal{C}_{F},\mathcal{C}_{L}),\]
  \ie{}
  \begin{enumerate}[(1)]
  \item both $\mathcal{C}_{F}$ and $\mathcal{C}_{L}$ are closed under
    $\otimes$,
  \item given a pair of pullback squares
  \[
    \begin{tikzcd}
      w \arrow{r}{v}  \arrow{d}{u} & z \arrow{d}{g} \\
      x \arrow{r}{f} & y,
    \end{tikzcd}
    \qquad
    \begin{tikzcd}
      w' \arrow{r}{v'}  \arrow{d}{u'} & z' \arrow{d}{g'} \\
      x' \arrow{r}{f'} & y',
    \end{tikzcd}
  \]
  with $f$ and $f'$ either both in $F$ or both in $L$, the commutative
  square
    \[
    \begin{tikzcd}
    w \otimes w' \arrow{r} \arrow{d} & z \otimes z' \arrow{d} \\
    x \otimes x' \arrow{r} & y \otimes y'
  \end{tikzcd}
\]
is cartesian,
\item given morphisms $f \colon x \to y$, $f' \colon x' \to y'$ in
  $\mathcal{C}_{L}$ and $g \colon y \to z$, $g' \colon y' \to z'$, the
  diagram
  \[
  \begin{tikzcd}
    {} & \bullet \otimes \bullet \arrow{r} \arrow{dd} \arrow{dl} & \bullet
    \otimes \bullet \arrow{dd}{g_{*}f \otimes g'_{*}f'} \\
    x \otimes x' \arrow{dr}{f \otimes f'} \\
    {} & y \otimes y' \arrow{r}{g \otimes g'} & z \otimes z'
  \end{tikzcd}
  \]
  obtained as the tensor product of the distributivity diagrams for
  $(f,g)$ and $(f',g')$, is a distributivity diagram for $(f \otimes
  f', g \otimes g')$.\footnote{Note that the previous condition implies that
  the square in this diagram is cartesian; the condition can therefore
  be
  interpreted as asking for the natural map
  \[ g_{*}f \otimes g'_{*}f' \to (g \otimes g')_{*}(f \otimes f')\]
  arising from this cartesian square to be an equivalence.}
  \end{enumerate}
  Then $\BISPAN_{F,L}(\mathcal{C})$ inherits a
  (symmetric) monoidal structure from that on $\mathcal{C}$.  \qed
\end{cor}

\begin{ex}\label{ex:bispancoprod}
  Suppose $(\mathcal{C}, \mathcal{C}_{F},\mathcal{C}_{L})$ is a bispan
  triple such that $\mathcal{C}$ is extensive and  morphisms in
  $\mathcal{C}_{F}$ and $\mathcal{C}_{L}$ are closed under coproducts.
  Then the coproduct satisfies the conditions
  of \cref{cor:bispanmon}: The descent condition implies that
  coproducts of cartesian squares are cartesian, and the condition on
  distributivity diagrams amounts to asking for the natural map
  \[ g_{*}f \amalg g'_{*}f' \to (g \amalg g')_{*}(f \amalg f')\]
  to be an equivalence; this is true because by descent we have
  \[
    \begin{split}
\Map_{/z \amalg z'}(u \amalg u', g_{*}f \amalg g'_{*}f') &  \simeq
\Map_{/z}(u, g_{*}f) \times \Map_{/z'}(u', g'_{*}f') \\ & \simeq
\Map_{/y}(g^{*}u, f) \times \Map_{/y'}(g'^{*}u',f') \\
 & \simeq \Map_{/y \amalg y'}(g^{*}u \amalg g'^{*}u', f \amalg f') \\
 & \simeq \Map_{/y \amalg y'}((g \amalg g')^{*}(u \amalg u'), f \amalg
 f') \\
 & \simeq \Map_{/y \amalg y'}(u \amalg u', (g \amalg g')_{*}(f \amalg
 f')).
    \end{split}
 \]
 for an object $u \amalg u'$ over $z \amalg z'$. The coproduct
 therefore induces a symmetric monoidal structure on
 $\BISPAN_{F,L}(\mathcal{C})$.\footnote{However, the cartesian
 product in $\mathcal{C}$ does not typically give a symmetric monoidal
 structure on bispans --- we do \emph{not} in general have an
 equivalence between $g_{*}f \times g'_{*}f'$ and
 $(g \times g')_{*}(f \times f')$.}
\end{ex}

\begin{remark}\label{rmk:coprodisprodinbispan}
  Suppose $\mathcal{C}$ is a locally cartesian closed and extensive
  \icat{}. Then the symmetric monoidal structure on
  $\Bispan(\mathcal{C})$ induced by the coproduct in $\mathcal{C}$ is
  a cartesian product: we have
  \[
    \begin{split}
      \Map_{\Bispan(\mathcal{C})}(c, x \amalg y) & \simeq \{c \from
      a \to b \to x \amalg y \} \\
      & \simeq \{c \from a_{x} \amalg a_{y} \to b_{x} \amalg b_{y} \to
      x \amalg y \} \\
      & \simeq \{c \from a_{x} \to b_{x} \to x\} \times \{c \from a_{y} \to
      b_{y} \to y \} \\
      & \simeq \Map_{\Bispan(\mathcal{C})}(c, x) \times \Map_{\Bispan(\mathcal{C})}(c, y).
    \end{split}
  \]
  Moreover, we have the same identification for \icats{} of
  morphisms, so this is actually an $(\infty,2)$-categorical product.
  However, this is \emph{not} a coproduct in bispans: in particular,
  $\emptyset$ is not an initial object, since we have
  \[ \Map_{\Bispan(\mathcal{C})}(\emptyset, x) \simeq \{\emptyset
    \from \emptyset \to b \to x\} \simeq \mathcal{C}_{/x}^{\simeq}\]
  which is not in general contractible.
\end{remark}

\begin{remark}
  For any \icat{} $\mathcal{C}$ we can consider the minimal bispan
  triple $\mathcal{C}^{\flat} := (\mathcal{C}, \mathcal{C}^{\simeq},
  \mathcal{C}^{\simeq})$ where the morphisms in $\mathcal{C}_{F}$ and
  $\mathcal{C}_{L}$ are just the equivalences. Any functor gives a
  morphism of minimial bispan triples, so we have a functor
  \[ (\blank)^{\flat} \colon \CatI \to \Trip \] that is moreover fully
  faithful. The \icat{} $\Trip$ is then a $\CatI$-module via cartesian
  products with $(\blank)^{\flat}$. We also have a natural equivalence
  $\BISPAN(\mathcal{C}^{\flat}) \simeq \mathcal{C}$ as all functors
  are distributive. This means the functor
  $\BISPAN \colon \Trip \to \CatIT$ is a morphism of $\CatI$-modules,
  which we can view as a functor of \itcats{} using the recent results
  of Heine~\cite{HeineEnrMod}. Moreover, the natural transformation
  $\Span_{F}(\mathcal{C}) \to \BISPAN_{F,L}(\mathcal{C})$ is a
  transformation of $\CatI$-modules, which means the universal
  property of $\BISPAN_{F,L}(\mathcal{C})$ is actually
  $\CatI$-natural.
\end{remark}

\section{Examples of distributivity}\label{sec:ex}
\subsection{Bispans in finite sets and symmetric monoidal
  $\infty$-categories}\label{subsec:BispanFin} 
In this subsection we consider the relationship between symmetric
monoidal \icats{} and bispans in finite sets. We first recall that symmetric monoidal
\icats{} can be described in terms of functors from spans of finite
sets, and then show that the resulting functor is distributive \IFF{}
the tensor product commutes with finite coproducts in each
variable. Our universal property then gives a (product-preserving)
functor from bispans in finite sets, which we can interpret as a
semiring structure with the coproduct as addition and the tensor
product as multiplication.

\begin{notation}
  We write $\xF$ for the category of finite sets and $\xF_{*}$ for
  the category of finite pointed sets and base-point preserving maps;
  every object of $\xF_{*}$ is isomorphic to one of the form
  $\angled{n} := (\{0,\ldots,n\},0)$. For $I \in \xF$ we write $I_{+}$
  for the pointed set $(I \amalg \{*\}, *)$ obtained by adding a
  disjoint base point to $I$.
\end{notation}

\begin{defn}
  If $\mathcal{C}$ is an \icat{} with finite products, a
  \emph{commutative monoid} in $\mathcal{C}$ is a functor $\Phi \colon
  \xF_{*} \to
  \mathcal{C}$ such that for every $n=0,1,\ldots$ the map
  \[ \Phi(\angled{n}) \xto{(\Phi(\rho_{i}))_{i=1,\ldots,n}}
    \prod_{i=1}^{n} \Phi(\angled{1}) \]
  is an equivalence, where $\rho_{i} \colon \angled{n} \to \angled{1}$
  is defined by 
  \[
  \rho_{i}(j) = \begin{cases}
  $1$ & \text{if $i = j$}\\
   $0$ & \text{otherwise}.
   \end{cases}
  \] We
  write $\CMon(\mathcal{C})$ for the full subcategory of
  $\Fun(\xF_{*}, \mathcal{C})$ spanned by the commutative monoids.
\end{defn}

\begin{notation}
  If $\mathcal{C}, \mathcal{D}$ are \icats{} with finite products, we
  write $\Fun^{\times}(\mathcal{C},\mathcal{D})$ for the full
  subcategory of $\Fun(\mathcal{C},\mathcal{D})$ spanned by the
  functors that preserve finite products.
\end{notation}

\begin{remark}
  The category $\xF_{*}$ can be identified with the subcategory of
  $\Span(\xF)$ whose morphisms are the spans $I \xfrom{f} J \xto{g} K$
  where $f$ is injective, with this corresponding to the morphism
  $I_{+} \to K_{+}$ given by
  \[ i \mapsto
    \begin{cases}
      g(j), & i = f(j),\\
      *, & \txt{otherwise};
    \end{cases}
  \]
  we write $j \colon \xF_{*} \to \Span(\xF)$ for this subcategory inclusion.
\end{remark}

The following description of commutative monoids in terms of spans
seems to have been first proved by
Cranch~\cite{CranchThesis,CranchSpan}; other proofs, as special cases
of various generalizations, can be found in
\cite{GlasmanStrat,HarpazAmbi,norms}.
\begin{propn}\label{lem:cranch}
  Let $\mathcal{C}$ be an \icat{} with finite
  products. Restriction along the inclusion
  $j \colon \xF_{*} \to \Span(\xF)$
  gives an equivalence
  \[ \Fun^{\times}(\Span(\xF), \mathcal{C}) \isoto \txt{CMon}(\mathcal{C}), \]
  with the inverse given by right Kan extension along the functor $j$.
\end{propn}
\begin{proof}
  In the stated form, this is \cite[Proposition C.1]{norms}.
\end{proof}

\begin{remark}\label{rmk:SpanFftrdesc}
  The functor $\Phi \colon \Span(\xF) \to \CatI$ corresponding to a
  symmetric monoidal \icat{} $\mathcal{C}$ admits the following
  description:
  \begin{itemize}
  \item $\Phi(I) \simeq \prod_{i \in I} \mathcal{C} \simeq
    \Fun(I,\mathcal{C})$,
  \item for $f \colon I \to J$, the functor $f^{\ostar} \colon
    \Fun(J,\mathcal{C}) \to \Fun(I, \mathcal{C})$ is that given by
    composition with $f$,
  \item for $f \colon I \to J$, the functor \[
      f_{\otimes} \colon
    \Fun(I, \mathcal{C}) \simeq \prod_{j \in J} \prod_{i \in I_{j}}
    \mathcal{C} \to \prod_{j \in J} \mathcal{C}\]
  is given by tensoring the components corresponding to the preimages
  of each $j \in J$.
\end{itemize}
In particular, for $q \colon I \to *$ the functor $q_{\otimes} \colon
\Fun(I, \mathcal{C}) \to \mathcal{C}$ takes $\phi \colon I \to \mathcal{C}$ to $\bigotimes_{i
  \in I} \phi(i)$, while $q^{\ostar} \colon \mathcal{C} \to \Fun(I, \mathcal{C})$ is the diagonal functor.
\end{remark}

\begin{propn}\label{prop:preserv}
  Suppose $\Phi \colon \Span(\xF) \to \CatI$ is a product-preserving
  functor, corresponding to a symmetric monoidal structure on
  $\mathcal{C} = \Phi(*)$. Then $\Phi$ is distributive \IFF{} $\mathcal{C}$ has finite
  coproducts and the symmetric monoidal structure is compatible with
  these, \ie{} the tensor product preserves finite coproducts in
  each variable.
\end{propn}
\begin{proof}
  For $I \in \xF$, the functor $q^{\ostar} \colon \mathcal{C} \to
  \Fun(I, \mathcal{C})$ corresponding to the morphism $q \colon I \to
  *$ is the diagonal. This
  has a left adjoint \IFF{} $\mathcal{C}$ admits all $I$-indexed
  colimits, \ie{} $\mathcal{C}$ has  $I$-indexed coproducts. 
  Moreover, if $\mathcal{C}$ has finite coproducts, then the functor $f^{\ostar}
  \colon \Fun(J, \mathcal{C}) \to \Fun(I,\mathcal{C})$ given by
  composition with $f \colon I \to J$ has a left adjoint for any $f$, 
  since all pointwise left Kan extensions along $f$ exist in
  $\mathcal{C}$. Given a cartesian square
\[
  \begin{tikzcd}
    I' \arrow{r}{f'} \arrow{d}{g'} &  J' \arrow{d}{g} \\
    I \arrow{r}{f} & J
  \end{tikzcd}
\]
in $\xF$, the mate transformation
\[ g'_{\oplus}f'^{\ostar} \to f^{\ostar}g_{\oplus}\]
is then automatically an equivalence, since for $\phi \colon
J' \to \mathcal{C}$ this is given at $i \in I$ by the natural map
\[ \coprod_{x \in I'_{i}} \phi(f'x) \to \coprod_{y \in J'_{f(i)}}
  \phi(y),\]
which is an equivalence since these fibres are canonically
isomorphic. This proves that $\Phi$ is left adjointable \IFF{}
$\mathcal{C}$ admits finite coproducts.

Given morphisms $l \colon I \to J$ and $f \colon J \to K$ in $\xF$,
we have the distributivity square
\[
    \begin{tikzcd}
      {} & J' \arrow{dd} \arrow{dl}[swap]{\epsilon} \arrow{r}{f'} & K'
      \arrow{dd}{h=f_{*}l} \\
      I \arrow[swap]{dr}{l} \\
      {} & J \arrow{r}{f} & K ,
    \end{tikzcd}
\]
where $K'_{k} \cong \prod_{j \in J_{k}} I_{j}$. The distributivity
transformation $h_{\oplus}f'_{\otimes}\epsilon^{\ostar} \to
f_{\otimes}l_{\oplus}$ is given for $\phi \colon I \to \mathcal{C}$
at $k \in K$ by the canonical map
\[ \coprod_{(i_{j})_{j} \in \prod_{j \in J_{k}} I_{j}} \bigotimes_{j
    \in J_{k}} \phi(i_{j}) \to
  \bigotimes_{j \in J_{k}} \left( \coprod_{i \in I_{j}}
    \phi(i)\right). \]
This is an equivalence if $\otimes$ preserves finite coproducts in
each variable. Conversely, for $K \in \xF$ we have in particular the
distributivity diagram
  \begin{equation} \label{eq:dis-fin}
    \begin{tikzcd}
      {} & K \amalg K \arrow{dd}{q \amalg q} \arrow{dl}{\epsilon} \arrow{r}{\nabla} & K
      \arrow{dd}{q} \\
      K \amalg * \arrow[swap]{dr}{q \amalg \id} \\
      {} & * \amalg * \arrow{r}{\nabla} & * ,
    \end{tikzcd}
  \end{equation}
  where $q$ is the unique map $K \to *$ and $\nabla$ are fold
  maps. The corresponding distributivity transformation is given for
  $\alpha \colon K \amalg * \to \mathcal{C}$ by
  \[ \coprod_{k \in K} \left( \alpha(k) \otimes \alpha(*)\right) \to
    \left(\coprod_{k \in K} \alpha(k) \right) \otimes \alpha(*).\]
  If $\Phi$ is distributive then this is an equivalence for all $K$
  and $\alpha$, which is precisely the condition that $\otimes$
  preserves finite coproducts in each variable.
\end{proof}

\begin{cor}\label{cor:prod-pres}
  Product-preserving functors $\BISPAN(\xF) \to \CATI$ correspond to
  symmetric monoidal \icats{} that are compatible with finite coproducts.
\end{cor}
\begin{proof}
  By \cref{thm:main2}, functors $\BISPAN(\xF) \to
  \CATI$ correspond to distributive functors $\Span(\xF) \to
  \CATI$. Moreover, from \cref{rmk:coprodisprodinbispan} we know that
  the product in $\BISPAN(\xF)$ is given by the coproduct in $\xF$,
  just as in $\Span(\xF)$, so product-preserving functors from
  $\BISPAN(\xF)$ correspond to product-preserving distributive
  functors under this equivalence. By \cref{prop:preserv} and 
  \cref{lem:cranch}, the latter are equivalent to symmetric monoidal
  \icats{} that are compatible with finite coproducts.
\end{proof}

\subsection{Bispans in spaces and symmetric monoidal
  $\infty$-categories}\label{sec:spacesymmmon}
In this section we consider a variant of the results of the preceding
one: symmetric monoidal \icats{} can also be
described in terms of spans of spaces, and we will prove that the resulting
functor is distributive (with respect to all morphisms of spaces)
\IFF{} the tensor product commutes with colimits indexed by
\igpds{}. This applies in many examples, since most naturally
occurring tensor products are compatible with all colimits.

\begin{notation}
  We write $\mathcal{S}_{\fin}$ for the subcategory of $\mathcal{S}$
  containing only the morphisms whose fibres are equivalent to finite
  sets. Then $(\mathcal{S}, \mathcal{S}_{\fin})$ is a span pair.
\end{notation}

\begin{remark}
  If $f \colon X \to I$ is a morphism in $\mathcal{S}_{\fin}$ and $I$
  is a finite set, then the straightening equivalence
  \[ \colim_{I}\colon \Fun(I, \mathcal{S}) \isoto \mathcal{S}_{/I} \]
  implies that $X$ is an $I$-indexed coproduct of finite sets, and so
  is itself a finite set. It follows that the functor $\Span(\xF) \to
  \Span_{\fin}(\mathcal{S})$ induced by the morphism of span pairs
  $(\xF,\xF) \to (\mathcal{S},\mathcal{S}_{\fin})$ is fully faithful.
\end{remark}

\begin{propn}\label{prop:bh}
  Let $\mathcal{C}$ be a complete \icat{}.
  Right Kan extension along the fully faithful functor $\Span(\xF) \to
  \Span_{\fin}(\mathcal{S})$ identifies $\Fun^{\times}(\Span(\xF),
  \mathcal{C})$ with the full subcategory
  $\Fun^{\txt{RKE}}(\Span_{\fin}(\mathcal{S}), \mathcal{C})$ of
  $\Fun(\Span_{\fin}(\mathcal{S}), \mathcal{C})$ spanned by functors
  $\Phi$ such that $\Phi|_{\mathcal{S}^{\op}}$ is right Kan extended
  from $\{*\}$.
\end{propn}

\begin{proof} 
  This is a special case of \cite[Proposition C.18]{norms}.%
\end{proof}

Combining this with \cref{lem:cranch}, we have:
\begin{cor}\label{cor:RKESpanfinCMon}
  Let $\mathcal{C}$ be a complete \icat{}. There is an equivalence
\[ \CMon(\mathcal{C}) \isoto
  \Fun^{\txt{RKE}}(\Span_{\fin}(\mathcal{S}), \mathcal{C}), \]
given by right Kan extension along $\xF_{*} \to \Span(\xF)
\hookrightarrow \Span_{\fin}(\mathcal{S})$.\qed
\end{cor}
\begin{remark}
  In particular, symmetric monoidal \icats{} can be identified with
  functors $\Span_{\fin}(\mathcal{S}) \to \CatI$ whose restrictions to
  $\mathcal{S}^{\op}$ preserve limits. If
  $\mathcal{C}$ is a symmetric monoidal \icat{}, the corresponding
  functor $\Phi \colon \Span_{\fin}(\mathcal{S}) \to \CatI$ admits the
  following description (analogous to \cref{rmk:SpanFftrdesc}):
  \begin{itemize}
  \item for $X \in \mathcal{S}$, $\Phi(X) \simeq \lim_{x \in X}
    \Phi(\{x\}) \simeq \Fun(X, \mathcal{C})$,
  \item for $f \colon X \to Y$ in $\mathcal{S}$, the functor
    $f^{\ostar} \colon \Fun(Y, \mathcal{C}) \to \Fun(X, \mathcal{C})$
    is given by composition with $f$,
  \item for $g \colon E \to B$ in $\mathcal{S}_{\fin}$, the functor
    $g_{\otimes} \colon \Fun(E, \mathcal{C}) \to \Fun(B,\mathcal{C})$
    is given by tensoring fibrewise along $g$, \ie{} for $\phi \colon
    E \to B$ we have
    \[ (g_{\otimes} \phi)(b) \simeq \bigotimes_{e \in E_{b}} \phi(e).\]
  \end{itemize}
  In particular, for $q \colon X \to *$ the functor $q^{\ostar}$ is
  the diagonal functor, while if $X$ is a finite set the functor
  $q_{\otimes}$ is the $X$-indexed tensor product.
\end{remark}

We will identify when such functors from $\Span_{\fin}(\mathcal{S})$
are distributive with respect to bispan triples of the following form:
\begin{lemma}\label{prop:allow}
  Let $\mathcal{K}$ be a full subcategory of $\mathcal{S}$ with the
  following properties:
  \begin{itemize}
  \item if $p \colon E \to B$ is a morphism in $\mathcal{S}$ such that
    $B \in \mathcal{K}$ and $E_{b} \in \mathcal{K}$ for all $b \in B$,
    then $E \in \mathcal{K}$,
  \item $\mathcal{K}$ is closed under finite products.
  \end{itemize}
  This implies that morphisms in $\mathcal{S}$ whose fibres lie in $\mathcal{K}$
  are closed under composition, giving a subcategory
  $\mathcal{S}_{\mathcal{K}}$ of $\mathcal{S}$.
  Then $(\mathcal{S}, \mathcal{S}_{\fin}, \mathcal{S}_{\mathcal{K}})$
  is a bispan triple.
\end{lemma}

\begin{proof}
  Suppose $f \colon X \to Y$ and $g \colon Y \to Z$ are morphisms whose
  fibres lie in $\mathcal{K}$. We have a morphism $X_{z} \to Y_{z}$
  between the fibres of $gf$ and $g$ at $z \in Z$, whose fibre at $y
  \in Y_{z}$ is equivalent to $X_{y}$. Since $Y_{z}$ and $X_{y}$ lie
  in $\mathcal{K}$ for all $z \in Z, y \in Y$, it follows that $X_{z}$
  also lies in $\mathcal{K}$. Thus we do indeed have a subcategory
  $\mathcal{S}_{\mathcal{K}}$ of morphisms whose fibres lie in
  $\mathcal{K}$. Such morphisms are obviously preserved under base
  change, and so $(\mathcal{S},\mathcal{S}_{\mathcal{K}})$ is a span pair.

  All distributivity diagrams exist in $\mathcal{S}$ since this is a
  locally cartesian closed \icat{}; see Remark~\ref{rem:lcc}. To show that $(\mathcal{S},
  \mathcal{S}_{\fin}, \mathcal{S}_{\mathcal{K}})$ is a bispan triple
  it therefore only remains to check that if $l \colon X \to Y$ is a
  morphism in $\mathcal{S}_{\mathcal{K}}$ and $f \colon Y \to Z$ is a
  morphism in $\mathcal{S}_{\fin}$, then $f_{*}l$ is also a morphism
  in $\mathcal{S}_{\mathcal{K}}$. But we have
  \[ (f_{*}l)_{z} \simeq \prod_{y \in Y_{z}} X_{y}, \] which is a
  finite product of objects of $\mathcal{K}$, and so again lies in
  $\mathcal{K}$ by assumption.
\end{proof}

\begin{exs}
  We can take $\mathcal{K}$ in \cref{prop:allow} to consist of
  \begin{itemize}
  \item finite sets,
  \item all spaces,
 \item $\pi$-finite spaces\footnote{These are the spaces $X$ such that (1)
     $X$ is $n$-truncated for some $n$, (2) $\pi_0(X)$ is finite, and
     (3) for each $x \in X$, the homotopy group $\pi_k(X,x)$ is finite
     for each $k \geq 1$.}, as follows by examining the long
   exact sequence in homotopy groups associated to a fibre sequence, 
 \item $\kappa$-compact spaces\footnote{Meaning $\kappa$-compact
     objects of the \icat{} $\mathcal{S}$ of spaces.} for any regular
   cardinal $\kappa$, since  $\kappa$-filtered colimits in
   $\mathcal{S}$ commute with $\kappa$-small
   limits, and the $\kappa$-compact spaces are precisely the (retracts
   of) $\kappa$-small $\infty$-groupoids.
  \end{itemize}
\end{exs}

\begin{propn}\label{propn:SpanKdist}
  Suppose $\Phi \colon \Span_{\fin}(\mathcal{S}) \to \CatI$
  corresponds to a symmetric monoidal structure on
  $\mathcal{C} = \Phi(*)$, and let $\mathcal{K}$ be as in
  \cref{prop:allow}. Then $\Phi$ is $\mathcal{K}$-distributive \IFF{}
  $\mathcal{C}$ has $\mathcal{K}$-indexed colimits (meaning
  $K$-indexed colimits for all $K \in \mathcal{K}$),
  and the tensor product on $\mathcal{C}$ is compatible with such
  colimits (\ie{} preserves them in each variable).
\end{propn}
\begin{proof}
  For $K \in \mathcal{K}$, the functor $q^{\ostar} \colon \mathcal{C} \to
  \Fun(K, \mathcal{C})$ corresponding to the morphism $q \colon K \to
  *$ is the diagonal. This
  has a left adjoint \IFF{} $\mathcal{C}$ admits all $K$-indexed
  colimits. Moreover, if $\mathcal{C}$ has $\mathcal{K}$-indexed
  colimits then the functor $f^{\ostar}
  \colon \Fun(Y, \mathcal{C}) \to \Fun(X,\mathcal{C})$ given by
  composition with $f \colon X \to Y$ has a left adjoint for any $f$
  in $\mathcal{S}_{\mathcal{K}}$, 
  since all pointwise left Kan extensions along $f$ then exist in
  $\mathcal{C}$. Given a cartesian square
\[
  \begin{tikzcd}
    X' \arrow{r}{f'} \arrow{d}{g'} &  Y' \arrow{d}{g} \\
    X \arrow{r}{f} & Y
  \end{tikzcd}
\]
in $\mathcal{S}$ with $g$ in $\mathcal{S}_{\mathcal{K}}$, the mate transformation
\[ g'_{\oplus}f'^{\ostar} \to f^{\ostar}g_{\oplus}\]
is then automatically an equivalence, since for $\phi \colon
Y' \to \mathcal{C}$ this is given at $x \in X$ by the natural map
\[ \colim_{p \in X'_{x}} \phi(f'p) \to \colim_{y \in Y'_{f(y)}}
  \phi(y),\]
which is an equivalence since these fibres are canonically
equivalent. This proves that $\Phi$ is left adjointable \IFF{}
$\mathcal{C}$ admits $\mathcal{K}$-indexed colimits.

Given morphisms $l \colon X \to Y$ in $\mathcal{S}_{\mathcal{K}}$ and $f \colon Y \to Z$ in $\mathcal{S}_{\fin}$,
we have the distributivity diagram
\[
    \begin{tikzcd}
      {} & Y' \arrow{dd} \arrow{dl}[swap]{\epsilon} \arrow{r}{f'} & Z'
      \arrow{dd}{h=f_{*}l} \\
      X \arrow[swap]{dr}{l} \\
      {} & Y \arrow{r}{f} & Z ,
    \end{tikzcd}
\]
where $Z'_{z} \simeq \prod_{y \in Y_{z}} X_{y}$. The distributivity
transformation $h_{\oplus}f'_{\otimes}\epsilon^{\ostar} \to
f_{\otimes}l_{\oplus}$ is given for $\phi \colon X \to \mathcal{C}$
at $z \in Z$ by the canonical map
\[ \colim_{(x_{y})_{y} \in \prod_{y \in Y_{z}} X_{y}} \bigotimes_{y
    \in Y_{z}} \phi(x_{y}) \to
  \bigotimes_{y \in Y_{z}} \left( \colim_{x \in X_{y}}
    \phi(x)\right). \]
This is an equivalence if $\otimes$ preserves $\mathcal{K}$-indexed colimits in
each variable. Conversely, for $K \in \mathcal{K}$ we have in particular the
distributivity diagram
\[
    \begin{tikzcd}
      {} & K \amalg K \arrow{dd}{q \amalg q} \arrow{dl}[swap]{\epsilon} \arrow{r}{\nabla} & K
      \arrow{dd}{q} \\
      K \amalg * \arrow[swap]{dr}{q \amalg \id} \\
      {} & * \amalg * \arrow{r}{\nabla} & * ,
    \end{tikzcd}
\]
  where $q$ is the unique map $K \to *$ and $\nabla$ are fold
  maps. The corresponding distributivity transformation is given for
  $\phi \colon K \amalg * \to \mathcal{C}$ by
  \[ \colim_{k \in K} \left( \phi(k) \otimes \phi(*)\right) \to
    \left(\colim_{k \in K} \phi(k) \right) \otimes \phi(*).\]
  If $\Phi$ is distributive then this is an equivalence for all $K \in
  \mathcal{K}$
  and $\phi$, which is precisely the condition that $\otimes$
  preserves $\mathcal{K}$-indexed colimits in each variable.
\end{proof}

\begin{cor}\label{cor:BISPANKSsymmmon}
  Let $\mathcal{K}$ be as in \cref{prop:allow}. Then functors
  $\Phi \colon \BISPAN_{\fin,\mathcal{K}}(\mathcal{S}) \to \CATI$ such
  that the restriction to $\mathcal{S}^{\op}$ preserves limits
  correspond to symmetric monoidal \icats{} that are compatible with
  $\mathcal{K}$-indexed colimits.
\end{cor}
\begin{proof}
  By \cref{thm:main2}, functors $\BISPAN_{\fin,\mathcal{K}} \to
  \CATI$ correspond to $\mathcal{K}$-distributive functors $\Span_{\fin}(\mathcal{S}) \to
  \CATI$. On the other hand, we know from \cref{cor:RKESpanfinCMon}
  that such functors whose restriction to $\mathcal{S}^{\op}$ is right
  Kan extended from $\{*\}$ correspond to symmetric monoidal \icats{},
  and in this case the functor is $\mathcal{K}$-distributive \IFF{}
  the tensor product is compatible with $\mathcal{K}$-indexed colimits
  by \cref{propn:SpanKdist}.
\end{proof}

\subsection{Bispans in spaces and analytic monads}\label{subsec:anmnd}
Our goal in this section is to relate bispans in the \icat{} of spaces
to the polynomial and analytic functors studied in \cite{polynomial},
where it is shown that analytic monads are equivalent to
$\infty$-operads. Using the results of the previous
section we will see that there is a canonical action of every
analytic monad on any symmetric monoidal \icat{} compatible with
\igpd{}-indexed colimits.

We first give a general construction of a functor from bispans to
\icats{} using slice \icats{}:
\begin{propn}\label{propn:bispanslice}
  Suppose $(\mathcal{C}, \mathcal{C}_{F},
  \mathcal{C}_{L})$ is a bispan triple. Then there is a functor of \itcats{}
  $\txt{Sl} \colon \BISPAN_{F,L}(\mathcal{C}) \to \CATI$ such that
  \begin{itemize}
  \item $\txt{Sl}(c) \simeq \mathcal{C}_{/c}^{L}$, the full
    subcategory of $\mathcal{C}_{/c}$ spanned by the morphisms to $c$
    that lie in $\mathcal{C}_{L}$.
  \item for $f \colon c \to c'$ in $\mathcal{C}$, the functor
    $f^{\ostar}$ is the functor $f^{*} \colon \mathcal{C}_{/c'}^{L} \to
    \mathcal{C}_{/c}^{L}$ given by pullback along $f$,
  \item for $f \colon c \to c'$ in $\mathcal{C}_{L}$, the functor
    $f_{\oplus}$ is the functor $f_{!} \colon \mathcal{C}_{/c}^{L} \to
    \mathcal{C}_{/c'}^{L}$ given by composition with $f$,
  \item for $f \colon c \to c'$ in $\mathcal{C}_{F}$, the functor
    $f_{\otimes}$ is the functor $f_{*}\colon \mathcal{C}_{/c}^{L} \to
    \mathcal{C}_{/c'}^{L}$ given by the partial right adjoint to $f^{*}
    \colon \mathcal{C}_{/c'} \to \mathcal{C}_{/c}$. 
  \end{itemize}
\end{propn}
\begin{proof}
  Let $\mathcal{X}$ be the full subcategory of
  $\mathcal{C}^{\Delta^{1}}$ spanned by the morphisms in
  $\mathcal{C}_{L}$. Then the restriction of
  $\txt{ev}_{1} \colon \mathcal{C}^{\Delta^{1}} \to \mathcal{C}$ to a
  functor $\mathcal{X} \to \mathcal{C}$ is a cartesian fibration, with
  cocartesian morphisms over morphisms in
  $\mathcal{C}_{L} \subseteq \mathcal{C}$. This corresponds to a
  functor $\lambda \colon \mathcal{C}^{\op} \to \CatI$, which takes
  $c \in \mathcal{C}$ to the \icat{} $\mathcal{C}^{L}_{/c}$ and
  $f \colon x \to y$ to the functor
  $f^{*} \colon \mathcal{C}^{L}_{/y} \to \mathcal{C}^{L}_{/x}$ given
  by pullback along $f$. The functor $\lambda$ is right
  $F$-adjointable: since $(\mathcal{C}, \mathcal{C}_{F},
  \mathcal{C}_{L})$ is a bispan triple, the functor $f^{*}$ for $f
  \colon x \to y$ in
  $\mathcal{C}_{F}$ has a right adjoint $f_{*} \colon
  \mathcal{C}^{L}_{/x}\to \mathcal{C}^{L}_{/y}$ by \cref{dist-adj},
  and given a cartesian square
  \[
    \begin{tikzcd}
      x' \arrow{r}{f'} \arrow{d}{\xi} & y' \arrow{d}{\eta} \\
      x \arrow{r}{f} & y
    \end{tikzcd}
  \]
  with $f$ and $f'$ in $\mathcal{C}_{F}$, the mate transformation
  \[ \eta^{*}f_{*} \to f'_{*}\xi^{*} \] is an equivalence by
  \cref{lem:bispantripmatecond}.  The functor $\lambda$ therefore
  extends canonically to a functor
  \[ \Lambda \colon \SPAN_{F}(\mathcal{C})^{\twop} \to \CATI \]
  by \cref{thm:spanuniv}. We claim the underlying functor of \icats{} $\lambda'\colon
  \Span_{F}(\mathcal{C}) \to \CatI$ is
  $L$-distributive. Certainly if $l \colon x \to y$ is a morphism in
  $\mathcal{C}_{L}$, then the pullback functor $l^{*} \colon
  \mathcal{C}^{L}_{/y}\to \mathcal{C}^{L}_{/x}$ has a left adjoint
  $l_{!}$ given by composition with $l$; to see that the restriction
  of $\lambda'$ to $\mathcal{C}^{\op}$ is left $L$-adjointable
  it remains to observe that for any cartesian square
  \[
    \begin{tikzcd}
      x' \arrow{r}{l'} \arrow{d}[swap]{\xi} & y' \arrow{d}{\eta} \\
      x \arrow{r}{l} & y
    \end{tikzcd}
  \]
  with $l$ and $l'$ in $\mathcal{C}_{L}$, the natural transformation
  $l'_{!}\xi^{*} \to \eta^{*}l_{!}$ is an equivalence, since for $g
  \colon z \to x$ in $\mathcal{C}^{L}_{/x}$ in the diagram
  \[
    \begin{tikzcd}
      z' \arrow{r}{g'} \arrow{d}{\zeta} & x' \arrow{r}{l'}
      \arrow{d}{\xi} & y' \arrow{d}{\eta} \\
      z \arrow{r}{g} & y \arrow{r}{l} & x, 
    \end{tikzcd}
  \]
  the left square is cartesian \IFF{} the composite square is
  cartesian. To see that $\lambda'$ is also $L$-distributive, consider
  $l \colon x \to y$ in $\mathcal{C}_{L}$ and $f \colon y \to z$ in
  $\mathcal{C}_{F}$ and form a distributivity diagram
  \cref{eq:dis-sq}. The distributivity transformation
  \[ g_{!}\tilde{f}_{*}\epsilon^{*} \to f_{*}l_{!}\]
  evaluated at $l' \colon c \to x$ is a canonical map
  \[ f_{*}l \circ \tilde{f}_{*}\epsilon^{*}l \to f_{*}(l \circ l');\]
  this is an equivalence by \cref{lem:lcompdist}. It follows by
  \cref{thm:main2} that $\lambda'$ extends uniquely to a
  functor $\txt{Sl} \colon \BISPAN_{F,L}(\mathcal{C}) \to \CATI$,
  which by construction has the required properties.
\end{proof}

Applying this to the bispan triple
$(\mathcal{S},\mathcal{S},\mathcal{S})$ we get in particular:
\begin{cor}
  There is a functor $\txt{Sl} \colon \BISPAN(\mathcal{S}) \to \CATI$ taking
  $X \in \mathcal{S}$ to $\mathcal{S}_{/X}$ and a bispan
  \[ X \xfrom{s} E \xto{p} B \xto{t} Y\]
  to the functor $t_{!}p_{*}s^{*} \colon \mathcal{S}_{/X} \to
  \mathcal{S}_{/Y}$, where $s^{*}$ is given by pullback along $s$,
  $p_{*}$ is the right adjoint to $p^{*}$, and $t_{!}$ is given by
  composition with $t$. \qed
\end{cor}

\begin{defn}\label{defn:Spoly}
  A \emph{polynomial functor}
  $F \colon \mathcal{S}_{/X} \to \mathcal{S}_{/Y}$ is an accessible
  functor that preserves weakly contractible limits. By \cite[Theorem
  2.2.3]{polynomial} the polynomial functors are equivalently those
  functors obtained as composites $t_{!}p_{*}s^{*}$ for some bispan of
  spaces
  \[ X \xfrom{s} E \xto{p} B \xto{t} Y.\] Let
  $\PolyFun(\mathcal{S})$ be the sub-$(\infty,2)$-category of
  $\CATI$ whose objects are the slices $\mathcal{S}_{/X}$ for
  $X \in \mathcal{S}$, whose 1-morphisms are the polynomial functors,
  and whose 2-morphisms are the cartesian natural transformations.
\end{defn}
\begin{remark}
  The \itcat{} $\PolyFun(\mathcal{S})$ is the underlying \itcat{}
  of the double \icat{} of polynomial functors considered in
  \cite{polynomial}.
\end{remark}

\begin{cor}\label{cor:BispanPoly}
  The functor $\txt{Sl} \colon \BISPAN(\mathcal{S}) \to \CATI$
  restricts to an
  equivalence
  \[\BISPAN(\mathcal{S}) \isoto \PolyFun(\mathcal{S}).\]
\end{cor}
\begin{proof}
  We apply the description from \cref{propn:distftrbispandesc} to
  understand $\txt{Sl}$: On objects, $\txt{Sl}$ takes $X \in \mathcal{S}$ to
  the \icat{} $\mathcal{S}_{/X}$, and on morphisms it takes the bispan
  \[ X \xfrom{s} E \xto{p} B \xto{t} Y\]
  to the functor $t_{!}p_{*}s^{*}$. The functors of this form are
  precisely the polynomial functors, by \cite[Theorem
  2.2.3]{polynomial}. The description of $\txt{Sl}$ on 2-morphisms
  from \cref{propn:distftrbispandesc}(3) implies that they are sent to
  composites of equivalences and (co)unit transformations that are
  cartesian by \cite[Lemma 2.1.5]{polynomial}. Hence $\txt{Sl}$
  factors through $\PolyFun(\mathcal{S})$, and is essentially
  surjective on objects and morphisms. To show that $\txt{Sl}$ factors
  through an
  equivalence it then suffices to show it gives an equivalence
  \[ \MAP_{\BISPAN(\mathcal{S})}(X,Y) \to
    \MAP_{\PolyFun(\mathcal{S})}(X,Y) \]
  on mapping \icats{} for all $X,Y$ in $\mathcal{S}$. Using
  \cref{cor:bc} we can identify this with the functor shown to be an
  equivalence in \cite[Proposition 2.4.13]{polynomial}.
\end{proof}

\begin{remark}
  We expect that there is an analogue of Corollary~\ref{cor:BispanPoly}  for
  bispans in any $\infty$-topos $\mathcal{X}$, but this requires
  working with \emph{internal} \icats{} in $\mathcal{X}$ (or
  equivalently sheaves of \icats{} on $\mathcal{X}$): A key step in
  the identification of polynomial functors with bispans is the
  description of colimit-preserving functors between slices of
  $\mathcal{S}$ as spans, \ie{} the equivalence
  \[ \Fun^{L}(\mathcal{S}_{/X}, \mathcal{S}_{/Y}) \simeq \Fun(X,
    \mathcal{S}_{/Y}) \simeq \Fun(X \times Y, \mathcal{S}) \simeq
    \mathcal{S}_{/X \times Y}. \]
  This certainly fails for any other $\infty$-topos $\mathcal{X}$, but
  an analogous statement should hold if we view $\mathcal{X}$ instead
  as an \icat{} internal to itself.
\end{remark}

\begin{defn}
  An \emph{analytic functor}
  $F \colon \mathcal{S}_{/X} \to \mathcal{S}_{/Y}$ is a functor that
  preserves sifted colimits and weakly contractible limits. By
  \cite[Proposition 3.1.9]{polynomial} the analytic functors are
  equivalently those functors
  obtained as composites $t_{!}p_{*}s^{*}$ for some bispan of spaces
  \[ X \xfrom{s} E \xto{p} B \xto{t} Y,\]
  where $p$ has finite discrete fibres.
\end{defn}

As a consequence, we get:
\begin{cor}\label{cor:BispanAnal}
  The functor $\txt{Sl} \colon \BISPAN(\mathcal{S}) \to \CATI$ factors through an
  equivalence
  \begin{equation}
    \label{eq:AnFuneq}
  \BISPAN_{\fin}(\mathcal{S}) \isoto \AnFun(\mathcal{S}),    
  \end{equation}
  where $\AnFun(\mathcal{S})$ is the locally full sub-\itcat{} of
  $\PolyFun(\mathcal{S})$ containing all objects, with the analytic
  functors as morphisms, as well as all 2-morphisms between these. \qed
\end{cor}

\begin{defn}
  An \emph{analytic monad} is a monad in the \itcat{}
  $\AnFun(\mathcal{S})$, \ie{} an associative algebra in the monoidal
  \icat{} $\MAP_{\AnFun(\mathcal{S})}(X,X)$ of endomorphisms of some
  object $X$, or a functor $\fmnd \to \AnFun(\mathcal{S})$, where
  $\fmnd$ is the universal 2-category containing a monad. In other
  words, it is a monad on the \icat{} $\mathcal{S}_{/X}$ whose
  underlying endofunctor is analytic and whose unit and multiplication
  transformations are cartesian. From the equivalence
  \cref{eq:AnFuneq} we see that analytic monads are equivalently
  monads in the \itcat{} $\BISPAN_{\fin}(\mathcal{S})$.
\end{defn}

\begin{cor}
  Suppose $T$ is an analytic monad on $\mathcal{S}_{/X}$ and
  $\mathcal{V}$ is a symmetric monoidal \icat{} compatible with
  \igpd{}-indexed colimits. Then $T$ induces a canonical monad
  $T_{\mathcal{V}}$ on $\Fun(X, \mathcal{V})$.
\end{cor}
\begin{proof}
  We can identify $T$ with a monad on $X \in
  \BISPAN_{\fin}(\mathcal{S})$. By \cref{cor:BISPANKSsymmmon}
  $\mathcal{V}$ induces a functor $\BISPAN_{\fin}(\mathcal{S}) \to
  \CATI$ that takes $Y \in \mathcal{S}$ to $\Fun(Y, \mathcal{V})$.
  Any functor of \itcats{} preserves monads, since they can be
  described as simply functors of \itcats{} from $\fmnd$. Hence under the
  functor induced by $\mathcal{V}$ the monad $T$ maps to a monad in
  $\CATI$ which indeed acts on $\Fun(X, \mathcal{V})$.
\end{proof}

\begin{remark}
  Suppose the underlying bispan of the monad $T$ is
  \[ X \xfrom{s} E \xto{p} B \xto{t} X.\]
  Then the underlying endofunctor of the monad $T_{\mathcal{V}}$ is
  given by
  \[ (T_{\mathcal{V}}\phi)(x) \simeq \colim_{b \in B_{x}}
    \bigotimes_{e \in E_{b}} \phi(s(e)).\]
  This has the same form as the formula for the free algebra monad of
  an \iopd{}, and the main result of \cite{polynomial} is that
  analytic monads are equivalent to \iopds{} in the form of dendroidal
  Segal spaces. We therefore expect that if $\mathcal{O}$ is the
  \iopd{} corresponding to $T$, then the monad $T_{\mathcal{V}}$
  is the free $\mathcal{O}$-algebra monad for $\mathcal{O}$-algebras
  in $\mathcal{V}$.
\end{remark}

\subsection{Equivariant bispans and $G$-symmetric monoidal
  $\infty$-categories}\label{subsec:equivt}
In this section we look at the $G$-equivariant version of our results
from \S\ref{subsec:BispanFin} on symmetric monoidal \icats{}
compatible with finite coproducts, where $G$ is a finite group: we
replace the category of finite sets by the category $\xF_{G}$ of
finite $G$-sets, and consider when a $G$-symmetric monoidal \icat{},
defined as a product-preserving functor $\Span(\xF_{G}) \to \CatI$, is
distributive, and so extends to a functor $\BISPAN(\xF_{G}) \to
\CATI$.

\begin{defn}
  Let $G$ be a finite group, and $BG$ the corresponding 1-object groupoid.
  We write $\xF_{G}$ for the category $\Fun(BG, \xF)$ of finite $G$-sets, and
  $\mathcal{O}_{G}$ for the full subcategory of \emph{orbits}, \ie{}
  finite $G$-sets of the form $G/H$ where $H$ is a subgroup of
  $G$. Then $\xF_{G}$ is obtained from $\mathcal{O}_{G}$ by freely
  adding finite coproducts, so that for any \icat{} $\mathcal{C}$ with
  finite products, restriction along the inclusion $\mathcal{O}_{G}
  \hookrightarrow \xF_{G}$ gives an equivalence
  \[ \Fun^{\times}(\xF_{G}^{\op}, \mathcal{C}) \isoto
    \Fun(\mathcal{O}_{G}^{\op}, \mathcal{C}). \] If $\mathcal{C}$ is
  the \icat{} of spaces, this says that the \icat{}
  $\mathcal{S}_{G} := \mathcal{P}(\mathcal{O}_{G})$ of
  \emph{$G$-spaces}\footnote{By Elmendorf's theorem
    \cite{ElmendorfGSpace} the \icat{} $\mathcal{S}_{G}$ is equivalent
    to that obtained from the category of topological spaces with
    $G$-action by inverting the maps that give weak homotopy
    equivalences on all spaces of fixed points.} is equivalent to
  $\Fun^{\times}(\xF_{G}^{\op}, \mathcal{S})$. By analogy with the
  case of $G$-spaces, we can think of a functor
  $\mathcal{F} \colon \mathcal{O}_{G}^{\op} \to \mathcal{C}$ as an
  ``object of $\mathcal{C}$ with $G$-action'', with
  $\mathcal{F}^{H}:= \mathcal{F}(G/H)$ the object of ``$H$-fixed
  points'' of $\mathcal{F}$. We will in particular apply this notation
  for functors $\mathcal{O}_{G}^{\op} \to \CatI$, which we will call
  \emph{$G$-\icats{}}.
\end{defn}

\begin{remark}
  The category $\xF_{G}$ is extensive, and
  so by \cite{BarwickMackey}*{Proposition 4.3} the coproduct in
  $\xF_{G}$ gives both the product and coproduct in $\Span(\xF_{G})$.
\end{remark}

\begin{defn}
  Let $\mathcal{C}$ be  an \icat{} with finite products. A
  \emph{$G$-commutative monoid} in $\mathcal{C}$ is a
  product-preserving functor $\Span(\xF_{G}) \to \mathcal{C}$. We
  write
  \[\CMon_{G}(\mathcal{C}) := \Fun^{\times}(\Span(\xF_{G}),
    \mathcal{C})\] for the \icat{} of $G$-commutative monoids in
  $\mathcal{C}$. A \emph{$G$-symmetric monoidal \icat{}} is a
  $G$-commutative monoid in $\CatI$.
\end{defn}

\begin{remark}
  When $G$ is the trivial group this is equivalent to the usual
  definition of commutative monoids (in terms of functors from
  $\xF_{*}$ satisfying a Segal condition) by \cref{lem:cranch}. More
  generally, see \cite{NardinThesis}*{Theorem 6.5} for an alternative
  description of $G$-commutative monoids in terms of ``finite pointed
  $G$-sets'' (where this must be read in a non-trivial parametrized
  sense).
\end{remark}

\begin{remark}\label{rem:tr-mult}
  A functor $\mathcal{F} \colon \Span(\xF_{G}) \to \mathcal{C}$ preserves products \IFF{}
  the restriction to $\xF_{G}^{\op} \to \mathcal{C}$ preserves
  products, and so is determined by its restriction to
  $\mathcal{O}_{G}^{\op} \to \mathcal{C}$. The additional structure
  given by the forwards maps in $\Span(\xF_{G})$ can be decomposed
  into
  \begin{itemize}
  \item multiplication maps $\mathcal{F}^{H} \times \mathcal{F}^{H}
    \to \mathcal{F}^{H}$ for each subgroup $H$ of $G$, coming from the
    fold map $G/H \amalg G/H \to G/H$, 
  \item multiplicative \emph{transfer} maps $\mathcal{F}^{H} \to
    \mathcal{F}^{K}$ for each inclusion $H \subseteq K$ of subgroups, coming from the
    quotient map $G/H \to G/K$,
  \end{itemize}
  together with various homotopy-coherent compatibilities that in
  particular make each $\mathcal{F}^{H}$ a commutative monoid.
\end{remark}

\begin{remark}
  Grouplike $G$-commutative monoids in $\mathcal{S}$ can be identified with
  connective genuine $G$-spectra, by \cite{NardinThesis}*{Corollary
    A.4.1}. Applying $\pi_{0}$, such a grouplike $G$-commutative
  monoid induces a $G$-commutative monoid in $\Set$, which factors
  through a product-preserving functor
  \[ \Span(\xF_{G}) \to \Ab \]
  because it is grouplike --- this is precisely a \emph{Mackey
    functor}, which is well-known as the structure appearing as
  $\pi_{0}$ of a genuine $G$-spectrum.
\end{remark}

For a functor $\Span(\xF_{G}) \to \CatI$ we can simplify the condition
that it is left adjointable as follows:
\begin{propn}\label{Gcatleftadjble}
  Suppose $\mathcal{F} \colon \xF_{G}^{\op} \to \CatI$ is a
  product-preserving functor. Then $\mathcal{F}$ is left adjointable \IFF{}
  \begin{enumerate}[(1)]
  \item for every subgroup $H \subseteq G$, the \icat{} $\mathcal{F}^{H}$ has
    finite coproducts,
  \item for every inclusion of subgroups $H \subseteq K$ the functor
    $(q_{H}^{K})^{\ostar} \colon \mathcal{F}^{K} \to \mathcal{F}^{H}$,
    corresponding to the quotient map $q_{H}^{K} \colon G/H \to G/K$,
    has a left adjoint $(q_{H}^{K})_{\oplus}$,
  \item for every inclusion of subgroups $H \subseteq K$ the functor
    $(q_{H}^{K})^{\ostar}$ preserves finite coproducts,
  \item for subgroups $H,K \subseteq L$, let $X$ be defined by the pullback
    \[
      \begin{tikzcd}
        X \arrow{r}{f_{K}} \arrow{d}{f_{H}} & G/K \arrow{d}{q_{K}^{L}} \\
        G/H \arrow{r}{q_{H}^{L}} & G/L;
      \end{tikzcd}
    \]
   then the square
   \[
     \begin{tikzcd}
       \mathcal{F}^{L} \arrow{r}{(q_{H}^{L})^{\ostar}}
      \arrow{d}{(q_{K}^{L})^{\ostar}} & \mathcal{F}^{H}\arrow{d}{f_{H}^{\ostar}} \\
      \mathcal{F}^{K} \arrow{r}{f_{K}^{\ostar}} & \mathcal{F}(X)
     \end{tikzcd}
   \]
   is left adjointable, \ie{} the mate transformation
   \[ f_{K,\oplus}f_{H}^{\ostar} \to (q^{L}_{K})^{\ostar}(q^{L}_{H})_{\oplus}\]
   is an equivalence.
  \end{enumerate}
\end{propn}

\begin{remark}\label{rmk:doublecosets}
  The pullback $X$ in condition (4) can be decomposed into a sum of
  orbits indexed by double cosets:
  \[ X \cong \coprod_{[g] \in H \setminus L / K} H \cap K_{g}, \]
  where $K_{g}$ denotes the conjugate $g K g^{-1}$.
  The left adjointability in (4) then amounts to the following 
  \emph{double coset formula}:
  \[ (q^{L}_{K})^{\ostar}(q^{L}_{H})_{\oplus} \simeq \coprod_{[g] \in
      H \setminus L / K} c_{g,\oplus} (q_{H \cap
      K_{g}}^{K_{g}})_{\oplus}  (q_{H \cap K_{g}}^{H})^{\ostar}, \]
  where $c_{g}$ is the isomorphism $G/K \cong G/K_{g}$.
\end{remark}

\begin{proof}[Proof of \cref{Gcatleftadjble}]
  Since $\xF_{G}$ is extensive, a morphism
  $\phi \colon X \to Y$ in $\xF_{G}$ where
  $Y \cong \coprod_{i} G/H_{i}$ decomposes as a coproduct
  $\coprod_{i} \phi_{i}$ for $\phi_{i} \colon X_{i} \to G/H_{i}$. 
  Since $\mathcal{F}$ is product-preserving, to show that
  $\phi^{\ostar}$ has a left adjoint it suffices to consider the case
  where $Y$ is an orbit $G/H$. Moreover, for $\phi \colon X \to G/H$
  where $X \cong \coprod_{j} G/K_{j}$ we can decompose $\phi$ as
  \[ \coprod_{j} G/K_{j} \xto{\coprod_{j} \phi|_{G/K_{j}}} \coprod_{j}
    G/H \xto{\nabla} G/H \] where $\nabla$ denotes the fold map. Since
  adjunctions compose, to prove that left adjoints exist it is enough
  to consider fold maps and morphisms between orbits. In the first
  case, the functor
  $\nabla^{\ostar}$ induced by the fold map
  $\nabla  \colon \coprod_{j \in J} G/H \to G/H$ can be identified with the
  diagonal functor $\mathcal{F}^{H} \to \Fun(J, \mathcal{F}^{H})$ and
  so has a left adjoint for all finite sets $J$ \IFF{}
  $\mathcal{F}^{H}$ admits finite coproducts, \ie{} \IFF{} assumption
  (1) holds. In the second case, a morphism $\phi \colon G/K \to G/H$ 
  can be decomposed as  $G/K \isoto G/gKg^{-1} \xto{q_{gKg^{-1}}^{H}}
  G/H$ and so $\phi^{\ostar}$ has a left adjoint for all such maps
  $\phi$ \IFF{} assumption (2) holds.

  Now we consider the adjointability condition. Again using that
  $\xF_{G}$ is extensive, a pullback square
  \[
    \begin{tikzcd}
      X \arrow{r} \arrow{d} & Y \arrow{d} \\
      Z \arrow{r} & W
    \end{tikzcd}
  \]
  where $W \cong \coprod_{i} G/H_{i}$ decomposes as a coproduct of
  pullback squares
  \[
    \begin{tikzcd}
      X_{i} \arrow{r} \arrow{d} & Y_{i} \arrow{d} \\
      Z_{i} \arrow{r} & G/H_{i}.
    \end{tikzcd}
  \]
  Since $\mathcal{F}$ preserves products and taking mate squares
  commutes with products, we see that $\mathcal{F}$ is left
  adjointable \IFF{} it is left adjointable for pullback squares
  \[
      \begin{tikzcd}
        X \arrow{r} \arrow{d} & Y \arrow{d} \\
        Z \arrow{r} & G/H
      \end{tikzcd}
  \]
  over
  an orbit. If $Y \cong \coprod_{i} G/K_{i}$ and $Z \cong \coprod_{j}
  G/L_{j}$ then we can decompose our pullback square into the diagram
  \[
    \begin{tikzcd}
      X \arrow{r} \arrow{d} & \coprod_{i,j} G/K_{i} \arrow{r}
      \arrow{d} & \coprod_{i} G/K_{i} \arrow{d} \\
      \coprod_{i,j} G/L_{j} \arrow{r} \arrow{d} & \coprod_{i,j} G/H
      \arrow{r} \arrow{d} & \coprod_{i} G/H \arrow{d} \\
      \coprod_{j} G/L_{j} \arrow{r} & \coprod_{j} G/H \arrow{r} & G/H.
    \end{tikzcd}
    \]
  where the top left square decomposes as a coproduct of pullback squares
    \[
      \begin{tikzcd}
        X_{i,j} \arrow{r} \arrow{d} & G/K_{i} \arrow{d} \\
        G/L_{j} \arrow{r} & G/H
      \end{tikzcd}
  \]
  of the form considered in (4) and the other squares are defined
  using fold maps. Since mate squares are compatible with both
  vertical and horizontal composition of squares, the functor
  $\mathcal{F}$ will be left adjointable if the images of the four
  squares in such decompositions are left adjointable. Using again the
  assumption that $\mathcal{F}$ preserves finite products, we see that
  this holds \IFF{} (4) holds and we have left adjointability for
  squares of the form
  \[
    \begin{tikzcd}
      \coprod_{i \in I} G/K \arrow{d}{\coprod_{i \in I} \phi}
      \arrow{r}{\nabla_{K}} & G/K \arrow{d}{\phi} \\
      \coprod_{i \in I} G/H \arrow{r}{\nabla_{H}} & G/H.
    \end{tikzcd}
  \]
  The latter means the canonical map $\nabla_{K,\oplus}(\prod_{i}
  \phi^{\ostar}) \to \phi^{\ostar}\nabla_{H,\oplus}$ is an
  equivalence, \ie{} the functor $\phi^{\ostar}$ preserves $I$-indexed
  coproducts where $\phi$ is a map between orbits in $\xF_{G}$. Since
  such maps are composites of isomorphisms and maps coming from
  subgroup inclusions, this is equivalent to condition (3).
\end{proof}

\begin{defn}
  We say a $G$-\icat{} $\mathcal{F}$ has \emph{additive transfers} if
  it satisfies the conditions of \cref{Gcatleftadjble} when viewed as
  a product-preserving functor $\xF_{G}^{\op} \to \CatI$.
\end{defn}

\begin{remark}
  In the terminology of \cite{ShahThesis,NardinThesis} a $G$-\icat{}
  has \emph{finite $G$-coproducts} \IFF{} it has additive transfers in
  our sense, cf.~\cite[Proposition 2.11]{NardinThesis}.
\end{remark}

\begin{propn}\label{propn:SpanFGdist}
  Suppose $\mathcal{F} \colon \Span(\xF_{G}) \to \CatI$ is a
  $G$-symmetric monoidal \icat{} whose underlying $G$-\icat{} has
  additive transfers. Then $\mathcal{F}$ is distributive
  \IFF{} for all morphisms $\phi \colon X \to Y, \psi \colon Y \to
  G/H$ in $\xF_{G}$, the distributivity transformation
\[ g_{\oplus}\tilde{\psi}_{\otimes}\epsilon^{\ostar} \to
  \psi_{\otimes}\phi_{\oplus} \]
  from the distributivity square
  \[
    \begin{tikzcd}
      {} & W \arrow{dd}{\tilde{g}} \arrow{dl}[swap]{\epsilon} \arrow{r}{\tilde{\psi}} & Z
      \arrow{dd}{g} \\
      X \arrow{dr}{\phi} \\
      {} & Y \arrow{r}{\psi} & G/H,
    \end{tikzcd}
\]
    is an equivalence.
\end{propn}

\begin{remark}
  Decomposing $Z = \psi_{*}X$ in such a distributivity diagram as a
  coproduct of orbits is often a non-trivial problem in finite
  group theory, so we do not expect that this distributivity condition
  can be simplified further in general.
\end{remark}

\begin{proof}[Proof of \cref{propn:SpanFGdist}]
  Since $\xF_{G}$ is extensive, we know
  from \cref{ex:bispancoprod} that finite coproducts of distributivity
  diagrams are again distributivity diagrams. Since $\mathcal{F}$
  preserves products, distributivity transformations associated to
  such coproducts decompose as products, hence the distributivity
  condition reduces to the case where the target of the second map is
  an orbit.
\end{proof}

\begin{defn}\label{def:comp-add}
  We say a $G$-symmetric monoidal \icat{} $\mathcal{F} \colon
  \Span(\xF_{G}) \to \CatI$ is \emph{compatible with additive
    transfers} if it satisfies the condition of \cref{propn:SpanFGdist}.
\end{defn}

\begin{cor}
  Product-preserving functors $\BISPAN(\xF_{G}) \to \CatI$ correspond
  to $G$-symmetric monoidal \icats{} that are compatible with additive
  transfers.
\end{cor}
\begin{proof}
  By \cref{thm:main2}, functors $\BISPAN(\xF_{G}) \to \CATI$
  correspond to distributive functors $\Span(\xF_{G}) \to
  \CATI$. Moreover, from \cref{rmk:coprodisprodinbispan} we know that
  the product in $\BISPAN(\xF_{G})$ is given by the coproduct in
  $\xF_{G}$, just as in $\Span(\xF_{G})$, so product-preserving
  functors from $\BISPAN(\xF_{G})$ correspond to product-preserving
  distributive functors under this equivalence. By
  \cref{Gcatleftadjble} and \cref{propn:SpanFGdist}, the latter are
  equivalent to $G$-symmetric monoidal \icats{} that are compatible
  with additive transfers.
\end{proof}

We now consider some examples of $G$-symmetric monoidal \icats{}
compatible with additive transfers:
\begin{ex}[Finite $G$-sets]\label{ex:finGset}
  As a special case of \cref{propn:bispanslice} we get a functor
  \[ \BISPAN(\xF_{G}) \to \CAT \]
  taking $X \in \xF_{G}$ to the slice $(\xF_{G})_{/X}$. Here we can
  identify $(\xF_{G})_{/(G/H)}$ with $\xF_{H}$, so the underlying
  $G$-category is given by $(G/H) \mapsto \xF_{H}$. Since $\xF_{G}$ is extensive, this is a product-preserving
  functor; the underlying $G$-symmetric monoidal category encodes the
  cartesian products of finite $G$-sets and their compatibility with
  the left and right adjoints to the restriction functor $\xF_{G} \to
  \xF_{H}$ for $H$ a subgroup of $G$.
\end{ex}

\begin{ex}[$G$-spaces]
  As a variant of the previous example, we can consider the
  $G$-\icat{} of $G$-spaces. Since $\mathcal{S}_{G}$ is locally
  cartesian closed (being a (presheaf) $\infty$-topos), we can apply
  \cref{propn:bispanslice} to it and then restrict to bispans in
  $\xF_{G}$ to get a functor
  \[ \BISPAN(\xF_{G}) \to \CATI \] that takes $X \in \xF_{G}$ to
  $(\mathcal{S}_{G})_{/X}$; here we can identify
  $(\mathcal{S}_{G})_{/(G/H)}$ with $\mathcal{S}_{H}$.  Since
  $\mathcal{S}_{G}$ is extensive, this is a product-preserving
  functor. The underlying $G$-symmetric monoidal \icat{} (compatible
  with additive transfers) encodes the cartesian products of
  $H$-spaces for all subgroups $H$ of $G$ and their compatibility with
  the left and right adjoints to the restriction functor
  $\mathcal{S}_{G} \to \mathcal{S}_{H}$.
\end{ex}

\begin{ex}[$G$-actions in a symmetric monoidal \icat{}] \label{ex:g-smc}
  Let $\mathcal{C}$ be a symmetric monoidal \icat{}. By \cref{prop:bh}
  this determines a functor $\Span_{\fin}(\mathcal{S}) \to \CatI$
  taking $X \in \mathcal{S}$ to $\Fun(X, \mathcal{C})$. For a finite
  group $G$ we have a functor $\mathcal{G} \colon \xF_{G} \to \mathcal{S}$ by restricting
  the colimit functor $\Fun(BG, \mathcal{S}) \to \mathcal{S}$; this
  takes a finite $G$-set $X$ to the groupoid $X/\!/G = X_{hG}$. The functor
  $\mathcal{G}$ preserves pullbacks since the colimit functor factors
  as the straightening equivalence $\Fun(BG, \mathcal{S}) \isoto
  \mathcal{S}_{/BG}$ followed by the forgetful functor to
  $\mathcal{S}$, which preserves all weakly contractible
  limits. Moreover, $\mathcal{G}$ takes values in
  $\mathcal{S}_{\fin}$, and so yields a functor
  $\Span(\mathcal{G})\colon \Span(\xF_{G}) \to
  \Span_{\fin}(\mathcal{S})$. This functor preserves products, since
  $\mathcal{G}$ preserves coproducts. It follows that we can restrict
  along $\Span(\mathcal{G})$ and obtain for any symmetric monoidal
  \icat{} $\mathcal{C}$ a $G$-symmetric monoidal structure on
  $\Fun(BG, \mathcal{C})$. Moreover, if the tensor product in
  $\mathcal{C}$ is compatible with finite coproducts, this
  $G$-symmetric monoidal structure will be compatible with additive
  transfers.
\end{ex}

\begin{ex}[$G$-representations]\label{ex:g-rep}
  As a special case of the previous example, we can take $\mathcal{C}$
  to be the category $\Vect_{k}$ of $k$-vector spaces with the tensor
  product as symmetric monoidal structure. We then obtain a functor
  \[ \rho_{k}\colon \BISPAN(\xF_{G}) \to \CAT \]
  such that $\rho_{k}(G/H)$ is the category
  $\Rep_{H}(k) := \Fun(BH, \Vect_{k})$ of $H$-representations and for
  for subgroups $H \subseteq K \subseteq G$,
  \begin{itemize}      
  \item $(q_{H}^{K})^{\ostar}$ is the restriction functor
    $\txt{Res}_{H}^{K} \colon \Rep_{K}(k)
    \to \Rep_{H}(k)$,
  \item $(q_{H}^{K})_{\oplus}$ is the induced representation functor
    $\txt{Ind}_{H}^{K} \colon \Rep_{H}(k) \to \Rep_{K}(k)$, left
    adjoint to $\txt{Res}_{H}^{K}$, and given on objects by taking an
    $H$-representation $V$ to $\bigoplus_{K/H} V$ with induced action
    of $K$,
  \item $(q_{H}^{K})_{\otimes}$ is the \emph{tensor-induction functor}
    $\Rep_{H}(k) \to \Rep_{H}(k)$, given on objects
    by taking an
    $H$-representation $V$ to $\bigotimes_{K/H} V$ with induced action
    of $K$.
  \end{itemize}
  As a more sophisticated version of this construction, we can instead
  consider the \icat{} $\Perf_{R}$ of perfect (\ie{} dualizable)
  modules over an $\mathbb{E}_{\infty}$-ring spectrum $R$. Plugging
  this into the previous example we see that the \icats{}
  $\dRep_{H}(R):=\Fun(BH, \Perf_{R})$ fit together into a functor
  \[ \rho^{\infty}_R\colon \BISPAN(\xF_{G}) \to \CAT \]
  such that $\rho_R^{\infty}(G)(G/H) = \dRep_{H}(R)$. This example will be
  important when we discuss Tambara functors arising from the
  algebraic $K$-theory of group actions in the next section. 
\end{ex}

Our final example of a $G$-symmetric monoidal \icat{} compatible with
additive transfers is the \icat{} of genuine $G$-spectra. This is less
formal than our previous examples, but the input we need is already in
the literature:
\begin{propn}\label{prop:genuine-g}
  The \icat{} $\Sp^{G}$ of genuine $G$-spectra is a $G$-symmetric
  monoidal \icat{} compatible with additive transfers.
\end{propn}
\begin{proof}
  Taking fixed point spectra for subgroups of $G$ gives a functor
  $\mathcal{O}_{G}^{\op} \to \CatI$ that takes $G/H$ to the \icat{}
  $\Sp^{H}$ of genuine $H$-spectra; this is the $G$-\icat{} of
  $G$-spectra. The corresponding product-preserving functor
  $\xF_{G}^{\op} \to \CatI$ extends to a functor
  $\sigma_{G} \colon \Span(\xF_{G}) \to \CatI$ such that for
  $H \subseteq K \subseteq G$ the functor
  $(q_{H}^{K})_{\otimes} \colon \Sp^{H} \to \Sp^{K}$ is the
  multiplicative norm of \cite{HHRKervaire}; this follows from the
  results of \cite[\S 9]{norms} by restricting the functor from spans
  of profinite groupoids defined there. For $\phi \colon X \to Y$ in
  $\xF_{G}$, the functor
  $\phi^{\ostar} \colon \sigma_{G}(Y) \to \sigma_{G}(X)$ has a left
  adjoint by \cite[Lemma 9.7(2)]{norms} --- for
  $H \subseteq K \subseteq G$ the left adjoint $(q_{H}^{K})_{\oplus}$
  is the classical (additive) \emph{transfer} or induction functor. To
  see that the functor $\sigma_{G}$ is left adjointable we check the 3
  remaining conditions in \cref{Gcatleftadjble}: conditions (1) and
  (3) hold since the \icats{} of $G$-spectra are stable (hence any
  right adjoint functor between them automatically preserves finite
  colimits). To check condition (4) we use that any $H$-spectrum is a
  sifted colimit of desuspensions of suspension spectra $\Sigma^{\infty}_{+}X$ with $X$
  a finite $H$-set, and the functors involved preserve (sifted)
  colimits and desuspensions. Hence it suffices to check that the
  canonical map
  \[ f_{K,\oplus}f_{H}^{\ostar}\Sigma^{\infty}_{+}X \to
    (q^{L}_{K})^{\ostar}(q^{L}_{H})_{\oplus}\Sigma^{\infty}_{+}X \] is
  an equivalence for $X \in \xF_{H}$. But here all the functors are
  given on suspension spectra of finite $H$-sets by the suspension spectra on the
  corresponding functors for finite $G$-sets, so this follows from
  \cref{ex:finGset}. The same argument works for distributivity, since
  we also have $f_{\otimes}\Sigma^{\infty}_{+}X \simeq \Sigma^{\infty}_{+}f_{*}X$.
\end{proof}

\begin{remark}
  The adjointability condition here reduces by \cref{rmk:doublecosets}
  to a double coset formula for additive transfers. This is a basic
  fact in equivariant stable homotopy theory that has surely long been
  well-known to the experts, but the only explicit references we could
  find are \cite{HHRKervaire}*{Proposition A.30} (applied to the
  direct sum in orthogonal spectra) and
  \cite{PatchkoriaRigid}*{Corollary 5.2}. The distributivity condition
  also appears (in terms of orthogonal spectra) as
  \cite{HHRKervaire}*{Proposition A.37}.
\end{remark}

\begin{variant}
  Following Blumberg and Hill \cite{blumberg-hill} we can consider
  subcategories $\xF_{G,\mathcal{I}}$ where $\mathcal{I}$ is an
  \emph{indexing system} as in \cite{blumberg-hill}*{Definition 1.2};
  by \cite{blumberg-hill}*{Theorem 1.4} these are precisely the
  subcategories of $\xF_{G}$ such that $(\xF_{G},\xF_{G,\mathcal{I}})$
  is a span pair. Since $\xF_{G}$ is locally cartesian closed, we then
  have a bispan triple $(\xF_{G}, \xF_{G,\mathcal{I}},
  \xF_{G})$. Product-preserving functors out of
  $\Span_{\mathcal{I}}(\xF_{G})$ are $G$-symmetric monoidal \icats{}
  where only some subclass of multiplicative norms exist, and we can
  characterize distributivity for such functors by the analogue of
  \cref{propn:SpanFGdist} with the map $\psi$ restricted to lie in
  $\xF_{G,\mathcal{I}}$.
\end{variant}

\begin{variant}
  We can also consider a $G$-equivariant analogue of
  \S\ref{sec:spacesymmmon}: Using \cite[Proposition C.18]{norms} a
  $G$-symmetric monoidal \icat{} determines by right Kan extension a
  functor $\Span_{\fin}(\mathcal{S}_{G}) \to \CatI$, where
  $\mathcal{S}_{G,\fin}$ denotes the subcategory of $\mathcal{S}_{G}$
  containing the maps $\phi \colon X \to Y$ such that for every map
  $G/H \to Y$, the pullback $X \times_{Y} G/H$ is a finite $G$-set.
  Here $(\mathcal{S}_{G}, \mathcal{S}_{G,\fin}, \mathcal{S}_{G})$ is a
  bispan triple, and we might say that the $G$-symmetric monoidal
  \icat{} is ``compatible with $G$-space-indexed $G$-colimits'' if
  this is distributive. We expect that this should hold for the
  $G$-symmetric monoidal \icat{} of genuine $G$-spectra and, by
  analogy with \cite{polynomial}, that monads in the \itcat{}
  $\BISPAN_{\fin}(\mathcal{S}_{G})$ should be related to a notion of
  \emph{$G$-\iopds{}} \cite{BDGNS1}.
\end{variant}

\begin{variant}
  In \cite[Chapter
  9]{norms}, Bachmann and Hoyois define \icats{} of equivariant
  spectra for \emph{profinite groupoids}, and we can also consider
  distributivity in this setting. We can take the $(2,1)$-category of finite groupoids
  $\mathrm{FinGpd} \subset \mathcal{S}$ to be the full subcategory of
  spaces spanned by $1$-truncated spaces with finite $\pi_0,
  \pi_1$. We then form the $(2,1)$-category of profinite groupoids
  by taking pro-objects:
  $\mathrm{ProfGpd}:=\mathrm{Pro}(\mathrm{FinGpd})$. Let
  $\mathrm{ProfGpd}_{\txt{fp}}$ be the subcategory containing only the
  \emph{finitely presented} maps as in \cite{norms}*{\S 9.1}. It then
  follows from \cite{norms}*{Lemmas 9.3 and 9.5} and
  \cref{lem:bispantripmatecond} that we have a bispan triple
  \[(\mathrm{ProfGpd}, \mathrm{ProfGpd}_{\txt{fp}},
    \mathrm{ProfGpd}).\]
  In \cite[Chapter 9]{norms} equivariant
  spectra are defined as a functor
  \[ \Span_{\txt{fp}}(\mathrm{ProfGpd}) \to \CatI;\]
  we expect that this is distributive, giving a functor of \itcats{}
  \[ \BISPAN_{\txt{fp}}(\mathrm{ProfGpd}) \to \CATI.\]
\end{variant}

\subsection{Motivic bispans and normed $\infty$-categories}
In this section we will relate the \emph{normed $\infty$-categories}
of Bachmann--Hoyois to functors from certain bispans in schemes (and
more generally algebraic spaces) to $\CATI$. As an example of this, we will
see that the results of \cite{norms} imply that \icats{} of motivic
spectra give such a functor. We begin by describing some bispan
triples on schemes.

\begin{warning}
  Throughout this section, schemes and algebraic spaces are always
  assumed to be quasi-compact and quasi-separated (qcqs).\footnote{Note that
    every morphism between qcqs schemes is automatically a qcqs
    morphism (see \cite[Tags 01KV and 03GI]{StacksProject}); this
    means we do not need to distinguish between morphisms of finite
    presentation and locally of finite presentation, since the
    additional qcqs assumption is automatic.}
\end{warning}

\begin{notation}
  If $S$ is a (qcqs) scheme, we write $\Sch_{S}$ for the category of (qcqs) schemes
  over $S$. (This has pullbacks since qcqs morphisms are closed under
  base change, see \cite[Tags 01KU and 01K5]{StacksProject}.)
\end{notation}

\begin{propn}
  The following are bispan triples:
  \begin{enumerate}[(i)]
  \item $(\Sch_{S}, \Sch_{S}^{\flf}, \Sch_{S}^{\qp})$ for any scheme
    $S$, where $\Sch_{S}^{\flf}$ consists of finite locally free
    (meaning finite, flat, and of finite presentation) morphisms of
    $S$-schemes and $\Sch_{S}^{\qp}$ of quasiprojective morphisms of
    $S$-schemes,
  \item $(\Sch_{S}, \Sch_{S}^{\fet}, \Sch_{S}^{\qp})$ for any scheme
    $S$, where $\Sch_{S}^{\fet}$ consists of finite \'etale morphisms
    of $S$-schemes,
  \item $(\Sch_{S}, \Sch_{S}^{\flf}, \Sch_{S}^{\smqp})$ for any scheme
    $S$, where $\Sch_{S}^{\smqp}$ consists of smooth and
    quasiprojective morphisms of $S$-schemes,
  \item\label{fetsmqp} $(\Sch_{S}, \Sch_{S}^{\fet}, \Sch_{S}^{\smqp})$ for any scheme
    $S$.
  \item $(\Sch_{S}, \Sch_{S}^{\fet}, \Sch_{S}^{\proj})$ for any scheme
    $S$, where $\Sch_{S}^{\proj}$ consists of projective morphisms of
    $S$-schemes.
  \end{enumerate}
\end{propn}
\begin{proof}
  The classes of morphisms of schemes that are finite locally free,
  quasiprojective, smooth, finite, and \'etale are all closed
  under base change by \cite[Tags 02KD,0B3G,01VB,01WL,02GO]{StacksProject},
  respectively. Hence the subcategories
  $\Sch_{S}^{\flf}$, $\Sch_{S}^{\qp}$, $\Sch_{S}^{\smqp}$, and
  $\Sch_{S}^{\fet}$ of $\Sch_{S}$ all give span pairs.

  Now the main point is the existence of \emph{Weil restrictions} for
  schemes: if $f \colon S' \to S$ is a morphism of schemes and $X$ is
  an $S'$-scheme, the Weil restriction $R_{f}X$, if it exists, is an
  $S$-scheme that represents the functor
  \[ \Hom_{/S'}((\blank)\times_{S'} S, X) \colon \Sch_{S}^{\op} \to
    \Set;\]
  note that this is exactly the requirement \cref{eq:uni-prop-dis} for
  a distributivity diagram for $X \to S'$ and $f$.

  By \cite[Theorem 7.6.4]{neron}, the Weil restriction $R_{f}X$ exists
  if $f$ is a finite locally free morphism and $X$ is
  quasiprojective. Moreover, $R_{f}X$ is quasiprojective over $S$ by
  \cite[Lemma 2.13]{norms}. This gives the bispan triple (i), from
  which (ii) is trivially obtained by restricting from finite locally
  free morphisms to the subclass of finite \'etale ones. For (iii) and
  (iv) the only additional input needed is that $R_{f}$ takes smooth
  morphisms to smooth morphisms, which holds by \cite[Proposition
  7.6.5(h)]{neron}, while for (v) we use that $R_{f}$ preserves proper
  morphisms if $f$ is finite \'etale by \cite[Proposition
  7.6.5(f)]{neron} and that a morphism is projective \IFF{} it is
  proper and quasiprojective by \cite[Tag 0BCL]{StacksProject}.
\end{proof}

We now review the construction of a distributive functor for the
bispan triple \ref{fetsmqp} from motivic spectra, due to Bachmann and
Hoyois.

\begin{notation}
  We write $\SH(S)$ for the \icat{} of motivic spectra over a base
  scheme $S$ and $\HH(S)$ for that of motivic spaces over $S$. For any
  morphism of schemes $f \colon S \to S'$ we have a pullback functor
  $f^{*} \colon \SH(S') \to \SH(S)$, and similarly in the unstable
  case. This gives functors $\SH,\HH \colon \Sch^{\op} \to \CatI$.
\end{notation}

In \cite{norms}, Bachmann and Hoyois promoted the contravariant
functor $X \mapsto \SH(X)$ to include a \emph{multiplicative
  pushforward} for finite \'etale morphisms, encoded as a functor out
of a span category
\begin{equation} \label{eq:bh}
\SH\colon \Span_{\fet}(\Sch) \rightarrow \CatI, \qquad X \xleftarrow{f} Z \xrightarrow{g} Y \,\mapsto\, g_{\otimes}f^*.
\end{equation}
Given a finite \'etale morphism $g\colon X \rightarrow Y$, the functor 
\[
g_{\otimes}\colon \SH(X) \rightarrow \SH(Y)
\]
is first constructed unstably as a functor on the level of the
pointed unstable motivic homotopy $\infty$-category,
\[
g_{\otimes}\colon \mathrm{H}(X)_{*} \rightarrow \mathrm{H}(Y)_{*}.
\]
This functor is in turn induced by the functor of Weil restriction \cite[\S 7.6]{neron},
$R_g\colon \mathrm{SmQP}_X \rightarrow \mathrm{SmQP}_Y$, where
$\mathrm{SmQP}_X$ denotes the full subcategory of $\Sch_{X}$ spanned
by smooth and quasiprojective $X$-schemes, using the fact
that the inclusion $\mathrm{SmQP}_X \subset \mathrm{Sm}_X$ into the
full subcategory of smooth $X$-schemes induces
equivalent motivic unstable categories (since every smooth $X$-scheme
is Zariski-locally also quasiprojective). We refer to \cite[\S
1.6]{norms} for a summary of the construction and \cite[\S 6.1]{norms}
for a detailed construction of~\eqref{eq:bh}.

Given a smooth morphism $f\colon X \rightarrow Y$ of schemes, the
pullback functor $f^*$ admits a left
adjoint
\[
f_{\sharp} \colon \SH(X) \rightarrow \SH(Y).
\]
This left adjoint should be thought of as an \emph{additive
  pushforward} along $f$; indeed, if $I$ is a finite set and
$\nabla_I \colon \coprod_I X \rightarrow X$ is the fold map then,
under the identification $\SH(\coprod_I X) \simeq \SH(X)^{\times I}$,
the functor $(\nabla_I)_{\sharp}$ is given by
\[
(X_i)_{i \in I} \mapsto \bigoplus_{i \in I} X_i.
\] 
The functor $f_{\sharp}$ is first constructed unstably as a functor 
\[
f_{\sharp}\colon\mathrm{H}(X) \rightarrow \mathrm{H}(Y),
\]
which in turn is induced by the functor
$\mathrm{Sm}_X \rightarrow \mathrm{Sm}_Y$ given by composition with
$f$ (\ie{} the functor that sends a smooth $X$-scheme $T$ to itself
regarded as a $Y$-scheme); see \cite[Section 4.1, Lemma
6.2]{hoyois-sixops} for a construction in the language of this paper
in the more general context of equivariant motivic homotopy theory.

The importance of this additional left adjoint functoriality in
formulating smooth base change was first pointed out by Voevodsky
\cite{voevodsky-four} and worked out by Ayoub in
\cite{ayoub-thesis1}; see \cite[Section 6.1, Proposition
4.2]{hoyois-sixops} for an \icatl{} formulation (in the more general
equivariant context). In our language, smooth base change for motivic spectra says that the functor
$\SH \colon \Sch^{\op} \to \CatI$ is left adjointable with respect to
smooth maps. Moreover, combining this with \cite[Proposition
5.10(1)]{norms} we get that the functor \cref{eq:bh} is
$\smqp$-distributive. Applying \cref{thm:main2} these results
imply:
\begin{thm}
  Motivic spectra give a $\smqp$-distributive functor
  \[ \SH \colon \Span_{\fet}(\Sch) \to \CatI,\]
  and so a functor of \itcats{}
  \[ \BISPAN_{\fet,\smqp}(\Sch) \to \CATI.\]
\end{thm}

As indicated already in \cite[Remark 2.14]{norms}, the restriction to
those smooth morphisms that are quasiprojective here is an artifact of
the restriction of $\SH$ to schemes instead of algebraic spaces. We
will therefore extend this result by working with algebraic spaces.

\begin{notation}
  For a (qcqs) scheme $S$, we write $\AlgSpc_S$ for the category
  of (qcqs) algebraic spaces over $S$.\footnote{Every morphism
    between qcqs algebraic spaces is a qcqs morphism by \cite[Tag
  03KR,03KS]{StacksProject}; thus we can still ignore the distinction
  between morphisms of finite
    presentation and locally of finite presentation.} This category
  has pullbacks since qcqs morphisms of algebraic spaces are closed
  under base change by \cite[Tags 03KL,03HF]{StacksProject}; note also
  that for an $S$-scheme $S'$ we have an equivalence
  \[ \AlgSpc_{S'} \simeq (\AlgSpc_{S})_{/S'}\]
  by \cite[Tag 04SG]{StacksProject}.
\end{notation}
Here we again have several bispan triples:
\begin{propn}
  The following are bispan triples:
  \begin{enumerate}[(i)]
  \item $(\AlgSpc_{S}, \AlgSpc_{S}^{\flf}, \AlgSpc_{S})$ for any
    scheme $S$, where $\AlgSpc_{S}^{\flf}$ consists of
    finite locally free morphisms of algebraic spaces over $S$,
  \item $(\AlgSpc_{S}, \AlgSpc_{S}^{\fet}, \AlgSpc_{S})$ for any
    scheme $S$, where $\AlgSpc_{S}^{\fet}$ consists of
    finite \'etale morphisms of algebraic spaces over $S$,
  \item $(\AlgSpc_{S}, \AlgSpc_{S}^{\flf}, \AlgSpc^{\sm}_{S})$ for any
    scheme $S$, where $\AlgSpc_{S}^{\sm}$ consists of
    smooth morphisms of algebraic spaces over $S$,
  \item $(\AlgSpc_{S}, \AlgSpc_{S}^{\fet}, \AlgSpc^{\sm}_{S})$ for any
    scheme $S$.
  \item $(\AlgSpc_{S}, \AlgSpc_{S}^{\fet}, \AlgSpc^{\prop}_{S})$ for any
    scheme $S$, where $\AlgSpc^{\prop}_{S}$ consists of proper
    morphisms of algebraic spaces over $S$.
  \end{enumerate}
\end{propn}

\begin{proof}
  Morphisms of algebraic spaces that are finite locally free, finite,
  \'etale, and smooth are closed under base change by \cite[Tags 03ZY,
  03ZS, 0466, 03ZE]{StacksProject}, respectively. Thus the subcategories
  $\AlgSpc_{S}^{\flf}$, $\AlgSpc_{S}^{\fet}$, $\AlgSpc_{S}^{\sm}$ of
  $\AlgSpc_{S}$ all give span pairs.

  Suppose $f \colon Y \to Z$ is a finite locally free morphism of
  algebraic spaces. Then the functor
  \[
    f^*\colon \mathrm{AlgSpc}_{Z} \rightarrow \mathrm{AlgSpc}_{Y}
  \]
  admits a right adjoint $R_{f}$, given by Weil restriction of algebraic
  spaces, by a result of Rydh \cite[Theorem
  3.7]{rydh-hilb}; note that $R_{f}$ preserves the qcqs property we require by
  \cite[Proposition 3.8(xiii,xix)]{rydh-hilb}. This gives the bispan triples (i) and (ii) (since
  finite \'etale morphisms are in particular finite locally free).

  To obtain (iii) and (iv) it suffices to note that $R_f$ converts
  smooth morphisms to smooth
  morphisms. If $f \colon X \to Y$ is a finite locally free morphism of schemes this follows from \cite[Proposition 3.5(i,iv)]{rydh-hilb} and
  \cite[Tag 0DP0]{StacksProject}. The extension to the general case is
  easy: First note that for
  $W$ over $X$ we have that $R_{f}W \to Y$ is smooth \IFF{} its
  pullback along any morphism $g \colon T \to Y$ with $T$ a scheme is
  smooth, by \cite[Tag 03ZF]{StacksProject}. In the pullback square
  \[
    \begin{tikzcd}
      U \arrow{r}{g'} \arrow{d}{f'} & X \arrow{d}{f} \\
      T \arrow{r}{g} & Y
    \end{tikzcd}
  \]
  the algebraic space $U$ is a scheme since $f$ is finite and so by
  definition representable. We also have the base change equivalence
  $g^{*}R_{f}W \cong R_{f'}g'^{*}W$, and if $W$ is smooth over $X$
  then $R_{f'}g'^{*}W$ is smooth over $T$ since the base change
  $g'^{*}W$ is smooth over $U$ and $f'$ is a finite locally free
  morphism of schemes. Finally, (v) holds since by the same argument
  starting with \cite[Remark 3.9]{rydh-hilb}
  $R_{f}$ preserves proper morphisms if $f$ is finite \'etale.
\end{proof}

In order to work with motivic spectra over algebraic spaces
effectively we record the following lemma which amounts to saying that
any Nisnevich sheaf on algebraic spaces is right Kan extended from
schemes.

\begin{lemma} \label{lem:rke} Let $\mathcal{C}$ be a complete
  $\infty$-category and $S$ a scheme, and let
\[
\iota\colon \mathrm{Sch}_S \hookrightarrow \mathrm{AlgSpc}_S
\]
be the inclusion. Then the restriction functor
$\iota^*\colon \PShv( \mathrm{AlgSpc}_S, \mathcal{C}) \rightarrow
\PShv( \mathrm{Sch}_S, \mathcal{C})$ induces an equivalence of
$\infty$-categories:
\[
\iota^*\colon\mathrm{Shv}_{\mathrm{Nis}}(\mathrm{AlgSpc}_S,
\mathcal{C}) \isoto \mathrm{Shv}_{\mathrm{Nis}}(\mathrm{Sch}_S, \mathcal{C}),
\]
with the inverse given by right Kan extension.
\end{lemma}

\begin{proof} Let 
\[
\iota_*\colon \PShv( \mathrm{Sch}_S, \mathcal{C}) \rightarrow \PShv( \mathrm{AlgSpc}_S, \mathcal{C})
\]
denote the right adjoint to $\iota^*$ which is computed by right Kan extension. We first claim that $\iota_*$ preserves Nisnevich sheaves, i.e., there exists a filler in the following diagram:
\[
\begin{tikzcd}
\mathrm{Shv}_{\mathrm{Nis}}(\mathrm{Sch}_S, \mathcal{C}) \ar[hook]{d} \ar[dashed]{r} & \mathrm{Shv}_{\mathrm{Nis}}(\mathrm{AlgSpc}_S, \mathcal{C}) \ar[hook]{d}\\
\PShv( \mathrm{Sch}_S, \mathcal{C}) \ar{r}{\iota_*} & \PShv( \mathrm{AlgSpc}_S, \mathcal{C})
\end{tikzcd}
\]
According to \cite[Lemma 2.23]{gepner-heller}, it suffices to verify
that the functor $\iota$ is topologically cocontinuous\footnote{In the
  sense of \cite[\'Expose III, D\'efinition 2.1]{sga4-1} where this is
  called ``continuous''; to avoid confusion with the notion of a
  functor that preserves limits, we borrow this terminology from
  \cite{adeel-mv}.} for the Nisnevich topology on $\mathrm{Sch}_S$
and on $\mathrm{AlgSpc}_S$, as this implies that $\iota^*$
preserves $\mathrm{Nis}$-local equivalences and thus $\iota_*$
preserves sheaves.

Unwinding definitions, this means we need to prove the following claim:
\begin{itemize}
\item[($\ast$)] For any $X \in \mathrm{Sch}_S$ and any Nisnevich sieve $R' \hookrightarrow \iota(X)$ of algebraic spaces, the sieve on $\Sch_S$ generated by morphisms of schemes $X' \rightarrow X$ such that $\iota(X') \rightarrow \iota(X)$ factors through $R'$ is a Nisnevich sieve of $X$.
\end{itemize}
This condition is verified by \cite[Chapter II, Theorem 6.4]{knutson};
indeed for $x \in X$ and $f\colon Y \rightarrow X$ an \'etale morphism such that we have a lift
\[
\begin{tikzcd}
 & Y \ar{d}{f} \\
 \Spec \kappa(x) \ar{r}{x} \ar{ur}& X,
\end{tikzcd}
\]
this result tells us that we can find a completely decomposed \'etale
morphism $U \rightarrow Y$ with $U$ an affine scheme such that
$\Spec \kappa(x) \rightarrow Y$ factors through $U$. Since the
composite $U \rightarrow X$ is an \'etale morphism, the desired claim
is verified.

Now we claim that $\iota^*$ also preserves sheaves. Indeed, if
$\{ U_{\alpha} \rightarrow X \}$ is a family which generates a
Nisnevich covering family of schemes, then it is still a Nisnevich
covering family of algebraic spaces. Therefore $\iota$ is also a
morphism of sites\footnote{In the sense reviewed in, say, \cite[Appendix B.1]{e-shah}.},  and thus $\iota^*$ preserves sheaves, i.e., there exists
a filler in the following diagram:
\[
\begin{tikzcd}
\mathrm{Shv}_{\mathrm{Nis}}(\mathrm{AlgSpc}_S, \mathcal{C}) \ar[hook]{d} \ar[dashed]{r} & \mathrm{Shv}_{\mathrm{Nis}}(\mathrm{Sch}_S, \mathcal{C}) \ar[hook]{d}\\
\PShv( \mathrm{AlgSpc}_S, \mathcal{C}) \ar{r}{\iota^*} & \PShv( \mathrm{Sch}_S, \mathcal{C}).
\end{tikzcd}
\]

Therefore we have an adjunction on the level of Nisnevich sheaves
\[
\iota^*\colon\mathrm{Shv}_{\mathrm{Nis}}(\AlgSpc_S, \mathcal{C})
\rightleftarrows \mathrm{Shv}_{\mathrm{Nis}}(\mathrm{Sch}_S,
\mathcal{C}) :\!\iota_*,
\]
where $\iota_{*}$ is fully faithful since it is given by right Kan
extension along the fully faithful functor $\iota$. Equivalently, the
counit transformation $\iota^{*}\iota_{*} \to \id$ is an equivalence. 

It then suffices to prove that the unit transformation $\id \to
\iota_{*}\iota^{*}$ is also an equivalence. According to \cite[Theorem 3.4.2.1]{SAG}, any qcqs algebraic space admits a scalloped decomposition in the sense of \cite[Definition 2.5.3.1]{SAG}. In this context, this means that any $X \in  \mathrm{AlgSpc}_S$ admits a sequence of open immersions
\[
\emptyset = U_0 \hookrightarrow U_1 \hookrightarrow U_2 \hookrightarrow \cdots U_i \hookrightarrow \cdots U_n = X,
\]
such that for any $1 \leq i \leq n$ we have a bicartesian diagram in algebraic spaces
\[
\begin{tikzcd}
W_i \ar{d} \ar{r} & V_i \ar{d}{p_i} \\
U_{i-1} \ar{r}{j_i} & U_i,
\end{tikzcd}
\]
where each $p_i$ is \'etale. A result of Morel-Voevodsky \cite{mv99} (in the form \cite[Theorem 3.7.5.1]{SAG}) states that any Nisnevich sheaf converts the above square to a cartesian square. Now although $W_i$ is not necessarily a scheme, it is an open sub-algebraic space of an affine scheme and hence is separated. Therefore, by induction and the fact that $\iota^*$ is computed on the level of presheaves, the map $F \rightarrow \iota_*\iota^*F$ is an equivalence for any Nisnevich sheaf $F$ as soon as we know that it is an equivalence when evaluated on a separated algebraic space.

Now, if $X$ is a separated algebraic space then, in the scalloped
decomposition of $X$, we see that $W_i$ is, in fact, affine, as argued
in the proof of \cite[Proposition 2.2.13]{adeel-mv}. Therefore, by
induction again, we need only check that
$F \rightarrow \iota_*\iota^*F$ is an equivalence on affine schemes
which is tautologically true.
\end{proof}

We can now prove:

\begin{thm}\label{thm:mot-bispans} The functor~\eqref{eq:bh} extends
  canonically to a $\sm$-distributive functor
  \[ \SH \colon \Span_{\fet}(\AlgSpc_{S}) \to \CatI \]
  and hence to a functor of $(\infty,2)$-categories
\[
\SH\colon\BISPAN_{\fet,\sm}(\AlgSpc_{S}) \rightarrow \CATI.
\]
\end{thm}

\begin{proof} First, by Lemma~\ref{lem:rke}, the right Kan extension of the Nisnevich sheaf
\[
\SH\colon \Sch^{\op} \rightarrow \CATI
\]
to algebraic spaces defines a Nisnevich sheaf
\[
\SH\colon \AlgSpc^{\op} \rightarrow \CATI.
\]
Therefore we can apply \cite[Proposition C.18]{norms} to obtain an extension 
\[
\SH\colon \Span_{\fet}(\AlgSpc) \rightarrow \CatI
\]
of \eqref{eq:bh}.

Now in order to use \cref{thm:main2}, we need to
verify that this extension satisfies the distributivity property with
respect to smooth morphisms. Since the functors involved in the
adjointability and distributivity transformations are stable under
base change, we are reduced to the case of schemes we already
discussed above. %
\end{proof}

\begin{remark}\label{rem:adeel}
  In the more general setting of spectral algebraic spaces, Khan has
  constructed the unstable motivic homotopy $\infty$-category (defined
  in \cite[Definition 2.4.1]{adeel-mv}). By Nisnevich descent, this
  agrees with the right Kan extended version appearing in the proof of
  Theorem~\ref{thm:mot-bispans}, using the uniqueness part of
  Lemma~\ref{lem:rke}.
\end{remark}

\subsection{Bispans in spectral Deligne-Mumford stacks and $\Perf$}

In this subsection we promote the functor
$\Perf\colon \SpDM \rightarrow \CatI$ to a functor out of an \itcat{}
of bispans. This extends a result of Barwick~\cite[Example
D]{BarwickMackey}, which gives a functor
\[ \Perf \colon \Span_{\FP}(\SpDM) \to \CatI\] encoding the usual
pullback $f^{*}$ and pushforward $f_{*}$ for a morphism in $\SpDM$,
with $\FP$ a class of morphisms for which $f_{*}$ restricts to
perfect objects and satisfies base change. Our version adds a
multiplicative pushforward $f_{\otimes}$ where $f$ is finite \'etale
(at the cost of restricting the class $\FP$ in order to guarantee the
existence of Weil restrictions).  We note that the multiplicative
pushforward in this situation is right Kan extended from the symmetric
monoidal structure in $\Perf$ and is thus not as complicated as in the
motivic and equivariant cases we considered above. However, we include
this section with a view towards applications in algebraic $K$-theory
in the following section.

We will freely use the language of spectral
Deligne--Mumford (DM) stacks introduced in \cite{SAG}. We denote by
$\SpDM_S$ the $\infty$-category of spectral Deligne-Mumford stacks
over a base $S$. We also adopt the following terminology from
\cite[Example D]{BarwickMackey}:
\begin{defn}
  Recall that for $X \in \SpDM$ an object $\mathcal{E} \in \QCoh(X)$ is
  called \emph{perfect} if for every map  $x\colon\Spec A \rightarrow
  X$, where $A$ is an $E_{\infty}$-ring spectrum, the $A$-module
  $x^*\mathcal{E}$ is perfect (\ie{} dualizable or equivalently
  compact in the symmetric monoidal \icat{} $\Mod_{A}$).
  Now suppose that $f\colon X \rightarrow Y$ is a morphism in $\SpDM$. We
  say that $f$ is \emph{perfect} if the pushforward functor
  \[
    f_*\colon\QCoh(X) \rightarrow \QCoh(Y),
  \]
  takes perfect objects to perfect objects.
\end{defn}

The following theorem of Lurie furnishes a large class of perfect
morphisms:
\begin{thm}[{\cite[Theorem 6.1.3.2]{SAG}}]\label{thm:perf}
  Let $f\colon X \rightarrow Y$ be a morphism in $\SpDM$. If $f$
  is proper, locally almost of finite presentation, and of finite
  Tor-amplitude, then $f$ is perfect.
\end{thm}

\begin{notation}
  Following \cite[Notation D.17]{BarwickMackey}, we label the
  class of morphisms in Theorem~\ref{thm:perf} by $\FP$.  
\end{notation}

We will also make use the existence of the spectral version of Weil
restriction; see the discussion of \cite[Section 19.1]{SAG} and note
that the definitions are completely analogous to the classical
situation. There Lurie proves the following existence theorem:
\begin{thm}[{\cite[Theorem 19.1.0.1]{SAG}}]\label{thm:weil}
  Suppose that $f\colon X \rightarrow Y$ is a morphism in $\SpDM$
  that is proper, flat, and locally almost of finite presentation. Let
  $p\colon Z \rightarrow X$ be a relative spectral algebraic space
  that is quasi-separated and locally almost of finite
  presentation. Then the Weil restriction $R_f(p) \in \SpDM_{Y}$
  exists.
\end{thm}

\begin{notation}
  In light of this, let us write
  \begin{itemize}
  \item $\mathcal{W}$ for the class of morphisms in $\SpDM$ that are
    proper, flat, and locally almost of finite presentation,
  \item $\mathcal{Q}$ for the class of morphisms in $\SpDM$ that are
    relative spectral algebraic spaces, quasi-separated, and locally
    almost of finite presentation,
  \item $\FP' \subset \FP$ for the class of
    morphisms in $\FP$ which are furthermore relative spectral algebraic
    spaces,
  \item $\fet$ for the class of finite \'etale morphisms in $\SpDM$.
  \end{itemize}
\end{notation}
Then Weil restrictions of morphisms in $\mathcal{Q}$ along ones in
$\mathcal{W}$ exist in $\SpDM$. Here $\FP' \subseteq \mathcal{Q}$
since proper morphisms are always quasi-separated, and $\fet \subseteq
\mathcal{W}$.

\begin{lemma}\label{lem:weil}
  Suppose that $f\colon X \rightarrow Y$ is a morphism in $\SpDM$ 
  of class $\mathcal{W}$ %
  and
  $p\colon Z \rightarrow X$ is one of class $\mathcal{Q}$. %
  Then $R_{f}(p)$ is again of class $\mathcal{Q}$. 
  Assuming $f$ is moreover finite \'etale, we also have:
  \begin{enumerate}[(a)]
  \item if $p$ is quasi-compact, then $R_f(p) \rightarrow Y$ is quasi-compact;
  \item if $p$ is proper, then $R_f(p) \rightarrow Y$ is proper;
  \item if $p$ is of finite Tor-amplitude, then $R_f(p) \to Y$ is of finite Tor-amplitude.
  \end{enumerate}
\end{lemma}

\begin{proof}
  The statement that $R_{f}(p)$ is of class $\mathcal{Q}$ is part of %
  \cite[Theorem 19.1.0.1]{SAG}.

  Let us now verify properties (a)--(c). To verify (a) and (c), note
  that quasi-compactness and having finite Tor-amplitude can be
  detected \'etale locally on the target (for the former, this is
  \cite[Remark 2.3.2.5]{SAG} and the equivalences of \cite[Proposition
  2.3.2.1]{SAG} and for the latter this is \cite[Proposition
  6.1.2.2]{SAG}). Therefore, we may work \'etale locally on $Y$. Since
  $f$ was assumed to be finite \'etale, it is \'etale locally a fold
  map $f\colon X \simeq \coprod_{i=1}^{n} Y \rightarrow Y$. In this
  case, we can write $p \colon Z \rightarrow X$ as a coproduct
  $\coprod_i f_i\colon\coprod_i Z_i \rightarrow \coprod_i
  Y$. Therefore the Weil restriction takes the form
  $R_f(p) \simeq Z_1 \times_{Y} Z_2 \times_{Y} \cdots \times_Y Z_{n}
  \rightarrow Y$. To conclude (a), we note that quasi-compactness is
  stable under base change \cite[Proposition 2.3.3.1]{SAG}, while for
  (c), we note that Tor-amplitudes add up under base change
  \cite[Lemma 6.1.1.6]{SAG}.

  To prove (b), we use the valuative criterion for properness
  \cite[Corollary 5.3.1.2]{SAG}, which applies since we have already
  verified (a) and (1)--(3), together with the functor-of-points description of
  the Weil restriction.
\end{proof}

\begin{propn} \label{prop:bispans-dm}
  Let $S$ be a spectral Deligne-Mumford stack. Then
  \[
    (\SpDM_S, \SpDM^{\fet}_S,\SpDM^{\FP'}_S)
  \] is a bispan triple.
\end{propn}

\begin{proof}
  After Lemma~\ref{lem:weil} it suffices to note that morphisms in
  $\fet$ and $\FP'$ are stable under base change. This follows from
  \cite[Proposition 5.1.3.1, Proposition 4.2.1.6, Proposition 6.1.2.2,
  Proposition 1.4.1.11(2), Proposition 3.3.1.8]{SAG}.
\end{proof}

\begin{thm}\label{thm:perfbispan}
Let $S$ be a spectral Deligne-Mumford stack. The functor
\[
\Perf\colon\SpDM^{\op}_S \rightarrow \CatI,
\]
canonically extends to a functor
\[ \Perf \colon \Span_{\fet}(\SpDM_{S}) \to \CatI.\]
Moreover, this is right $\FP'$-distributive (in the sense of \cref{var:rightdist}), and so canonically
extends further to a functor of \itcats{}
\[
\Perf\colon \BISPAN_{\fet, \FP'}(\SpDM_S)^{\twop} \rightarrow \CATI.
\]
\end{thm}
\begin{proof}
  We first apply \cite[Proposition
  C.9]{norms} to extend $\Perf$ to a functor
  \[\Span_{\fold}(\SpDM_{S}) \to \CatI,\]
  where $\SpDM^{\fold}_{S}$ consists of the finite fold maps, \ie{} the
  maps $\coprod_{I}X \to \coprod_{J}X$ with $I \to J$ a map of finite
  sets. Here the pushforward
  \[
    \nabla_{\otimes}\colon \Perf(\coprod_I X) \cong \Perf(X)^{\times
      I} \rightarrow \Perf(X)
  \]
  is just the tensor product, and the base change simply encodes the
  fact that the pullback functors are symmetric monoidal.

  Next, we use \cite[Corollary C.13]{norms} for
  $\mathcal{C} = \mathrm{SpDM}$, $t$ the \'etale topology and $m$ the
  class of finite \'etale map to obtain a functor
  \[
    \Perf \colon \Span_{\fet}(\SpDM_{S}) \to \CatI.
  \]
  The content of this result is that since $\Perf$ is an \'etale sheaf and
  finite \'etale morphisms are \'etale-locally contained in the class
  of fold maps, we can extend the symmetric monoidal structure to
  norms along finite \'etale morphisms.

  In order to show that this functor is right $\FP'$-distributive, we
  first check that its restriction $\Perf \colon \SpDM_{S} \to \CatI$
  is right $\FP'$-adjointable. For any morphism $f \colon X \to Y$ in
  $\SpDM_{S}$ the functor $f^{*} \colon \QCoh(Y) \to \QCoh(Y)$ has a
  right adjoint $f_{*}$. Given a pullback square
  \[
    \begin{tikzcd}
      X' \arrow{r}{f'} \arrow{d}{g'} & Y' \arrow{d}{g} \\
      X \arrow{r}{f} & Y
    \end{tikzcd}
  \]
  the commutative square
  \[
    \begin{tikzcd}
    \QCoh(Y) \arrow{r}{f^{*}} \arrow{d}{g^{*}} & \QCoh(X) \arrow{d}{g'^{*}} \\
    \QCoh(Y') \arrow{r}{f'^{*}} & \QCoh(X')
  \end{tikzcd}
\]
  is right adjointable by \cite[Corollary 3.4.2.2]{SAG} provided $f$
  is quasi-compact and quasi-separated. This is true by definition
  \cite[Definition 5.1.2.1]{SAG} for any proper morphism and so for
  any morphism in $\FP'$. Moreover, if $f$ is of finite Tor amplitude then
  $f_{*}$ preserves perfect complexes by Theorem~\ref{thm:perf}, so in
  this case the adjunction restricts to an adjunction
  \[ f^{*} : \Perf(Y) \rightleftarrows \Perf(X) : f_{*} \]
  on the full subcategories of perfect objects, which still satisfies
  the right adjointability condition if $f$ is also quasi-compact and
  quasi-separated. This holds in particular if $f$ is in $\FP'$, so
  that $\Perf$ is indeed right $\FP'$-adjointable.

  It remains to check the (right) distributivity condition for $p
  \colon X\to Y$ in $\FP'$ and $f \colon Y \to Z$ finite \'etale:
  given a distributivity diagram
  \begin{equation} \label{eq:dis-sq-perf}
    \begin{tikzcd}
      {} & f^{*}R_f(p) \arrow{dd}{\tilde{g}} \arrow{dl}[swap]{\epsilon} \arrow{r}{\tilde{f}} & R_f(p)
      \arrow{dd}{g} \\
      X \arrow{dr}{p} \\
      {} & Y \arrow{r}{f} & Z,
    \end{tikzcd}
  \end{equation}
  the (right) distributivity transformation
  \[ f_{\otimes}p_{*} \to g_{*}\tilde{f}_{\otimes}\epsilon^{*}\]
  must be invertible. Since $\Perf$ is an \'etale sheaf and
  distributivity transformations satisfy base change by
  \cref{propn:pbdisttr}, we may check this \'etale-locally on
  $Z$. Since finite \'etale morphisms are \'etale-locally given by
  finite fold maps, this means we may assume that $f$ is a fold map
  \[ \nabla \colon Y \simeq \coprod_{i=1}^{n} Z \to Z.\]
  Since $\SpDM$ is extensive, we get a
  decomposition of $p$ as
  \[ \coprod_{i=1}^{n}p_{i} \colon \coprod_{i=1}^{n} X_{i} \to
    \coprod_{i=1}^{n} Z, \]
  and an equivalence
  \[ R_{\nabla}(p) \simeq X_{1} \times_{Z} X_{2} \times_{Z} \cdots
    \times_{Z} X_{n},\]
  since the universal property of $R_{\nabla}(p)$ is equivalent to that of
  this iterated fibre product:
  \[
    \begin{split}
     \Map_{\SpDM_{S/\coprod_{i}Z}}(W,R_{\nabla}(p))  & \simeq
     \Map_{\SpDM_{S/\coprod_{i}Z}}(\nabla^{*}W, X) \\
     & \simeq \Map_{\SpDM_{S/\coprod_{i}Z}}(\coprod_{i} W, \coprod_{i} X_{i}) \\     
     & \simeq \prod_{i} \Map_{\SpDM_{S/Z}}(W, X_{i}) \\
     & \simeq \Map_{\SpDM_{S/Z}}(W, X_{1} \times_{Z} \cdots \times_{Z}
     X_{n}).
    \end{split}
  \]
  If $\pi_{i}$ denotes the projection $X_{1} \times_{Z} \cdots \times_{Z}
     X_{n} \to X_{i}$, then $\epsilon \simeq \coprod_{i} \pi_{i}$. 
  Now given $\mathcal{F} \in \Perf(X)$ corresponding to
  $\mathcal{F}_{i} \in \Perf(X_{i})$ under the equivalence $\Perf(X)
  \simeq \prod_{i} \Perf(X_{i})$,
  we can write
  \[ \nabla_{\otimes}p_{*}\mathcal{F} \simeq p_{1,*}\mathcal{F}_{1}
    \otimes \cdots \otimes p_{n,*}\mathcal{F}_{n},\]
  \[ g_{*}\tilde{\nabla}_{\otimes}\epsilon^{*}\mathcal{F} \simeq
    g_{*}(\pi_{1}^{*}\mathcal{F}_{1} \otimes \cdots \otimes
    \pi_{n}^{*}\mathcal{F}_{n}),\]
  with the distributivity map $\nabla_{\otimes}p_{*}\mathcal{F} \to
  g_{*}\tilde{\nabla}_{\otimes}\epsilon^{*}\mathcal{F}$ given by the
  composite
  \[
    \begin{split}
    p_{1,*}\mathcal{F}_{1}
    \otimes \cdots \otimes p_{n,*}\mathcal{F}_{n} & \to
    g_{*}g^{*}(p_{1,*}\mathcal{F}_{1}
    \otimes \cdots \otimes p_{n,*}\mathcal{F}_{n}) \\
    & \simeq g_{*}(g^{*}p_{1,*}\mathcal{F}_{1}
    \otimes \cdots \otimes g^{*}p_{n,*}\mathcal{F}_{n})  \\
    & \simeq g_{*}(\pi_{1}^{*}p_{1}^{*}p_{1,*}\mathcal{F}_{1}
    \otimes \cdots \otimes \pi_{n}^{*}p_{n}^{*}p_{n,*}\mathcal{F}_{n})
    \\
    & \to g_{*}(\pi_{1}^{*}\mathcal{F}_{1}
    \otimes \cdots \otimes \pi_{n}^{*}\mathcal{F}_{n}).
    \end{split}
  \]
  That this is an equivalence now follows from base change and the
  projection formula (which applies for maps in $\FP'$ by \cite[Remark
  3.4.2.6]{SAG}). To keep the notation bearable we spell this out only
  in
  the case $n = 2$, where it follows from \cref{rmk:dist=adj} that the
  distributivity condition is equivalent to the following commutative
  square being right adjointable (where $V:= X_{1} \times_{Z} X_{2}$):
  \[
    \begin{tikzcd}
      \Perf(Z) \times \Perf(Z) \arrow{r}{p_{1}^{*} \times p_{2}^{*}}
      \arrow{dd}{\otimes} & \Perf(X_{1}) \times \Perf(X_{2})
      \arrow{d}{\pi_{1}^{*} \times \pi_{2}^{*}} \\
      & \Perf(V) \times \Perf(V) \arrow{d}{\otimes} \\
      \Perf(Z) \arrow{r}{g^{*}} & \Perf(V).
    \end{tikzcd}
  \]
  We can decompose this diagram as follows:
  \[
    \begin{tikzcd}
      \Perf(Z) \times \Perf(Z) \arrow{r}{p_{1}^{*} \times \id}
      \arrow{ddd}{\otimes} & \Perf(X_{1}) \times \Perf(Z) \arrow{r}{\id
        \times p_{2}^{*}}  \arrow{d}{\id \times p_{1}^{*}} & \Perf(X_{1}) \times \Perf(X_{2})
      \arrow{d}{\id \times \pi_{2}^{*}} \\
      & \Perf(X_{1}) \times \Perf(X_{1}) \arrow{r}{\id \times
        \pi_{1}^{*}} \arrow{dd}{\otimes} &
      \Perf(X_{1}) \times \Perf(V) \arrow{d}{\pi_{1}^{*} \times \id}
      \\
      & & \Perf(V) \times \Perf(V) \arrow{d}{\otimes} \\
      \Perf(Z) \arrow{r}{p_{1}^{*}} &  \Perf(X_{1})
      \arrow{r}{\pi_{1}^{*}} & \Perf(V).
    \end{tikzcd}
  \]
  Since horizontal and vertical pastings of right adjointable squares
  are again right adjointable, it suffices to check the three smaller
  squares in this diagram are all right adjointable. This is true
  since the mate of the
  left square is the projection formula transformation for $p_{1}$,
  \[ p_{1,*}(\blank) \otimes \blank \,\to\, p_{1,*}(\blank \otimes
    p_{1}^{*}(\blank)),\]
  while the mate of the bottom right square is the projection formula
  transformation for $\pi_{1}$, and finally the mate of the top right
  square is $\id_{\Perf(X_{1})}$ times the base change transformation
  \[ p_{1}^{*}p_{2,*} \to \pi_{1,*}\pi_{2}^{*}\]
  corresponding to the pullback  square
    \[
      \begin{tikzcd}
        V \arrow{r}{\pi_{1}} \arrow{d}{\pi_{2}} &
        X_{1} \arrow{d}{p_{1}} \\
        X_{2} \arrow{r}{p_{2}} & Z.
      \end{tikzcd}
    \]

  We have shown that the functor
  \[ \Perf \colon \Span_{\fet}(\SpDM_{S}) \to \CatI \]
  is right $\FP'$-adjointable, and it therefore extends canonically to
 a functor of \itcats{}
 \[ \BISPAN_{\fet,\FP'}(\SpDM_{S}) \to \CATI \]
 by \cref{thm:main2}.
\end{proof}

\section{Norms in algebraic $K$-theory}\label{sec:tambara}

\subsection{Algebraic $K$-theory and polynomial functors}
Our goal in this final section is to combine our results so far with
recent work of Barwick, Glasman, Mathew, and
Nikolaus~\cite{polynomials} in order to construct additional structure
on algebraic $K$-theory spectra. In this subsection we will review the
polynomial functoriality of $K$-theory constructed in
\cite{polynomials}.

Assume that $\mathcal{C}, \mathcal{D}$ are small $\infty$-categories,
then a functor
\[f\colon \mathcal{C} \rightarrow \mathcal{D}\] is said to be a
\emph{polynomial functor} if it is $n$-excisive in the sense of
Goodwillie calculus \cite{goodwillie} for some $n$. For our purposes
it is more convenient to follow the inductive definition in \cite[Definition 2.4, Definition
2.11]{polynomials}, which is based on work of Eilenberg and Maclane~\cite{eilenberg-maclane}:

\begin{defn}\label{def:polynomial}
  Let $\mathcal{A}$ and $\mathcal{B}$ be additive $\infty$-categories and
  assume that $\mathcal{B}$ is idempotent-complete. Then we
  inductively define what it means for a functor
  $F \colon \mathcal{A} \rightarrow \mathcal{B}$ to be \emph{polynomial of
    degree $\leq n$}:
  \begin{itemize}
  \item if $n = -1$ then $f$ must be the zero functor;
  \item if $n = 0$ then $f$ must be constant;
  \item if $n > 0$ then for each fixed $x \in \mathcal{A}$, the functor
    \[
      D_x(f): \mathcal{A} \rightarrow \mathcal{B}, \qquad y \mapsto \mathrm{Fib}(F(x \oplus y) \rightarrow F(y)),
    \]
    must be polynomial of degree $\leq n-1$. (This fibre exists since
    $\mathcal{B}$ is assumed to be idempotent-complete: it is the
    complementary summand to $F(y)$ in $F(x \oplus y)$.)
\end{itemize}
\end{defn}

\begin{remark}
  Via the comparison result of \cite[Proposition 2.15]{polynomials},
  this definition of polynomial functors agrees with the one via
  Goodwillie calculus for all idempotent-complete stable
  $\infty$-categories.   
\end{remark}

\begin{warning}
  The notion of polynomial functor from \cref{def:polynomial}, which
  makes sense in \emph{additive} contexts, is completely unrelated to
  the concept with the same name that we considered in
  \cref{defn:Spoly}, which only exists for slices of the \icat{}
  $\mathcal{S}$.
\end{warning}

\begin{remark}\label{rmk:exactispoly}
  Any exact functor between stable \icats{} is polynomial of degree
  $\leq 1$. Note that this applies in particular to any functor that
  has a left or right adjoint.
\end{remark}

\begin{propn}\label{propn:coprodpoly}
  Suppose $F,G \colon \mathcal{C} \to \mathcal{D}$ are polynomial of
  degree $\leq n$. Then $F \oplus G \colon \mathcal{C} \to
  \mathcal{D}$ is also polynomial of degree $\leq n$.
\end{propn}
\begin{proof}
  This is a special case of \cite[Lemma 5.24(3)]{norms}.
\end{proof}

\begin{defn}
  We have the (non-full) subcategory
\[
 \CatI^{\poly} \subset \CatI
\]
whose objects are the small, idempotent-complete stable
$\infty$-categories and whose morphisms are the polynomial functors between them.%
\end{defn}

Recall that (connective) algebraic $K$-theory can be defined as a
functor from stable \icats{} to spectra. Passing to the underlying
infinite loop spaces, we get a functor
\[
  \Omega^{\infty}K \colon \CatI^{\stab} \rightarrow \mathcal{S},
\]
where $\CatI^{\stab}$ denotes the \icat{} of stable \icats{} and exact
functors. This is equipped with a transformation $(\blank)^{\simeq}
\rightarrow \Omega^{\infty}K$ that exhibits $\Omega^{\infty}K$ as the
universal additivization of $(\blank)^{\simeq}$ in the sense of
\cite{BarwickK,BlumbergGepnerTabuada}. The main result of
\cite{polynomials} shows that if we restrict to the full subcategory
$\CatI^{\stab,\txt{idem}}$ of idempotent-complete stable \icats{},
then we can extend this to be functorial in all polynomial functors
(rather than only the exact ones):
\begin{thm}[Barwick, Glasman, Mathew, Nikolaus] \label{thm:bgmn}
  The space-valued $K$-theory functor extends to a functor
  $\Omega^{\infty}K^{\poly} \colon   \CatI^{\poly}  \rightarrow
  \mathcal{S}$ rendering the following diagram commutative
  \[
\begin{tikzcd}
\CatI^{\stab,\txt{idem}} \ar{r}{\Omega^{\infty}K} \ar{d} & \mathcal{S}\\
\CatI^{\poly} \ar[dashed,swap]{ur}{\Omega^{\infty}K^{\poly}}. &
\end{tikzcd}
\]
(We will usually just refer to this extension as $\Omega^{\infty}K$.)
\end{thm}

We want to apply this to construct additional norms in $K$-theory, in
the following way:
\begin{cor}\label{cor:Knorm}
  Suppose $(\mathcal{C}, \mathcal{C}_{F},\mathcal{C}_{L})$ is a bispan
  triple, and that $\Phi \colon \Span_{F}(\mathcal{C}) \to \CatI$ is a
  functor such that
  \begin{enumerate}[(1)]
  \item $\Phi(X)$ is an idempotent-complete stable \icat{} for every
    $X \in \mathcal{C}$,
  \item\label{it:fstarpoly} $f^{\ostar} \colon \Phi(Y) \to \Phi(X)$ is a polynomial
    functor for every $f \colon X \to Y$ in $\mathcal{C}$,
  \item $f_{\otimes} \colon \Phi(X) \to \Phi(Y)$  is a polynomial
    functor for every $f \colon X \to Y$ in $\mathcal{C}_{F}$,
  \item $\Phi$ is $L$-distributive.
  \end{enumerate}
  Then $\Phi$ induces a functor
  \[ \Bispan_{F,L}(\mathcal{C}) \to
    \mathcal{S},\]
  which takes $X \in \mathcal{C}$ to $\Omega^{\infty}K(\Phi(X))$. 
\end{cor}
\begin{proof}
  Since $\Phi$ is $L$-distributive, it extends to a functor
  $\BISPAN_{F,L}(\mathcal{C}) \to \CATI$ by \cref{thm:main2}. The
  induced functor
  $\Bispan_{F,L}(\mathcal{C}) \to \CatI$ on underlying \icats{} then factors through the
  subcategory $\CatI^{\poly}$: on objects and for the maps of the form
  $f^{\ostar}$ and $f_{\otimes}$ this is true by assumption, and for
  $f_{\oplus}$ because this is a left adjoint and so polynomial by
  \cref{rmk:exactispoly}. We can then combine the resulting functor
  with that of \cref{thm:bgmn} to complete the proof.
\end{proof}

\begin{remark}
  If $\mathcal{C}_{L} = \mathcal{C}$, then condition \ref{it:fstarpoly}
  is automatic, since $f^{\ostar}$ is a right adjoint for all $f$.
\end{remark}

\subsection{Polynomial functors from distributivity}\label{sec:poly-p}
In order to apply \cref{cor:Knorm} in practice, we need to know that
the functors $p_{\otimes}$ are polynomial for $p$ in
$\mathcal{C}_{F}$. Our goal in this section is to derive a convenient
criterion for this using distributivity, by generalizing an argument
due to Bachmann--Hoyois~\cite[Section 5.5]{norms}. We start by introducing some
conditions and extra structure on bispan triples:

\begin{defn}\label{def:extensive}
  A span pair $(\mathcal{C}, \mathcal{C}_F)$ is said to be an
  \emph{extensive span pair} if
  \begin{enumerate}[(1)]
  \item the \icat{} $\mathcal{C}$ is extensive,
  \item $\mathcal{C}_F$ is closed under
    coproducts,
  \item for any $x \in \mathcal{C}$ the unique morphism
    $\emptyset \rightarrow x$ is in $\mathcal{C}_{F}$,
  \item for any $x\in \mathcal{C}$ the fold map
    $\nabla \colon x \coprod x \rightarrow x$ is in $\mathcal{C}_F$.
  \end{enumerate}
\end{defn}

\begin{remark}
  If $(\mathcal{C}, \mathcal{C}_{F})$ is an extensive span pair, then
  the symmetric monoidal structure on $\Span_{F}(\mathcal{C})$ induced
  by the coproduct in $\mathcal{C}$ via \cref{ex:spancocart} is both
  cartesian and cocartesian. In other words, in this case the
  coproduct in $\mathcal{C}$ gives both the product and coproduct in
  $\Span_{F}(\mathcal{C})$.
\end{remark}

\begin{defn}\label{def:degree}
  Suppose 
  $(\mathcal{C}, \mathcal{C}_F)$ is an extensive span pair. Then a
  \emph{degree structure} on $(\mathcal{C}, \mathcal{C}_F)$ consists
  of a collection of morphisms
  $F_{n} \subseteq \Map([1], \mathcal{C}_{F})$ for each
  $n \in \mathbb{N}$, called \emph{morphisms of degree $n$}, such
  that:
\begin{enumerate}[(1)]
\item for each morphism $f\colon x \rightarrow y$ in $\mathcal{C}_F$,
  there exists a natural number $N < \infty$ and an essentially unique (finite) coproduct decomposition (called the \emph{degree decomposition})
\[
y \simeq \coprod^N_{i=0} y_n^{(f)},
\] such that for each $0 \leq n \leq N$, the morphism
\[
f_n:=f \times_y y_n^{(f)}\colon x^{(f)}_n = x \times_y y_n^{(f)} \to y_n^{(f)}
\]
is of degree $n$.
\item The collection $F_{0}$ consists of the morphisms $\emptyset \rightarrow X$.
\item\label{it:degbc} Morphisms in $F_n$ are stable under base change: if $f\colon x
  \rightarrow y$ is degree $n$, then for any $w \rightarrow y$, the
  morphism $w \times_y x \rightarrow w$ is also of degree $n$.
\item\label{it:degadd} Given morphisms $f\colon x \rightarrow y, g\colon z \rightarrow
  y$ in $\mathcal{C}_F$ which are of degrees $n$ and $m$,
  respectively, then the morphism
\[
x \coprod z \xrightarrow{f \coprod g} y \coprod y \xrightarrow{\nabla} y
\]
 in $\mathcal{C}_F$ is of degree $m + n$. 
\end{enumerate}
Furthermore, given a degree structure, we say a morphism
$f \colon x \to y$ is of \emph{degree $\leq n$} if in the degree
decomposition we have $y_{i}^{(f)} \simeq \emptyset$ for $i > n$.  We
also say that a morphism is of \emph{degree $-1$} if it is the
essentially unique morphism $\emptyset \rightarrow \emptyset$.
\end{defn}

\begin{remark}\label{rem:degree}
  Suppose that $(\mathcal{C}, \mathcal{C}_F)$ is equipped with a
  degree structure. We note that not all morphisms have a well-defined
  degree although each morphism does have a degree
  decomposition. Additionally, the reader is encouraged to think
  of $x^{(f)}_0$ as the ``locus'' in $x$ where $f$ has empty fibres,
  \ie{} fails to be surjective.
\end{remark}

\begin{ex}\label{ex:degree}
  Here are the main examples of degree structures on the span pairs that
  have appeared throughout this paper:
  \begin{enumerate}
  \item Consider the extensive span pair $(\xF, \xF)$. There is a
    degree structure where a morphism of finite sets
    $f\colon x \rightarrow y$ is of degree $n$ when the fibres all
    have cardinality exactly $n$. %
  \item Consider the extensive span pair
    $(\mathcal{S}, \mathcal{S}_{\fin})$. There is a degree structure
    where a morphism
    $f\colon x \rightarrow y$ in  $\mathcal{S}_{\fin}$ is of degree $n$ if all its fibres have
    cardinality exactly $n$.
  \item Consider the extensive span pair $(\xF_{G}, \xF_{G})$. Then we
    say that morphism of finite $G$-sets $f\colon x \rightarrow y$ is
    of degree $n$ when (the underlying sets of) the fibres all have
    cardinality exactly $n$.
  \item Consider the extensive algebro-geometric span pairs
    $(\Sch_{S}, \Sch_{S}^{\flf})$, $(\Sch_{S}, \Sch_{S}^{\fet})$,
    $(\AlgSpc_{S}, \AlgSpc_{S}^{\flf})$,
    $(\AlgSpc_{S}, \AlgSpc_{S}^{\fet})$ and
    $(\SpDM_S, \SpDM^{\fet}_S)$. In each of these, the degree of a
    morphism $f\colon X \rightarrow Y$ can be defined to be $n$ if the
    sheaf of finite, locally free $\mathcal{O}_Y$-modules given by
    $f_*(\mathcal{O}_X)$ is of constant rank $n$.
\end{enumerate}
\end{ex}

\begin{remark}\label{rem:equiv} In all the examples from
  \cref{ex:degree}, the degree 1 morphisms are precisely the
  equivalences. This is not assumed for a general degree structure
  since we do not need it for our main results.
\end{remark}

\begin{construction}\label{const:folddegdecomp} 
We will need to compute how the degree decomposition interacts with
coproducts of morphisms. So let
$f\colon x \rightarrow y, g\colon z \rightarrow y$ be two
morphisms. Each morphism then induces a coproduct decomposition (where
the index runs through a finite set)
\[
y \simeq \coprod_n y_n^{(f)} \qquad y \simeq \coprod_m y_m^{(g)}.
\]
We set 
\[
y_{mn}:= y_n^{(f)} \times_{y} y_m^{(g)}.
\]
From this, we can form the following diagram where each square is cartesian
    \begin{equation*}
    \begin{tikzcd}
       & z_{mn}  \ar{r} \ar{d} & z_{m} \ar{d}{g_m} \\
    x_{mn} \ar{r} \ar{d}   & y_{mn} \ar{r} \ar{d} & y_{m}^{(g)} \arrow[hook]{d} \\
     x_n \ar{r}{f_n}  & y_n^{(f)}  \arrow[hook]{r} & y.
    \end{tikzcd}
  \end{equation*}
  \end{construction}

  \begin{lemma}\label{lem:deg-coprod}
    Let $(\mathcal{C}, \mathcal{C}_F)$ be an extensive pair with a
    degree structure. Let
    $f\colon x \rightarrow y, g\colon z \rightarrow y$ be two
    morphisms in $\mathcal{C}_F$. Then in the degree decomposition of the map
    $\nabla_y \circ (f \amalg g) \colon x \amalg z \rightarrow y$,
    the degree-$k$ component is
\[
\coprod_{m+n=k} (x_{mn} \amalg z_{mn}) \rightarrow \coprod_{m+n=k} y_{mn},
\]
where $z_{mn} \rightarrow y_{mn}$ is of degree $m$ and $x_{mn}
\rightarrow y_{mn}$ is of degree $n$.
\end{lemma}

\begin{proof}
  Since coproduct decompositions are stable under pullbacks we have
  that for each $n \in \mathbb{N}$,
\[
y_n^{(f)} \simeq y \times_y y_n^{(f)} \simeq (\coprod_{m} y_m^{(g)}) \times_y y_n^{(f)} \simeq \coprod_m y_{mn}.
\]
By \ref{it:degbc} in \cref{def:degree} the pullback $x_{mn}
\to y_{mn}$ of $f$ is therefore also of degree $n$. Similarly, the
pullback of $g$ to $z_{mn} \to y_{mn}$ is of degree $m$, and so by 
\ref{it:degadd} the composite
\[ x_{mn} \amalg z_{mn} \to y_{mn} \amalg y_{mn} \to y_{mn} \]
is of degree $m+n$. Now note that we have
\[y \simeq \coprod_{n} y_{n}^{(f)} \simeq \coprod_{n}\coprod_{m}
  y_{mn} \simeq \coprod_k \coprod_{m+n=k} y_{mn}.\]
Here $\nabla_{y} \circ (f \amalg g)$ restricts over $y_{mn}$ to the
map
$x_{mn} \amalg z_{mn} \to y_{mn}$ of degree $m+n$, so this is a
decomposition into components of fixed degree; by uniqueness this is
therefore the degree composition.
\end{proof}

The purpose of a degree structure as in \cref{def:degree} is
to allow us to prove that certain functors are polynomial by induction
on degrees using distributivity, by abstracting the arguments in
\cite[Section 5.5]{norms}.

\begin{defn}
  We say a bispan triple $(\mathcal{C},
  \mathcal{C}_{F},\mathcal{C}_{L})$ is an \emph{extensive} bispan
  triple if both $(\mathcal{C},\mathcal{C}_{F})$ and
  $(\mathcal{C},\mathcal{C}_{L})$ are extensive span pairs.
\end{defn}

\begin{remark}
  If $(\mathcal{C}, \mathcal{C}_{F},\mathcal{C}_{L})$ is an extensive
  bispan triple, then the symmetric monoidal structure on
  $\Bispan_{F,L}(\mathcal{C})$ induced by the coproduct in
  $\mathcal{C}$ as in \cref{ex:bispancoprod} is cartesian, \ie{} the
  coproduct in $\mathcal{C}$ gives a product in
  $\Bispan_{F,L}(\mathcal{C})$.
\end{remark}

\begin{construction}\label{const:distcoprodext}
  Let $(\mathcal{C},\mathcal{C}_{F},\mathcal{C}_{L})$ be an extensive
  bispan triple, where furthermore the \icat{} $\mathcal{C}$ is
  idempotent-complete.
  Fix $p\colon x \rightarrow y$ in
  $\mathcal{C}_F$ and consider the distributivity diagram for
  $x \amalg x \xrightarrow{\nabla_x} x \xto{p} y$ in the sense of
  Definition~\ref{def:dis-square}:
    \begin{equation} \label{eq:poly-dis}
    \begin{tikzcd}
      {} & p^*w \arrow{dd}{\tilde{g}} \arrow{dl}[swap]{\epsilon} \arrow{r}{\tilde{p}} & w
      \arrow{dd}{g} \\
      x \amalg x \arrow{dr}[swap]{\nabla_x} \\
      {} & x  \arrow{r}{p} & y.
    \end{tikzcd}
  \end{equation}
  Note that, according to Remark~\ref{dist-adj}, $w \rightarrow y$ is
  equivalent to $p_*(x \amalg x \rightarrow x)$. Since $\mathcal{C}$
  is idempotent-complete, we have a decomposition, over $y$:
  \[
    w \simeq y \amalg c \amalg y.
  \]
  Indeed, we can produce two sections
  $s_0, s_1 \colon y \rightarrow w \simeq p_*(x \amalg x \rightarrow x)$
  adjoint to the two coproduct inclusions $x \rightarrow x \amalg
  x$. Since coproduct decompositions are preserved under pullbacks, we
  get $p^*w \simeq x \amalg (c \times_y x) \amalg x$. Since
  coproducts are disjoint, we can further decompose the restriction of
  $\epsilon$ to 
  $c \times_y x \rightarrow x \amalg x$ as a coproduct of two
  morphisms:
\[
\epsilon_L \amalg \epsilon_R: c \times_y x \simeq c_L \amalg c_R \rightarrow x \amalg x. 
\]
Restricting $\widetilde{p}$ to $c_L$ and $c_R$ gives us two maps
\[
\widetilde{p}_L :c_L \rightarrow c \qquad \widetilde{p}_R:c_R \rightarrow c.
\]
We also have the restriction of $g$ to $c$:
\[
k: c \rightarrow y.
\]

All in all~\eqref{eq:poly-dis} is equivalent to:
    \begin{equation} \label{eq:poly-dis2}
    \begin{tikzcd}
      {} & x \amalg c_L \amalg c_R \amalg x \arrow{dd}{} \arrow{dl}[swap]{} \arrow{r}{} & y \amalg c \amalg y
      \arrow{dd}{} \\
      x \amalg x \arrow{dr}{\nabla} \\
      {} & x  \arrow{r}{p} & y.
    \end{tikzcd}
  \end{equation}
  
  We are particularly concerned with the following diagram which we can extract from~\eqref{eq:poly-dis2}, where the middle square is cartesian:
      \begin{equation} \label{eq:poly-dis2part}
    \begin{tikzcd}
      {} & c_L \amalg c_R \arrow{dd} \arrow{dl}[swap]{\epsilon_L \amalg \epsilon_R} \arrow{r}{\widetilde{p}|_{c \times_y x}} & c 
      \arrow{dd}{k} \\
      x \amalg x \arrow{dr}{\nabla} \\
      {} & x  \arrow{r}{p} & y.
    \end{tikzcd}
  \end{equation}  
\end{construction}

The next lemma roughly states that if $p$ is ``surjective'' then each
of $\widetilde{p}_L, \widetilde{p}_R$ must be as well. This will be used to
prove polynomiality of $p_{\otimes}$ via an inductive argument.
    
\begin{lemma}\label{lem:deg}
  Keeping the notation of \cref{const:distcoprodext}, assume
  furthermore that $(\mathcal{C}, \mathcal{C}_{F})$ is equipped with a
  degree structure, and $y^{(p)}_0 \simeq \emptyset$. For $i = L,R$, we
  have $c_0^{(\widetilde{p}_i)} \simeq \emptyset$.
\end{lemma}

\begin{proof}
  To show that $\widetilde{p}_{L}$ has no component of degree $0$, let
  us decompose $c$ as 
  \[ c \simeq c_{0}^{(\widetilde{p}_{L})} \amalg
    c_{>0}^{(\widetilde{p}_{L})}. \]
  Then $p^{*}c_{0}^{(\widetilde{p}_{L})} \to p^{*}c \simeq c_{L} \amalg
  c_{R}$ factors through $c_{R}$ (since by definition its component
  over $c_{L}$ is $\emptyset$), and hence the composite
  $p^{*}c_{0}^{(\widetilde{p}_{L})} \to x \amalg x$ factors through the
  right copy of $x$. But then the adjoint map
  $c_{0}^{(\widetilde{p}_{L})} \to p_{*}(x \amalg x) \simeq y \amalg c
  \amalg y$ factors through the right copy of $y$, which means
  $c_{0}^{(\widetilde{p}_{L})} \simeq \emptyset$ since it also factors
  through $c$.
  \end{proof}
  
  For the remainder of this section, we fix an extensive bispan triple
  $(\mathcal{C}, \mathcal{C}_{F}, \mathcal{C}_{L})$ such that
  \begin{itemize}
  \item $\mathcal{C}$ is idempotent-complete,%
  \item $(\mathcal{C}, \mathcal{C}_{F})$ has a degree structure,
  \end{itemize}
  and an $L$-distributive
  functor $\Phi \colon \Span_{F}(\mathcal{C}) \to \CatI$ such that
  \begin{enumerate}[(1)]
  \item $\Phi$ preserves finite products,
  \item for each $x \in \mathcal{C}$, the \icat{} $\Phi(x)$ is
    additive,
  \end{enumerate}
  The following computation follows \cite[Corollary 5.15]{norms} closely:  
  \begin{lemma}\label{lem:computation}
  For any $p \colon x
  \rightarrow y$ in $\mathcal{C}_F$, we have for any $E, F \in \Phi(x)$ an
  equivalence
  \[
    p_{\otimes}(E \oplus F) \simeq p_{\otimes}(E) \,\oplus\, k_{\oplus}\nabla_{c,\otimes}\left(\widetilde{p}_{L,\otimes}\epsilon_L^{\ostar}(E),\widetilde{p}_{R,\otimes}\epsilon_R^{\ostar}(F)\right) \,\oplus\, p_{\otimes}(F)
  \]
  in terms of \cref{eq:poly-dis2}.
\end{lemma}

\begin{proof}
  The distributivity transformation (\cref{def:L-dis}) for \eqref{eq:poly-dis} gives an equivalence:
  \[
    p_{\otimes}\nabla_{\oplus} \simeq
    g_{\oplus}\tilde{p}_{\otimes}\epsilon^{\ostar}.\]
  The claim then follows from this equivalence, the transitivity of
  all the functors involved, and the identification of
  $\nabla_{\oplus}$ with the direct sum functor, which follows from
  the assumption that $F$ is product-preserving.
\end{proof}

\begin{remark}\label{rmk:folddist}
  For any morphism $\phi \colon y \to x$ in $\mathcal{C}_{L}$, it is
  easy to see that the
  following is a distributivity diagram:
  \[
    \begin{tikzcd}
      {} & y \amalg y \arrow{dl}[swap]{\id \amalg \phi} \arrow{dd}{\phi
        \amalg \phi} 
      \arrow{r}{\nabla_{y}} & y \arrow{dd}{\phi} \\
      y \amalg x \arrow{dr}{\phi \amalg \id} \\
       & x \amalg x \arrow{r}{\nabla_{x}} & x,
    \end{tikzcd}
  \]
  since for $\alpha \colon z \to x$ we have an equivalence
  \[ \left\{
      \begin{tikzcd}
        z \amalg z \arrow{rr} \arrow{dr}[swap]{\alpha \amalg \alpha} & & y
        \amalg x \arrow{dl}{\phi \amalg \id} \\
         & x \amalg x
      \end{tikzcd}
      \right\}
      \simeq
      \left\{
      \begin{tikzcd}
        z  \arrow{rr} \arrow{dr}[swap]{\alpha} & & y
        \arrow{dl}{\phi} \\
        & x
      \end{tikzcd}
      \right\}.
    \]
  It follows that for our $L$-distributive functor $\Phi$ we have an equivalence
  \[ \nabla_{x,\otimes}(\phi_{\oplus}, \id) \simeq
    \phi_{\oplus}\nabla_{y,\otimes}(\id, \phi^{\ostar})\]
  of functors $\Phi(y)\times \Phi(x) \to \Phi(x)$. If we take $\phi =
  \nabla_{x}$ then $\nabla_{x}^{\ostar}$ is the diagonal $\Phi(x) \to
  \Phi(x) \times \Phi(x)$ (since $\Phi$ by assumption preserves
  products) and hence its left adjoint $\nabla_{x,\oplus}$ is the
  coproduct on $\Phi(x)$. From this the previous equivalence
  specializes for $E,E',F \in \Phi(x)$ to a natural equivalence
  \[ \nabla_{x,\otimes}(E \oplus E', F) \simeq \nabla_{x,\otimes}(E,F)
    \oplus \nabla_{x,\otimes}(E',F),\]
  so that the functor $\nabla_{x,\otimes}$ preserves coproducts
  in each variable.
\end{remark}

\begin{propn}\label{propn:polynom}
  Suppose $p \colon x \rightarrow y$ is a morphism of degree $\leq n$ in
  $\mathcal{C}_F$ (for $n \geq -1$). Then the functor $p_{\otimes}
  \colon \Phi(x) \to \Phi(y)$ is
  polynomial of degree $\leq n$.
\end{propn}

\begin{proof}
  By convention, a degree-$(-1)$ morphism is given by
  $\emptyset \rightarrow \emptyset$ and thus defines the zero functor,
  a polynomial functor of degree $-1$.  We now proceed by
  induction. Since any morphism $\emptyset \rightarrow y$ of degree
  zero induces a constant functor
  $\Phi(\emptyset) \simeq \ast \rightarrow \Phi(y)$, a degree-0
  morphism gives a polynomial functor of degree $\leq 0$. Now let us assume
  that the result has been proved for any morphism of degree
  $\leq n-1$ for $n \geq 1$. Consider the
  diagram~\eqref{eq:poly-dis2}. Since degree-$n$ morphisms are stable
  under pullback, the morphism
  $\widetilde{p}_L \amalg \widetilde{p}_R \colon c_L \amalg c_R \rightarrow
  c$ is also of degree $n$. Since $\Phi$ preserves coproduct
  decompositions in
  $\mathcal{C}$ (which are the products in bispans) by assumption,
  coproducts of polynomials of degree $\leq n$ are likewise polynomial
  of degree $\leq n$ by \cref{propn:coprodpoly}, and we already proved
  the case of degree $0$, we may assume that $y^{(p)}_0 \simeq \emptyset$.

Now, $c_L \amalg c_R \rightarrow c$ is the composite
\begin{equation}
  \label{eq:clrcomp}
c_L \amalg c_R \xrightarrow{\widetilde{p}_L \amalg \widetilde{p}_R} c \amalg c \xrightarrow{\nabla_{c}} c.   
\end{equation}
We claim that here $\widetilde{p}_L$ and $\widetilde{p}_R$ are both of
degree $\leq n-1$ --- indeed, since both maps have no component of
degree $0$ it follows from the description of the
degree decomposition of \cref{eq:clrcomp} in
\cref{const:folddegdecomp} and the additivity condition
\ref{it:degadd} in \cref{def:degree} that a component of degree $\geq
n$ in either map would produce a component of degree $> n$ in
\cref{eq:clrcomp}, which is impossible.

Hence by the inductive hypothesis $(\widetilde{p}_L)_{\otimes}$ and
$(\widetilde{p}_R)_{\otimes}$ are polynomial functors of
degree $\leq n-1$. 
To conclude, we fix $E$ and note that Lemma~\ref{lem:computation} yields an equivalence:
\[
  D_E(p_{\otimes})(\blank) \simeq
p_{\otimes}(E) \oplus
k_{\oplus}\nabla_{c,\otimes}\left(\widetilde{p}_{L,\otimes}\epsilon_L^{\ostar}(E),\widetilde{p}_{R,\otimes}\epsilon_R^{\ostar}(\blank)\right)
\]
Here $\nabla_{c,\otimes}(X,\blank)$ preserves finite coproducts by
\cref{rmk:folddist}, as does $k_{\oplus}$ since it is a left adjoint,
so the composite
$k_{\oplus}\nabla_{c,\otimes}\left(\widetilde{p}_{L,\otimes}\epsilon_L^{\ostar}(E),\widetilde{p}_{R,\otimes}\epsilon_R^{\ostar}(\blank)\right)$
is polynomial of degree $\leq n-1$ by \cite[Lemma
5.24(4)]{norms}. Using \cref{propn:coprodpoly} again we get that
$D_E(p_{\otimes})(\blank)$ is polynomial of degree $\leq n-1$ and
hence $p_{\otimes}$ is polynomial of degree $\leq n$.
\end{proof}

Since any morphism  in $\mathcal{C}_{F}$ has a degree decomposition
consisting of finitely many terms and so is of degree $\leq n$ for
some $n$, we have shown:
\begin{cor}\label{cor:polynom}
  Let $p \colon x \rightarrow y$ be a
  morphism in $\mathcal{C}_F$, then $p_{\otimes}$ is a polynomial
  functor.
\end{cor}

\begin{ex}
  \cref{cor:polynom} applies for instance for any finite group $G$ to the bispan triple
  $(\xF_{G},\xF_{G},\xF_{G})$ of finite $G$-sets and the $G$-symmetric
  monoidal \icat{} of $G$-spectra as discussed in
  \cref{prop:genuine-g}. In this case we can in particular conclude
  that for any subgroup inclusion $H \subseteq G$, the
  Hill--Hopkins--Ravenel norm $\Sp^{H} \to \Sp^{G}$ is a polynomial
  functor. This example has already been studied in more detail by
  Konovalov~\cite{Konovalov}, though the proof of polynomiality is
  similar to ours.
\end{ex}

\subsection{From bispans to Tambara functors}
Summarizing the results of the previous two subsections, we have shown
\begin{cor}\label{cor:polyext}
  Let $(\mathcal{C}, \mathcal{C}_{F}, \mathcal{C}_{L})$ be an
  extensive bispan triple such that $\mathcal{C}$ is
  idempotent-complete and $(\mathcal{C}, \mathcal{C}_{F})$ has a
  degree structure. Suppose that $\Phi\colon \Span_{F} \to \CatI$ is an
  $L$-distributive functor such that
  \begin{itemize}
  \item $\Phi$ preserves finite products,
  \item for each $x \in \mathcal{C}$, the \icat{} $\Phi(x)$ is
    stable and idempotent-complete,
  \item for each morphism $f \colon x \to y$ in $\mathcal{C}$, the
    functor $f^{\ostar} \colon \Phi(x) \to \Phi(y)$ is
    polynomial.\footnote{This is automatic if $\mathcal{C}_{L}= \mathcal{C}$.}
  \end{itemize}
  Then $\Phi$ extends to a functor $\widetilde{\Phi} \colon
  \Bispan_{F,L}(\mathcal{C}) \to \CatI^{\poly}$, and hence induces a product-preserving
  functor $\Bispan_{F,L}(\mathcal{C}) \to \mathcal{S}$ given on
  an object $x \in \mathcal{C}$ by $\Omega^{\infty}K(\Phi(x))$.
\end{cor}
Our goal in this final subsection is to give some applications of this
result towards obtaining new structure on algebraic $K$-theory.

We first consider the equivariant situation, as previously discussed
in \S\ref{subsec:equivt}. For a finite group $G$ the category
$\xF_{G}$ of finite $G$-sets is extensive and idempotent-complete, and
by Example~\ref{ex:degree}(3) it also has a degree structure. We
therefore get:
\begin{cor}\label{cor:tambara}
  Let $\mathcal{C} \colon \Span(\xF_{G}) \to \CatI$ be a $G$-symmetric
  monoidal \icat{} that is compatible with additive transfers. If the
  \icat{} $\mathcal{C}^{H}$ is stable and idempotent-complete for
  every $H \subseteq G$, then $\mathcal{C}$ extends to a functor
  \[\widetilde{\mathcal{C}} \colon \Bispan(\xF_{G}) \to \CatI^{\poly},\]
  and hence induces a product-preserving
  functor
  \[ \Bispan(\xF_{G}) \to \mathcal{S}, \qquad G/H \mapsto \Omega^{\infty}K(\mathcal{C}^{H}).\]
\end{cor}

\begin{remark}
  A \emph{Tambara functor} \cite{Tambara} is a functor
  $T \colon \Bispan(\xF_{G}) \to \Ab$, or equivalently a functor
  $T \colon \Bispan(\xF_{G}) \to \Set$ such that the induced
  commutative monoid structure on $T(G/H)$ is a group for every
  $H \subseteq G$. The output of \cref{cor:tambara} is an \icatl{}
  analogue of this: it is a functor
  $T \colon \Bispan(\xF_{G}) \to \mathcal{S}$ such that the
  commutative monoid structure on $T(G/H)$ is grouplike for every
  $H \subseteq G$. We will refer to this as a \emph{homotopical
    Tambara functor}. Note that it is expected that connective genuine
  $G$-$E_{\infty}$-ring spectra are equivalent to these homotopical
  Tambara functors.
\end{remark}

\begin{ex}
  By \cref{prop:genuine-g} the $G$-symmetric monoidal
  \icat{} of $G$-spectra (given by $G/H \mapsto \Sp^{H}$) satisfies the
  hypotheses of \cref{cor:tambara}. Hence the $G$-equivariant
  $K$-theory of the sphere spectrum is a homotopical Tambara functor:
  the functor
\[
  \Omega^{\infty}K(\mathbb{S}_{G}) \colon\mathcal{O}_{G}^{\op}
  \rightarrow \mathcal{S} \qquad G/H \mapsto \Omega^{\infty} K(\Sp^H),
\]
extends canonically to a product-preserving functor
$ \Bispan(\xF_G) \rightarrow \mathcal{S}.$  
\end{ex}

To generalize this, we introduce some terminology:
\begin{defn}
  Let $G$ be a finite group. For us a \emph{genuine
    $G$-$\mathbb{E}_{\infty}$-ring spectrum} $E$ will be a section of the
  cocartesian fibration $\int \Sp_G \rightarrow \Span(\xF_G)$ sending
  a backward arrow to a cocartesian edge.
\end{defn}

\begin{remark}
This definition mimics the
  definition of a normed motivic spectrum in the context of motivic
  homotopy theory \cite{norms}; see especially \cite[Definition
  9.14]{norms}. As explained there, it is also equivalent to the
  classical definition of a genuine $G$-$\mathbb{E}_{\infty}$-ring
  spectrum by comparison of associated monads.
\end{remark}

By a similar argument as in \cite[Proposition 7.6(4)]{norms}, the
formation of modules assemble into a functor
\[
\Mod_E\colon\Span(\xF_G) \rightarrow \CatI, \qquad G/H \mapsto \Mod_{E^{H}}(\Sp^{H}),
\]
which satisfies the hypotheses of \cref{cor:tambara}. We then have:
\begin{thm}\label{genuinering}
  The algebraic $K$-theory of a genuine $G$-$\mathbb{E}_{\infty}$-ring
  spectrum $E$ is a homotopical Tambara functor:
  the functor
\[
  \Omega^{\infty}K(E) \colon \mathcal{O}_{G}^{\op}
  \rightarrow \mathcal{S} \qquad G/H \mapsto \Omega^{\infty} K(E^H) := \Omega^{\infty} K(\Mod_{E^H}(\Sp^{H})),
\]
extends canonically to a product-preserving functor
$ \Bispan(\xF_G) \rightarrow \mathcal{S}.$
\end{thm}

\begin{remark}\label{rem:g-stuff}
  As already mentioned, a homotopical Tambara functor is expected to
  be the same thing as a connective $G$-$\mathbb{E}_{\infty}$-ring
  spectrum. Assuming this, \cref{genuinering} says that algebraic
  $K$-theory preserves $G$-$\mathbb{E}_{\infty}$-ring structures. As
  far as we are aware, this is a completely new structure on algebraic
  K-theory --- indeed the recent paper \cite{GMMOKth} seems to be the
  first to construct even an associative ring structure valued in
  $G$-spectra, though the results of \cite{BarwickMackey2} should also
  suffice to construct $K(E)$ as an ordinary $E_{\infty}$-algebra in
  $G$-spectra.
\end{remark}

Another interesting class of examples arises from the $G$-symmetric
monoidal \icats{} obtained from group actions as in \cref{ex:g-smc}:
\begin{cor}
  Let $\mathcal{C}$ be an idempotent-complete small stable \icat{},
  equipped with a symmetric monoidal structure that is compatible with
  finite coproducts. Then the functor
\[
  \Omega^{\infty}K_{G}(\mathcal{C}) \colon \mathcal{O}_{G}^{\op}
  \rightarrow \mathcal{S} \qquad G/H \mapsto \Omega^{\infty}
  K(\Fun(BG,\mathcal{C})),
\]
extends canonically to a homotopical Tambara functor
$ \Bispan(\xF_G) \rightarrow \mathcal{S}$. \qed
\end{cor}

\begin{remark} In \cite{BarwickMackey2}, the authors proved that the
  $K$-theory of a ``na\"ive'' $G$-symmetric monoidal \icat{} admits
  the structure of a Green functor. The corollary above treats the
  case of a $G$-symmetric monoidal \icat{} where the action is trivial
  and produces a Tambara functor. We expect that a more general
  statement holds: the $K$-theory of certain ``na\"ive'' $G$-symmetric
  monoidal \icats{} should also form Tambara functors. We leave this to
  the interested reader.
\end{remark}

\begin{ex}
 As in Example~\ref{ex:g-rep}, let $R$ be an
 $\mathbb{E}_{\infty}$-ring spectrum and consider the $G$-symmetric
 monoidal \icat{} from that example, given by $G/H \mapsto \Fun(BG,
 \Perf(R))$. We obtain a spectral Tambara functor $\Bispan(\xF_{G})
 \to \mathcal{S}$ given by
 \[ G/H \mapsto K(\Fun(BG, \Perf(R))). \]
\end{ex}

Lastly, we note that plugging in our algebro-geometric examples of
bispans also gives additional structure on $K$-theory. Applying \cref{cor:polyext} to the extension of $\Perf$ to bispans as
in Theorem~\ref{thm:perfbispan} gives us the following result about the
$K$-theory of spectral Deligne--Mumford stacks:
\begin{thm}
  The $K$-theory presheaf
\[
\Omega^{\infty}K\colon \SpDM_{S}^{\op} \rightarrow \mathcal{S}, \qquad
X \mapsto \Omega^{\infty} K(\Perf(X)),
\]
canonically extends to a product-preserving functor $\Bispan_{\fet,\FP'}(\SpDM_{S}) \rightarrow \mathcal{S}$.
\end{thm}
In other words, the $K$-theory of spectral Deligne--Mumford stacks has
multiplicative norms along finite \'etale maps. On the other hand,
applying \cref{thm:mot-bispans} gives us a result about the $K$-theory
of $\SH$, which is in some sense a stable analogue of the secondary
$K$-theory explored in the thesis of R\"ondigs \cite{rondigs}:
\begin{thm}
 The $K$-theory presheaf
\[
\Omega^{\infty}K\colon \AlgSpc^{\op} \rightarrow \mathcal{S}, \qquad X \mapsto K(\SH(X)),
\]
canonically extends to a product-preserving functor $\Bispan_{\fet,\sm}(\AlgSpc_{S}) \rightarrow \mathcal{S}$.
\end{thm}

\begin{remark} The multiplicative pushforwards along finite \'etale morphisms on $\SH$ induce a kind of ``Adams operations'' on $K(\SH)$. It would be interesting to explore some computational consequences of this structure. 
\end{remark}


\begin{bibdiv}
\begin{biblist}
\bib{sga4-1}{book}{
  label={SGA4},
  title={Th\'{e}orie des topos et cohomologie \'{e}tale des sch\'{e}mas. {T}ome 1: {T}h\'{e}orie des topos},
  series={Lecture Notes in Mathematics, Vol. 269},
  note={S\'{e}minaire de G\'{e}om\'{e}trie Alg\'{e}brique du Bois-Marie 1963--1964 (SGA 4), Dirig\'{e} par M. Artin, A. Grothendieck, et J. L. Verdier. Avec la collaboration de N. Bourbaki, P. Deligne et B. Saint-Donat},
  publisher={Springer-Verlag, Berlin-New York},
  date={1972},
}

\bib{StacksProject}{webpage}{
  label={Stacks},
  title={The Stacks Project},
  url={http://stacks.math.columbia.edu},
}

\bib{ayoub-thesis1}{article}{
  author={Ayoub, Joseph},
  title={Les six op\'erations de {G}rothendieck et le formalisme des cycles \'evanescents dans le monde motivique. {II}},
  journal={Ast\'erisque},
  number={315},
  date={2007},
}

\bib{mot-tambara}{article}{
  author={Bachmann, Tom},
  title={Motivic Tambara functors},
  journal={Math. Z.},
  volume={297},
  date={2021},
  number={3-4},
  pages={1825--1852},
  eprint={arXiv:1807.02981},
}

\bib{norms}{article}{
  author={Bachmann, Tom},
  author={Hoyois, Marc},
  title={Norms in motivic homotopy theory},
  journal={Ast\'{e}risque},
  number={425},
  date={2021},
  label={BaHo21},
  eprint={arXiv:1711.03061},
}

\bib{BarwickThesis}{book}{
  author={Barwick, Clark},
  title={$(\infty ,n)$-{C}at as a closed model category},
  note={Thesis (Ph.D.)--University of Pennsylvania},
  date={2005},
}

\bib{BarwickK}{article}{
  author={Barwick, Clark},
  title={On the algebraic $K$-theory of higher categories},
  journal={J. Topol.},
  volume={9},
  date={2016},
  number={1},
  pages={245--347},
}

\bib{BarwickMackey}{article}{
  author={Barwick, Clark},
  title={Spectral {M}ackey functors and equivariant algebraic $K$-theory ({I})},
  journal={Adv. Math.},
  volume={304},
  date={2017},
  pages={646--727},
  eprint={arXiv:1404.0108},
  date={2014},
}

\bib{BDGNS1}{article}{
  author={Barwick, Clark},
  author={Dotto, Emanuele},
  author={Glasman, Saul},
  author={Nardin, Denis},
  author={Shah, Jay},
  title={Parametrized higher category theory and higher algebra: A general introduction},
  eprint={arXiv:1608.03654},
  date={2016},
}

\bib{BarwickMackey2}{article}{
  author={Barwick, Clark},
  author={Glasman, Saul},
  author={Shah, Jay},
  title={Spectral Mackey functors and equivariant algebraic $K$-theory, II},
  journal={Tunis. J. Math.},
  volume={2},
  date={2020},
  number={1},
  pages={97--146},
  eprint={arXiv:1505.03098},
}

\bib{polynomials}{article}{
  author={Barwick, Clark},
  author={Glasman, Saul},
  author={Mathew, Akhil},
  author={Nikolaus, Thomas},
  title={K-theory and polynomial functors},
  date={2021},
  eprint={arXiv:2102.00936},
}

\bib{BarwickSchommerPriesUnicity}{article}{
  eprint={arXiv:1112.0040},
  author={Barwick, Clark},
  author={Schommer-Pries, Christopher},
  title={On the unicity of the theory of higher categories},
  journal={J. Amer. Math. Soc.},
  volume={34},
  date={2021},
  number={4},
  pages={1011--1058},
}

\bib{BlumbergGepnerTabuada}{article}{
  author={Blumberg, Andrew J.},
  author={Gepner, David},
  author={Tabuada, Gon\c {c}alo},
  title={A universal characterization of higher algebraic $K$-theory},
  journal={Geom. Topol.},
  volume={17},
  date={2013},
  number={2},
  pages={733--838},
}

\bib{BlumbergHillNinfty}{article}{
  label={BlHi15},
  author={Blumberg, Andrew J.},
  author={Hill, Michael A.},
  title={Operadic multiplications in equivariant spectra, norms, and transfers},
  journal={Adv. Math.},
  volume={285},
  date={2015},
  pages={658--708},
  issn={0001-8708},
  review={\MR {3406512}},
  doi={10.1016/j.aim.2015.07.013},
}

\bib{blumberg-hill}{article}{
  label={BlHi18},
  author={Blumberg, Andrew J.},
  author={Hill, Michael A.},
  title={Incomplete {T}ambara functors},
  journal={Algebr. Geom. Topol.},
  volume={18},
  date={2018},
  number={2},
  pages={723--766},
}

\bib{BohmannNorm}{article}{
  author={Bohmann, Anna Marie},
  title={A comparison of norm maps},
  note={With an appendix by Bohmann and Emily Riehl},
  journal={Proc. Amer. Math. Soc.},
  volume={142},
  date={2014},
  number={4},
  pages={1413--1423},
}

\bib{neron}{book}{
  author={Bosch, Siegfried},
  author={L{\"u}tkebohmert, Werner},
  author={Raynaud, Michel},
  date={2017},
  publisher={Springer},
  title={N{\'e}ron Models},
  date={1990},
}

\bib{BrunTambara}{article}{
  author={Brun, M.},
  title={Witt vectors and equivariant ring spectra applied to cobordism},
  journal={Proc. Lond. Math. Soc. (3)},
  volume={94},
  date={2007},
  number={2},
  pages={351--385},
}

\bib{cisinski-deglise}{book}{
  author={Cisinski, D.-C.},
  author={D\'eglise, Fr\'ed\'eric},
  title={Triangulated categories of mixed motives},
  series={Springer Monographs in Mathematics},
  publisher={Springer, Cham},
  year={2019},
}

\bib{CranchThesis}{article}{
  author={Cranch, James},
  title={Algebraic theories and $(\infty ,1)$-categories},
  date={2010},
  eprint={arXiv:1011.3243},
}

\bib{CranchSpan}{article}{
  author={Cranch, James},
  title={Algebraic theories, span diagrams and commutative monoids in homotopy theory},
  date={2011},
  eprint={arXiv:1109.1598},
}

\bib{DressMackey}{book}{
  author={Dress, Andreas W. M.},
  title={Notes on the theory of representations of finite groups. Part I: The Burnside ring of a finite group and some AGN-applications},
  note={With the aid of lecture notes, taken by Manfred K\"{u}chler},
  publisher={Universit\"{a}t Bielefeld, Fakult\"{a}t f\"{u}r Mathematik, Bielefeld},
  date={1971},
}

\bib{drew-gallauer}{article}{
  author={Drew, Brad},
  author={Gallauer, Martin},
  title={The universal six-functor formalism},
  year={2020},
  eprint={arXiv:2009.13610},
}

\bib{eilenberg-maclane}{article}{
  author={Eilenberg, Samuel},
  author={Mac Lane, Saunders},
  title={On the groups {$H(\Pi ,n)$}. {II}. {M}ethods of computation},
  journal={Ann. of Math. (2)},
  volume={60},
  date={1954},
  pages={49--139},
}

\bib{EHKSY1}{article}{
  author={Elmanto, Elden},
  author={Hoyois, Marc},
  author={Khan, Adeel A.},
  author={Sosnilo, Vladimir},
  author={Yakerson, Maria},
  title={Motivic infinite loop spaces},
  journal={Camb. J. Math.},
  volume={9},
  date={2021},
  number={2},
  pages={431--549},
  eprint={arXiv:1711.05248},
}

\bib{e-shah}{article}{
  author={Elmanto, Elden},
  author={Shah, Jay},
  title={Scheiderer motives and equivariant higher topos theory},
  journal={Adv. Math.},
  volume={382},
  date={2021},
  pages={Paper No. 107651, 116},
  eprint={arXiv:1912.11557},
}

\bib{ElmendorfGSpace}{article}{
  author={Elmendorf, A. D.},
  title={Systems of fixed point sets},
  journal={Trans. Amer. Math. Soc.},
  volume={277},
  date={1983},
  number={1},
  pages={275--284},
}

\bib{fulton-mac}{article}{
  author={Fulton, William},
  author={MacPherson, Robert},
  title={Characteristic classes of direct image bundles for covering maps},
  journal={Ann. of Math. (2)},
  volume={125},
  date={1987},
  number={1},
  pages={1--92},
}

\bib{GagnaHarpazLanaryScale}{article}{
  author={Gagna, Andrea},
  author={Harpaz, Yonatan},
  author={Lanari, Edoardo},
  title={On the equivalence of all models for $(\infty ,2)$-categories},
  date={2020},
  eprint={arXiv:1911.01905},
}

\bib{GaitsgoryRozenblyum1}{book}{
  author={Gaitsgory, Dennis},
  author={Rozenblyum, Nick},
  title={A study in derived algebraic geometry. Vol. I. Correspondences and duality},
  series={Mathematical Surveys and Monographs},
  volume={221},
  publisher={American Mathematical Society, Providence, RI},
  date={2017},
  note={Available from \url {http://www.math.harvard.edu/~gaitsgde/GL/}.},
}

\bib{GambinoKock}{article}{
  author={Gambino, Nicola},
  author={Kock, Joachim},
  title={Polynomial functors and polynomial monads},
  journal={Math. Proc. Cambridge Philos. Soc.},
  volume={154},
  date={2013},
  number={1},
  pages={153--192},
}

\bib{GGN}{article}{
  author={Gepner, David},
  author={Groth, Moritz},
  author={Nikolaus, Thomas},
  title={Universality of multiplicative infinite loop space machines},
  journal={Algebr. Geom. Topol.},
  volume={15},
  date={2015},
  number={6},
  pages={3107--3153},
}

\bib{enr}{article}{
  author={Gepner, David},
  author={Haugseng, Rune},
  title={Enriched $\infty $-categories via non-symmetric $\infty $-operads},
  journal={Adv. Math.},
  volume={279},
  pages={575--716},
  eprint={arXiv:1312.3178},
  date={2015},
}

\bib{polynomial}{article}{
  author={Gepner, David},
  author={Haugseng, Rune},
  author={Kock, Joachim},
  title={$\infty $-operads as analytic monads},
  date={2017},
  eprint={arXiv:1712.06469},
}

\bib{freepres}{article}{
  author={Gepner, David},
  author={Haugseng, Rune},
  author={Nikolaus, Thomas},
  title={Lax colimits and free fibrations in $\infty $-categories},
  eprint={arXiv:1501.02161},
  journal={Doc. Math.},
  volume={22},
  date={2017},
  pages={1225--1266},
}

\bib{gepner-heller}{article}{
  author={Gepner, David},
  author={Heller, Jeremiah},
  title={The tom Dieck splitting theorem in equivariant motivic homotopy theory},
  year={2019},
  eprint={arXiv:1910.11485},
}

\bib{GlasmanStrat}{article}{
  author={Glasman, Saul},
  title={Stratified categories, geometric fixed points and a generalized Arone-Ching theorem},
  date={2017},
  eprint={arXiv:1507.01976},
}

\bib{goodwillie}{article}{
  author={Goodwillie, Thomas G.},
  title={Calculus. II. Analytic functors},
  journal={$K$-Theory},
  volume={5},
  date={1991/92},
  number={4},
  pages={295--332},
}

\bib{GreenleesMayMU}{article}{
  author={Greenlees, J. P. C.},
  author={May, J. P.},
  title={Localization and completion theorems for $M{\rm U}$-module spectra},
  journal={Ann. of Math. (2)},
  volume={146},
  date={1997},
  number={3},
  pages={509--544},
}

\bib{GuillouMaySpMack}{article}{
  author={Guillou, Bertrand},
  author={May, J. P.},
  title={Models of $G$-spectra as presheaves of spectra},
  date={2017},
  eprint={arXiv:1110.3571},
}

\bib{GMMOinfloops}{article}{
  author={Guillou, Bertrand},
  author={May, J. Peter},
  author={Merling, Mona},
  author={Osorno, Ang\'{e}lica M.},
  title={A symmetric monoidal and equivariant Segal infinite loop space machine},
  journal={J. Pure Appl. Algebra},
  volume={223},
  date={2019},
  number={6},
  pages={2425--2454},
}

\bib{GMMOsymmonGcat}{article}{
  author={Guillou, Bertrand J.},
  author={May, J. Peter},
  author={Merling, Mona},
  author={Osorno, Ang\'{e}lica M.},
  title={Symmetric monoidal $G$-categories and their strictification},
  journal={Q. J. Math.},
  volume={71},
  date={2020},
  number={1},
  pages={207--246},
}

\bib{GMMOKth}{article}{
  author={Guillou, Bertrand J.},
  author={May, J. Peter},
  author={Merling, Mona},
  author={Osorno, Ang\'{e}lica M.},
  title={Multiplicative equivariant $K$-theory and the Barratt-Priddy-Quillen theorem},
  date={2021},
  eprint={arXiv:2102.13246},
}

\bib{HarpazAmbi}{article}{
  author={Harpaz, Yonatan},
  title={Ambidexterity and the universality of finite spans},
  journal={Proc. Lond. Math. Soc. (3)},
  volume={121},
  date={2020},
  number={5},
  pages={1121--1170},
  eprint={arXiv:1703.09764},
}

\bib{HarpazNuitenPrasmaInfty2}{article}{
  author={Harpaz, Yonatan},
  author={Nuiten, Joost},
  author={Prasma, Matan},
  title={Quillen cohomology of $(\infty ,2)$-categories},
  journal={High. Struct.},
  volume={3},
  date={2019},
  number={1},
  pages={17--66},
}

\bib{enrcomp}{article}{
  author={Haugseng, Rune},
  title={Rectifying enriched $\infty $-categories},
  journal={Algebr. Geom. Topol.},
  volume={15},
  number={4},
  pages={1931--1982},
  eprint={arXiv:1312.3178},
  date={2015},
}

\bib{spans}{article}{
  author={Haugseng, Rune},
  title={Iterated spans and classical topological field theories},
  journal={Math. Z.},
  volume={289},
  number={3},
  pages={1427--1488},
  date={2018},
  eprint={arXiv:1409.0837},
}

\bib{adjmnd}{article}{
  author={Haugseng, Rune},
  title={On lax transformations, adjunctions, and monads in $(\infty ,2)$-categories},
  date={2020},
  eprint={arXiv:2002.01037},
}

\bib{HeineEnrMod}{article}{
  author={Heine, Hadrian},
  title={An equivalence between enriched $\infty $-categories and $\infty $-categories with weak action},
  date={2020},
  eprint={arXiv:2009.02428},
}

\bib{Hermida}{article}{
  author={Hermida, Claudio},
  title={Representable multicategories},
  journal={Adv. Math.},
  volume={151},
  date={2000},
  number={2},
  pages={164--225},
}

\bib{HHRKervaire}{article}{
  author={Hill, M. A.},
  author={Hopkins, M. J.},
  author={Ravenel, D. C.},
  title={On the nonexistence of elements of Kervaire invariant one},
  journal={Ann. of Math. (2)},
  volume={184},
  date={2016},
  number={1},
  pages={1--262},
  eprint={arXiv:0908.3724v2},
}

\bib{HinichYoneda}{article}{
  author={Hinich, Vladimir},
  title={Yoneda lemma for enriched $\infty $-categories},
  journal={Adv. Math.},
  volume={367},
  date={2020},
  pages={107129, 119},
  eprint={arXiv:1805.07635},
}

\bib{HopkinsHill}{article}{
  author={Hopkins, Michael J.},
  author={Hill, Michael A.},
  title={Equivariant symmetric monoidal structures},
  date={2016},
  eprint={arXiv:1610.03114},
}

\bib{hoyois-sixops}{article}{
  author={Hoyois, Marc},
  title={The six operations in equivariant motivic homotopy theory},
  journal={Adv. Math.},
  volume={305},
  date={2017},
  pages={197--279},
}

\bib{kzero}{article}{
  author={Joukhovitski, Seva},
  title={{$K$}-theory of the {W}eil transfer functor},
  note={Special issues dedicated to Daniel Quillen on the occasion of his sixtieth birthday, Part I},
  journal={$K$-Theory},
  volume={20},
  date={2000},
  number={1},
  pages={1--21},
}

\bib{adeel-mv}{article}{
  author={Khan, Adeel A.},
  title={The {M}orel-{V}oevodsky localization theorem in spectral algebraic geometry},
  journal={Geom. Topol.},
  volume={23},
  date={2019},
  number={7},
  pages={3647--3685},
}

\bib{knutson}{book}{
  author={Knutson, Donald},
  title={Algebraic spaces},
  series={Lecture Notes in Mathematics, Vol. 203},
  publisher={Springer-Verlag, Berlin-New York},
  date={1971},
}

\bib{Konovalov}{article}{
  author={Konovalov, Nikolai},
  title={Goodwillie tower of the norm functor},
  date={2020},
  eprint={arXiv:2010.09097},
}

\bib{LMMS}{book}{
  author={Lewis, L. G., Jr.},
  author={May, J. P.},
  author={McClure, J. E.},
  author={Steinberger, M.},
  title={Equivariant stable homotopy theory},
  series={Lecture Notes in Mathematics},
  volume={1213},
  note={With contributions by J. E. McClure},
  publisher={Springer-Verlag},
  place={Berlin},
  date={1986},
}

\bib{mv99}{article}{
  author={Morel, Fabien},
  author={Voevodsky, Vladimir},
  date={1999},
  journal={Inst. Hautes \'{E}tudes Sci. Publ. Math.},
  number={90},
  pages={45--143},
  title={${\bf A}^{1}$-homotopy theory of schemes},
}

\bib{HTT}{book}{
  author={Lurie, Jacob},
  title={Higher Topos Theory},
  series={Annals of Mathematics Studies},
  publisher={Princeton University Press},
  address={Princeton, NJ},
  date={2009},
  volume={170},
  note={Available from \url {http://math.harvard.edu/~lurie/}},
}

\bib{HA}{book}{
  author={Lurie, Jacob},
  title={Higher Algebra},
  date={2017},
  note={Available at \url {http://math.harvard.edu/~lurie/}.},
}

\bib{SAG}{book}{
  author={Lurie, Jacob},
  title={Spectral Algebraic Geometry},
  date={2018},
  note={Available at \url {http://math.harvard.edu/~lurie/}.},
}

\bib{MacphersonEnr}{article}{
  author={Macpherson, Andrew W.},
  title={The operad that co-represents enrichment},
  journal={Homology Homotopy Appl.},
  volume={23},
  date={2021},
  number={1},
  pages={387--401},
  eprint={arXiv:1902.08881},
}

\bib{MacphersonCorr}{article}{
  author={Macpherson, Andrew W.},
  title={A bivariant {Y}oneda lemma and $(\infty ,2)$-categories of correspondences},
  date={2020},
  eprint={arXiv:2005.10496},
}

\bib{MayMerlingOsorno}{article}{
  title={Equivariant infinite loop space theory, {I}. The space level story},
  author={May, J. Peter},
  author={Merling, Mona},
  author={Osorno, Ang\'elica M.},
  date={2017},
  eprint={arXiv:1704.03413},
}

\bib{NardinThesis}{article}{
  author={Nardin, Denis},
  title={Parametrized higher category theory and higher algebra: Expos\'e {IV} -- Stability with respect to an orbital $\infty $-category},
  date={2016},
  eprint={arXiv:1608.07704},
}

\bib{Ostermayr}{article}{
  author={Ostermayr, Dominik},
  title={Equivariant $\Gamma $-spaces},
  journal={Homology Homotopy Appl.},
  volume={18},
  date={2016},
  number={1},
  pages={295--324},
}

\bib{PatchkoriaRigid}{article}{
  author={Patchkoria, Irakli},
  title={Rigidity in equivariant stable homotopy theory},
  journal={Algebr. Geom. Topol.},
  volume={16},
  date={2016},
  number={4},
  pages={2159--2227},
}

\bib{RezkThetaN}{article}{
  author={Rezk, Charles},
  title={A Cartesian presentation of weak $n$-categories},
  journal={Geom. Topol.},
  volume={14},
  date={2010},
  number={1},
  pages={521--571},
}

\bib{RiehlVerityAdj}{article}{
  author={Riehl, Emily},
  author={Verity, Dominic},
  title={Homotopy coherent adjunctions and the formal theory of monads},
  journal={Adv. Math.},
  volume={286},
  date={2016},
  pages={802--888},
  eprint={arXiv:1310.8279},
}

\bib{rondigs}{article}{
  author={R\"ondigs, Oliver},
  title={The Grothendieck ring of varieties and algebraic K-theory of spaces},
  date={2016},
  eprint={arXiv:1611.09327},
}

\bib{RubinNorm}{article}{
  author={Rubin, Jonathan},
  title={Normed symmetric monoidal categories},
  date={2017},
  eprint={arXiv:1708.04777},
}

\bib{rydh-hilb}{article}{
  author={Rydh, David},
  title={Representability of {H}ilbert schemes and {H}ilbert stacks of points},
  journal={Comm. Algebra},
  volume={39},
  date={2011},
  number={7},
  pages={2632--2646},
}

\bib{ShahThesis}{article}{
  author={Shah, Jay},
  title={Parametrized higher category theory and higher algebra: {E}xpos\'e {II} --- {I}ndexed homotopy limits and colimits},
  date={2018},
  eprint={arXiv:1809.05892},
}

\bib{Shimakawa}{article}{
  author={Shimakawa, Kazuhisa},
  title={Infinite loop $G$-spaces associated to monoidal $G$-graded categories},
  journal={Publ. Res. Inst. Math. Sci.},
  volume={25},
  date={1989},
  number={2},
  pages={239--262},
}

\bib{StefanichCorr}{article}{
  title={Higher sheaf theory {I}: Correspondences},
  author={Stefanich, Germán},
  date={2020},
  eprint={arXiv:2011.03027},
}

\bib{StreetPoly}{article}{
  author={Street, Ross},
  title={Polynomials as spans},
  date={2019},
  eprint={arXiv:1903.03890},
}

\bib{StricklandTambara}{article}{
  author={Strickland, Neil},
  title={Tambara functors},
  date={2012},
  eprint={arXiv:1205.2516},
}

\bib{Tambara}{article}{
  author={Tambara, D.},
  title={On multiplicative transfer},
  journal={Comm. Algebra},
  volume={21},
  date={1993},
  number={4},
  pages={1393--1420},
}

\bib{treumann}{article}{
  author={Treumann, David},
  title={Representations of finite groups on modules over K-theory},
  date={2015},
  eprint={arXiv:1503.02477},
}

\bib{voevodsky-four}{article}{
  title={Four functors formalism},
  date={1999},
  note={Unpublished, available from \url {http://www.math.ias.edu/vladimir/files/2015_todeligne3_copy.pdf}},
  author={Voevodsky, Vladimir},
}

\bib{WalkerBispan}{article}{
  author={Walker, Charles},
  title={Universal properties of bicategories of polynomials},
  journal={J. Pure Appl. Algebra},
  volume={223},
  date={2019},
  number={9},
  pages={3722--3777},
  eprint={arXiv:1806.10477},
}
\end{biblist}
\end{bibdiv}
\end{document}